\documentclass[10pt,11pt]{amsart}%
\usepackage[dvips]{graphicx}
\usepackage{amsmath}
\usepackage{color}
\usepackage{tikz}
\usepackage{amssymb}
\usepackage{amsfonts}
\usepackage{comment}
\usepackage{amssymb}
\usepackage{hyperref}%
\providecommand{\U}[1]{\protect\rule{.1in}{.1in}}

\providecommand{\U}[1]{\protect\rule{.1in}{.1in}}
\providecommand{\U}[1]{\protect\rule{.1in}{.1in}}
\providecommand{\U}[1]{\protect\rule{.1in}{.1in}}
\providecommand{\U}[1]{\protect\rule{.1in}{.1in}}
\providecommand{\U}[1]{\protect\rule{.1in}{.1in}}
\providecommand{\U}[1]{\protect\rule{.1in}{.1in}}
\providecommand{\U}[1]{\protect\rule{.1in}{.1in}}
\newtheorem{theorem}{Theorem}[section]

\newtheorem{algorithm}[theorem]{Algorithm}

\newtheorem{corollary}[theorem]{Corollary}

\newtheorem{definition}[theorem]{Definition}
\newtheorem{example}[theorem]{Example}

\newtheorem{lemma}[theorem]{Lemma}

\newtheorem{proposition}[theorem]{Proposition}
\newtheorem{remark}[theorem]{Remark}

\begin{document}
\title[Non-symmetric Cauchy kernel and LPP]{Non symmetric Cauchy kernel,  crystals and last passage percolation}
\date{}
\author{Olga Azenhas, Thomas Gobet, Cédric Lecouvey}

\begin{abstract}
We use non-symmetric Cauchy kernel identities to get the law of last passage
percolation models in terms of Demazure characters. The construction is based
on some restrictions of the RSK correspondence that we rephrase in a unified
way which is compatible with crystal basis theory.

\end{abstract}
\maketitle

\tableofcontents

%

\section{Introduction}

The Cauchy kernel identity is a classical corner stone in the theory of
symmetric functions and characters of the linear groups over the complex
field. Given two sets of indeterminates $X=\{x_{1},\ldots,x_{m}\}$ and
$Y=\{y_{1},\ldots,y_{n}\}$ it asserts that
\[
\prod_{i=1}^{m}\prod_{j=1}^{n}\frac{1}{1-x_{i}y_{j}}=\sum_{\lambda
\in\mathcal{P}_{\min(m,n)}}s_{\lambda}(X)s_{\lambda}(Y)
\]
where $\mathcal{P}_{\min(m,n)}$ is the set of partitions with at most
$\min(m,n)$ parts and, for each such partition $\lambda$, $s_{\lambda}(X)$ and
$s_{\lambda}(Y)$ are the Schur polynomials in the indeterminates $X$ and $Y$,
respectively.\ In fact the Schur functions $s_{\lambda}(X)$ and $s_{\lambda
}(Y)$ can be interpreted as the characters of the irreducible
finite-dimensional representations of highest weight $\lambda$ for the linear
Lie algebras $\mathfrak{gl}_{m}(\mathbb{C})$ and $\mathfrak{gl}_{n}%
(\mathbb{C})$. The aforementioned Cauchy identity can then be regarded as the
character of the $\mathfrak{gl}_{m}\times\mathfrak{gl}_{n}$ bi-module
$S(\mathbb{C}^{m}\times\mathbb{C}^{n})$ where $S(\mathbb{C}^{m}\times
\mathbb{C}^{n})$ is the symmetric tensor space associated to $\mathbb{C}%
^{m}\times\mathbb{C}^{n}$. This can be proved in a very elegant way (see
\cite{Fu,Sta}) by using the Robinson-Schensted-Knuth correspondence. Recall this
is a one-to-one map $\psi$ between the set $\mathcal{M}_{m,n}$ of matrices $M$
with $m$ rows, $n$ columns and entries in $\mathbb{Z}_{\geq0}$, and the pairs
$(P,Q)$ of semistandard tableaux both with the same shape $\lambda$ where $P$
and $Q$ have entries in $\{1,\ldots,m\}$ and $\{1,\ldots,n\}$, respectively.
The RSK correspondence has many interesting properties.\ In particular, for
each matrix $M$ in $\mathcal{M}_{m,n},$ the greatest integer which can be
obtained by summing up the entries in all the possible paths starting at
position $(1,n)$ and ending at position $(m,1)$ with steps $\longleftarrow$ or
$\downarrow$ coincides\footnote{We here consider the paths which are
compatible with the version of RSK that will be used in the paper.} with the
longest row in the tableaux $P,Q$ such that $\psi(M)=(P,Q)$. It is then
natural to study percolation models based on the RSK correspondence where
random matrices whose entries follow independent geometric laws are
considered (see \cite{BDS} for a recent exposition). This type of model has been deeply studied by Johansson in
\cite{joh}, who proved that the fluctuations of the previous last passage
percolation, once correctly normalized, are controlled by the Tracy-Widom
distribution (defined from the study of the largest eigenvalues of random
Hermitian matrices). The RSK correspondence admits various generalizations
which can also be used to get interesting last passage percolation models.
These models involve symmetric polynomials or generalizations of symmetric
polynomials, in particular characters of representations of Lie algebras other
than $\mathfrak{gl}_{n}$ (which are also symmetric with respect to the
associated Weyl group). We refer the reader to \cite{BZ} for a recent survey
and numerous new interesting results in this direction. In a connected area, the various Cauchy identities also yield rich random structures as those studied for instance in the recent papers \cite{BNT,MS, NPS}.

In this paper, we shall follow a different approach and consider percolation
models based on the non-symmetric Cauchy kernel as initially studied by
Lascoux in \cite{Las2}. It was also later considered in \cite{FL} just as computations on polynomials.\ This means that the ordinary Cauchy identity will be
replaced by its non-symmetric analogue%
\begin{equation}
\prod_{1\leq j\leq i\leq n}\frac{1}{1-x_{i}y_{j}}=\sum_{\mu\in\mathbb{Z}%
_{\geq0}^{n}}\overline{\kappa}^{\mu}(X)\kappa_{\mu}(Y) \label{NSCK}%
\end{equation}
where $\overline{\kappa}^{\mu}(X)$ and $\kappa_{\mu}(Y)$ are this time
Demazure atoms and Demazure characters (see \S \ \ref{Subsec_DemaModules}
below for complete definitions) in the indeterminates $X$ and $Y$ (with
$m=n$). It is important to emphasize here that these polynomials are not
symmetric in $X$ and $Y$. They only correspond to characters of
representations for subalgebras of the enveloping algebra $U(\mathfrak{gl}%
_{n})$. It was proved in \cite{Las2} that the identity (\ref{NSCK}) can be
obtained by restricting the RSK correspondence $\psi$ to the set of lower
triangular matrices\footnote{In fact, the convention of our paper differs from
that in \cite{Las2} which considers matrices with nonzero entries in positions
$(i,j)$ with $1\leq i+j\leq n+1$ rather than lower-triangular matrices.}%
.\ Since then, different other proofs have been proposed, in particular in
\cite{AE} (using the combinatorics of semi-skyline augmented fillings) and \cite{CK} (using
the combinatorics of crystal bases). We note  that recently in \cite{AS} an explicit tableau crystal on Mason’s semi-skyline augmented fillings \cite{Mas} has been developed using  the combinatorially equivalent objects, semi-standard key tableaux, introduced by the first author \cite{A}. The seminal paper \cite{Las2} also
established generalizations of the formula (\ref{NSCK}) where positions with
nonzero entries are authorized in the matrices outside their lower
part.\ These so-called extended staircase formulas (see \S \ \ref{Subsec_truncated} and \S \ \ref{Subsec_ASC}) were then obtained just by
computations on polynomials and thus not related to the RSK
correspondence.\ This connection was partially done in \cite{AO2} where other
truncated staircases formulas are also proved to be
compatible with the RSK correspondence using the combinatorics of semi-skyline augmented fillings \cite{Mas0,Mas} and Fomin's growth diagrams \cite{Fom,Sta}. This corresponds to the case where nonzero entries are authorized only in
positions $(i,j)$ with $n-p\leq i\leq j\leq q$, for $p$ and $q$
two nonnegative integers such that $n\geq q\geq p\geq1$.

 The goal of our paper is two-fold. First, we establish all the
existing variants of the non-symmetric Cauchy Kernel identities in the setting
of crystal basis theory and make it compatible with the RSK construction based
on bi-crystals. Recall here that crystals are oriented graphs which can be
interpreted as the combinatorial skeletons of irreducible finite-dimensional
representations of $\mathfrak{gl}_{n}$. We refer the reader to
\cite{BumpSchilling2017} and the references therein for a recent exposition.
Crystal bases were introduced by Lusztig (for any finite root system) \cite{Luca} and
Kashiwara (for classical root systems) \cite{kash0} in 1990. The graph structure arises
from the action of the so-called Kashiwara operators, which are certain
renormalizations of the Chevalley operators. It was later proved that crystals
coincide with Littelmann's graphs defined by using his path model \cite{Lit1}. Crystal
theory allows one to get an illuminating interpretation of the
RSK-correspondence and thus, in particular, of the Cauchy identity.\ A similar
interpretation was discovered by Choi and Kwon in \cite{CK} for the
non-symmetric case (\ref{NSCK}). Here we complete the picture with the
truncated and augmented staircase formulas.\ Our second objective is to use
the previous compatibility of the aforementioned map $\psi$ with the generalized Cauchy
identities to give the law of some last passage percolation models where
constraints are imposed on the locations of nonzero positions in the random
matrices considered.\ \ These laws will be expressed in terms of Demazure
characters and Demazure atoms and thus will have less symmetries than the
existing ones which rather use symmetric polynomials. There is nevertheless an
interesting intersection in the case $x_{i}=y_{i}$ for any $i=1,\ldots,n$.
Then, the identity (\ref{NSCK}) becomes symmetric and can be expanded in terms
of Schur functions by using an identity due to Littlewood (see \cite{CK}). This case yields a
last passage percolation model already studied (see \cite{BZ}). We emphasize that $x_i\neq y_i$ in our case, which explains why we need to consider Demazure characters, which are non-symmetric in general.

\bigskip

The paper is organized as follows. In Section 2, we recall the background on
representation theory of $\mathfrak{gl}_{n}$, the corresponding character
theory (its usual and Demazure versions) and its links with the Coxeter monoid and crystal basis
theory. Some key results for the purposes of this article are established here for which we did not find references in the literature. We also relate the RSK correspondence with bi-crystal structures and
interpret the Cauchy and non-symmetric Cauchy identities in this context. The
non-symmetric Cauchy identity is in particular obtained as the restriction of
the usual RSK to lower triangular matrices. The goal of Section 3 is to prove
that one can also get the truncated staircase Cauchy identity by restriction
of RSK to a relevant subset of matrices. To this end, we consider parabolic
restrictions of Demazure crystals and show that they admit a simple
combinatorial structure. {In particular, \S\ \ref{Subsec_ASC} is devoted to the extended staircase Cauchy
identity which is yet obtained by restriction of RSK.\ The idea here is to use
suitable adaptations of Demazure operators (defined on polynomials) acting on
crystals.\ It is also explained in \S\ \ref{SubsecSEmutilde} how the extended staircase result allows one
to rederive the truncated staircase identity by making more explicit its
formulation and connecting it to the approach proposed in \cite{AE,AO2}.\ Finally
in Section \ref{Secpercolation}, we use the previous combinatorial constructions to get the law
of various percolation models in terms of Demazure characters. In the Appendix \ref{Appe}, for the reader convenience, Coxeter monoids and Coxeter-theoretic techniques are given.

\bigskip
\noindent\textbf{MSC classification:} 05E05, 05E10, 60K35.

\noindent\textbf{Keywords:} Cauchy identity, Demazure characters, crystals, Coxeter monoids, percolation models.

\medskip

\noindent\textbf{Acknowledgments:}  O. A. is partially supported by the Center for Mathematics of the University of Coimbra - UIDB/00324/2020, funded by the Portuguese Government through FCT/MCTES. C. L. is partially supported by the Agence Nationale de la Recherche funding ANR CORTIPOM 21-CE40-001.

\section{Background on representations and characters of $\mathfrak{gl}_{n}$}

In this section, we review some classical results about representation theory
of the linear Lie algebra $\mathfrak{gl}_{n}=\mathfrak{gl}_{n}(\mathbb{C})$
over the field of complex numbers \cite{FH}. Firstly, recall the triangular decomposition
$\mathfrak{gl}_{n}=\mathfrak{gl}_{n}^{+}\oplus\mathfrak{h}\oplus
\mathfrak{gl}_{n}^{-}$ of $\mathfrak{gl}_{n}$ into its upper, diagonal and
lower parts.

\subsection{Representations and characters}

\label{Subsec_representation}Let $\mathcal{P}_{n}$ be the set of partitions
$\lambda=(\lambda_{1}\geq\cdots\geq\lambda_{n}\geq0)$ with at most $n$ parts.
A partition will be identified with its Young diagram written in French
convention (see Example \ref{Example_tableau}).\ The finite-dimensional
irreducible polynomial representations of $\mathfrak{gl}_{n}$ are parametrized
by the partitions in $\mathcal{P}_{n}$. To any $\lambda\in\mathcal{P}_{n}$, we denote by $V(\lambda)$ the corresponding finite-dimensional representation (or $\mathfrak{gl}_{n}$-module). By considering only the action of
the (commutative) Cartan subalgebra $\mathfrak{h}$ on $V(\lambda)$, one gets
the weight space decomposition%
\[
V(\lambda)=%
{\textstyle\bigoplus\limits_{\mu\in P}}
V(\lambda)_{\mu}%
\]
where the weight space $P=\mathbb{Z}_{\geq0}^{n}=\oplus_{i=1}^{n}%
\mathbb{Z}_{\geq0}{\mathbf{e}_{i}}$ is regarded as a subset of $\mathfrak{h}%
^{\ast}$ and for any $\mu\in P$%
\[
V(\lambda)_{\mu}=\{v\in V(\lambda)\mid h(v)=\mu(h)v\text{ for any }%
h\in\mathfrak{h}\}\text{.}%
\]
The symmetric group $\mathfrak{S}_{n}$ (which is the Weyl group of
$\mathfrak{gl}_{n}$) acts on $P$ by permuting the coordinates of the weights
and one then has $\dim V(\lambda)_{\mu}=\dim V(\lambda)_{\sigma(\mu)}$ for any
$\sigma\in\mathfrak{S}_{n}$ and any $\mu\in P$.\ The weight space
decomposition leads to the notion of \textsf{character} of $V(\lambda)$ which is the
polynomial in the indeterminates $x_{1},\ldots,x_{n}$ defined by
\[
s_{\lambda}=\sum_{\mu\in P}\dim V(\lambda)_{\mu}x^{\mu}%
\]
where for any $\mu=(\mu_{1},\ldots,\mu_{n})\in\mathbb{Z}_{\geq0}^{n}$ we use
the notation $x^{\mu}=x_{1}^{\mu_{1}}\cdots x_{n}^{\mu_{n}}$.\ By the previous
considerations, the polynomial $s_{\lambda}$ belongs in fact to the ring
$\mathrm{Sym}_{\mathbb{Z}}[x_{1},\ldots,x_{n}]$ of symmetric polynomials in
the indeterminates $x_{1},\ldots,x_{n}$ with coefficients in $\mathbb{Z}$.
This is the celebrated \textsf{Schur polynomial} which can also be obtained as the
quotient of two skew-symmetric polynomials using the formula%
\[
s_{\lambda}=\frac{\sum_{\sigma\in\mathfrak{S}_{n}}\varepsilon(\sigma
)x^{\sigma(\lambda+\rho)}}{\sum_{\sigma\in\mathfrak{S}_{n}}\varepsilon
(\sigma)x^{\sigma(\rho)}},
\]
{where $\rho=(n-1,n-2,\dots,1,0)$.}

\begin{remark}
Instead of considering the representation theory of $\mathfrak{gl}_{n}$, we
can proceed similarly with the representation theory of its enveloping algebra
$U(\mathfrak{gl}_{n})$.\ Simple finite-dimensional $U(\mathfrak{gl}_{n}%
)$-modules are still parametrized by the elements of $\mathcal{P}_{n}$ and we
will use the same notation for both representation theories.
\end{remark}

\subsection{Bruhat order and Coxeter monoid}

Recall that $\mathfrak{S}_{n}$ is generated by $S=\{s_{1},\ldots
,s_{n-1}\}$ where for any $i=1,\ldots,n-1$, $s_{i}$ is the simple transposition (or simple reflection)
flipping $i$ and $i+1$; this yields a realization of $\mathfrak{S}_n$ as a Coxeter group. We denote by $\ell(\sigma)$ the \textsf{length} of a permutation $\sigma\in \mathfrak{S}_n$, defined as the smallest integer $k\geq 0$ such that $\sigma=s_{i_{1}}\cdots s_{i_{k}}$, where the $s_{i_j}$'s are simple  reflections. A word of the form $s_{i_1} s_{i_2} \cdots s_{i_k}$ representing $\sigma\in\mathfrak{S}_n$ and such that all the $s_{i_j}$'s are simple reflections and $\ell(\sigma)=k$ is called a \textsf{reduced decomposition} of $\sigma$. We refer the reader to~\cite{BB} for basic statements on the symmetric group viewed as a Coxeter group.

The \textsf{(strong) Bruhat order} $\leq$ on $\mathfrak{S}_{n}$
can be defined by $\sigma^{\prime}\leq\sigma$ in $\mathfrak{S}_{n}$ if and
only if there is a reduced decomposition of $\sigma$ admitting a subword (not
necessarily made of consecutive letters) which is a reduced decomposition of
$\sigma^{\prime}$, if and only if every reduced decomposition of $\sigma$
admits a subword which is a reduced decomposition of $\sigma^{\prime}$
(see~\cite[Corollary 2.2.3]{BB}). The longest element of $\mathfrak{S}_{n}$ is
denoted by $\sigma_{0}$. Given any partition $\lambda$ in $\mathcal{P}_{n}$,
we denote by $\mathfrak{S}_{\lambda}$ its stabilizer under the action of
$\mathfrak{S}_{n}$.\ Each coset in $\mathfrak{S}_{n}/\mathfrak{S}_{\lambda}$
contains a unique element of minimal length and the set of elements of minimal
length is denoted by $\mathfrak{S}_{n}^{\lambda}$.\ Then each $\sigma
\in\mathfrak{S}_{n}$ admits a unique decomposition of the form $\sigma=uv$
with $v\in\mathfrak{S}_{\lambda},$ $u\in\mathfrak{S}_{n}^{\lambda}$ and
$\ell(\sigma)=\ell(u)+\ell(v)$. One then has a one-to-one correspondence
between the elements of $\mathfrak{S}_n^{\lambda}$ and the $\mathfrak{S}_{n}%
$-orbit of $\lambda$ which we denote by $\mathfrak{S}_{n}{\lambda}$.


The elementary bubble sort operator $\pi_{i}$, $1\leq i<n$, on the
weak composition $\alpha=(\alpha_1,\alpha_2,$ $\dots,\alpha_n)$ $\in \mathbb{Z}^n_{\ge 0}$, sorts the entries in positions $i$ and $i+1$ by
swapping $\alpha_{i}$ and $\alpha_{i+1}$ if $\alpha_{i}>\alpha_{i+1}$, and
fixing $\alpha$ otherwise, namely,
\begin{equation}
	\pi_{i}(\alpha)=\left\{
	\begin{array}
		[c]{c}%
		s_{i}\alpha\text{ if }~\alpha_{i}>\alpha_{i+1}\\
		\alpha~\text{if}~\alpha_{i}\leq\alpha_{i+1}.
	\end{array}
	\right.  \label{Bubble}%
\end{equation}
Thus elementary bubble sort operators $\pi_{i}$, $1\leq i<n$, satisfy the
relations%
\begin{equation}
\pi_{i}^{2}=\pi_{i}\;(i=1,\dots,n),\;\pi_{i}\pi_{i+1}\pi_{i}=\pi_{i+1}\pi
_{i}\pi_{i+1}\;(i=1,\dots,n-1),\;\pi_{i}\pi_{j}=\pi_{j}\pi_{i},\;(|i-j|>1).
\label{bubblesort}%
\end{equation}

It follows from Matsumoto's Lemma \cite{Mat, BumpSchilling2017} that for every
$w\in\mathfrak{S}_{n}$, we may write $\pi_{w}$ to mean $\pi_{i_{1}}\pi_{i_{2}%
}\cdots\pi_{i_{k}}$, whenever $s_{i_{1}}s_{i_{2}}\cdots s_{i_{k}}$ is a
reduced word of $w$ in $\mathfrak{S}_{n}$. Later we will see that the above set of relations define the so-called \textsf{Coxeter monoid} $\mathfrak{M}_{n}$ \cite{RS} (see Section~\ref{SubsecPR} )}.

\begin{lemma}\label{lem_cox_mon_1}
Let $\lambda\in \mathcal{P}_n$, $w\in \mathfrak{S}_n$, and let $\mu=w\lambda$.
\begin{enumerate}
	\item Let $t=(i ~j)$ be a transposition in $\mathfrak{S}_n$ with $i<j$. If $\mu_i < \mu_j$, then $\ell(tw)<\ell(w)$.
	\item If $s_{i_1} s_{i_2} \cdots s_{i_k}$ is any reduced decomposition of $w$, then $w\lambda= \pi_{i_1} \pi_{i_2} \cdots \pi_{i_k}(\lambda)=\pi_w(\lambda)$.
\end{enumerate}
\end{lemma}

\begin{proof}
Recall that, given an element $w\in \mathfrak{S}_n$, the set $N(w)=\{t\in T~|~\ell(tw)<\ell(w)\}$ is the set of (left) inversions of $w$; it satisfies $|N(w)|=\ell(w)$ and for all $u,v\in \mathfrak{S}_n$, we have the equality (see for instance~\cite[Chapter 1, Exercise 12]{BB}) \begin{equation}\label{n_function} N(uv)=N(u)\Delta u N(v) u^{-1},\end{equation}
where $\Delta$ denotes the symmetric difference (note that, in particular, the product $uv$ does not need to be reduced).

The proof of the first point is by induction on $\ell(w)$. If $\ell(w)=0$, then $w=1$ and the set of transpositions $t$ such that $\ell(tw)<\ell(w)$ is empty. We have $\mu=\lambda$ and $\lambda_i \geq \lambda_j$ for all $j>i$ in this case. Hence assume that $\ell(w)>0$. Let $s=s_k$ be a simple transposition such that $w=s_k u$ and $\ell(w)=\ell(u)+1$. If $s_k=t$, then $i=k$, $j=k+1$, and $\ell(tw)<\ell(w)$, hence we are done. We can thus assume that $t\neq s$. Using~\eqref{n_function} above we have $N(w)=N(su)=N(s)\Delta s N(u) s= \{s\} \Delta s N(u)s $ and using the fact that $s\neq t$, we deduce the equivalence $$\ell(tw)<\ell(w) \Leftrightarrow \ell(stsu)<\ell(u),$$ denoting $sts=(i' ~j')$, by induction it suffices to show that $\nu_{i'}<\nu_{j'}$, where $\nu=u \lambda$. We can assume that $sts\neq t$, otherwise the supports of $s$ and $t$ are disjoint, hence $i'=i$, $j'=j$, and $\nu_i=\mu_i$, $\nu_j=\mu_j$. We can thus assume that $s\in\{s_{i-1}, s_i, s_{j-1}, s_{j}\}$. We treat the case where $s=s_j$, the other cases are similar. We have $sts=(i ~j+1)$, and we have $\nu_j=\mu_{j+1}$, $\nu_{j+1}=\mu_j$. On the other hand, since $i$ is not in the support of $(j~j+1)$, we have $\nu_i=\mu_i$. Hence $\nu_i=\mu_i < \mu_j=\nu_{j+1}$, which by induction yields $\ell(stsu)<\ell(u)$.

Let us prove the second point. We argue by induction on $k$. If $k=0$ then there is nothing to prove. Assume that $k\geq 1$. By induction we have that $s_{i_2} \cdots s_{i_k}\lambda=\pi_{i_2} \cdots \pi_{i_k}(\lambda)$. Now writing $\mu=s_{i_2} \cdots s_{i_k}\lambda$ and $i=i_1$, by the first point we have that $\mu_i \geq \mu_{i+1}$, otherwise we would have $\ell(s_{i_1} s_{i_2}\cdots s_{i_k})= k-1$, contradicting the fact that $s_{i_1} s_{i_2}\cdots s_{i_k}$ is reduced. It follows that $s_{i_1}\mu = \pi_{i_1}(\mu)$, hence that $w\lambda= \pi_{i_1} \pi_{i_2} \cdots \pi_{i_k}(\lambda)$, as required.
\end{proof}

\begin{lemma} Consider the set $\mathfrak{S}_n \lambda$, which is in bijection with $\mathfrak{S}_n^\lambda$ through $w \lambda \mapsto w^\lambda$, where $w^\lambda$ is the representative of minimal length of $w \mathfrak{S}_\lambda$. Then the transitive closure of the relations \[
	\mu<t\mu
	,\;\mbox{if $\mu_i>\mu_{j}$, $i<j$, $t$ is the transposition $(i\,j)\in \mathfrak{S}_n $  and $\mu=(\mu_1,\ldots,\mu_n)\in
		\mathfrak{S}_n\lambda$}
	\]
	yields a partial order on $\mathfrak{S}_n \lambda$, which coincides through the aforementioned bijection with the restriction of the strong Bruhat order on $\mathfrak{S}_n$ to $\mathfrak{S}_n^\lambda$.
\end{lemma}

\begin{proof}
Assume that $\mu < t \mu$, and let $w\in\mathfrak{S}_n$ such that $\mu=w\lambda$. Denoting $\mu'=tw\lambda=t \mu$, we have $\mu_i' < \mu_j'$. By point 1 of Lemma~\ref{lem_cox_mon_1}, we have $\ell(w)=\ell(ttw)< \ell(tw)$, which shows that $w < tw$ in the strong Bruhat order. It follows that $w^\lambda < (tw)^\lambda$.

Conversely, let $u,v\in \mathfrak{S}_n^\lambda$ such that $u \leq v$. By definition of the strong Bruhat order, there is a sequence $t_1, t_2, \dots, t_k$ of transpositions such that $u < t_1 u < t_2 t_1 u < \cdots < t_k t_{k-1} \cdots t_2 t_1 u$. Note that the elements in this sequence are in $\mathfrak{S}_n$ but, apart from $u$ and $v$, not necessarily in $\mathfrak{S}_n^\lambda$. To conclude the proof it therefore suffices to show that, if $u < tu$ with $u\in\mathfrak{S}_n$, $t\in T$, then $u \lambda \leq tu \lambda$. Letting $\mu= u \lambda$, if $\mu_i < \mu_j$, then by the first point of Lemma~\ref{lem_cox_mon_1} we have $\ell(tu) < \ell(u)$, contradicting $u < tu$. Hence $\mu_i \geq \mu_j$. If $\mu_i > \mu_j$ then we have $\mu < t \mu$. If $\mu_i=\mu_j$, we have $u \lambda= t u\lambda$. This concludes the proof.
\end{proof}

\begin{lemma}\label{lem_equiv_mu_weyl}
Let $\lambda\in \mathcal{P}_n$ and let $\sigma\in\mathfrak{S}_n^\lambda$. Let $\mu=\sigma \lambda$ and $s_i$ a simple reflection of $\mathfrak{S}_n$. We have the equivalences
\begin{gather*}
	\mu_{i}>\mu_{i+1}\text{ iff }\ell(s_{i}\sigma)=\ell(\sigma)+1\text{ and }%
	s_{i}\sigma\in\mathfrak{S}_{n}^{\lambda},\\
	\mu_{i}=\mu_{i+1}\text{ iff }s_{i}\sigma\notin\mathfrak{S}_{n}^{\lambda
	}~(\text{in which case we must have }\ell(s_{i}\sigma)=\ell(\sigma)+1),\\
	\mu_{i}<\mu_{i+1}\text{ iff }\ell(s_{i}\sigma)=\ell(\sigma)-1~(\text{in which
		case we must have }s_{i}\sigma\in\mathfrak{S}_{n}^{\lambda}).
\end{gather*}
\end{lemma}

\begin{proof}
Assume that $w=s_i \sigma \notin \mathfrak{S}_n^\lambda$. Then $\ell(w)= \ell(\sigma)+1$. Since $w\notin \mathfrak{S}_n^\lambda$, there is a simple reflection $s_j\in \mathfrak{S}_\lambda$ such that $\ell(w s_j) < \ell(w)$. Take any reduced decomposition $s_{n_1} \cdots s_{n_\ell}$ of $\sigma$. We have that $s_i s_{n_1} \cdots s_{n_\ell}$ is a reduced decomposition of $w$ and since $\ell(w s_j) < \ell(w)$, by the exchange lemma there is a reduced decomposition of $w s_j$ obtained from $s_i s_{n_1} \cdots s_{n_\ell}$ obtained by just removing a letter. If this letter is not $s_i$, we get that $\ell(\sigma s_j) =\ell(\sigma)-1$, in contradiction with $\sigma\in\mathfrak{S}_n^\lambda$, since $s_j \in \mathfrak{S}_\lambda$. We thus have that $w s_j = s_{n_1} \cdots s_{n_\ell}  =\sigma = s_i w$. It follows that $\sigma^{-1} s_i \sigma = s_j \in \mathfrak{S}_\lambda$. This yields $\sigma^{-1} s_i \sigma \lambda = \lambda$, hence $s_i \mu = \mu$, hence $\mu_i=\mu_{i+1}$. Conversely, assume that $\mu_i=\mu_{i+1}$. We thus have $s_i \sigma \lambda = \sigma \lambda= \mu$. Since $\sigma \in \mathfrak{S}_n^\lambda$, by uniqueness of the element of the element $w\in \mathfrak{S}_n^\lambda$ such that $\mu= w \lambda$, we cannot have $s_i \sigma \in\mathfrak{S}_n^\lambda$. Hence the two statements in the middle line are equivalent.

Assume that $\mu_i > \mu_{i+1}$. Then, since the middle equivalence is already shown, we know that $s_i \sigma \in \mathfrak{S}_n^\lambda$. By Lemma~\ref{lem_cox_mon_1}~(1), we must have $\ell(s_i s_i \sigma) < \ell( s_i \sigma)$, forcing $\ell(s_i \sigma)=\ell(\sigma)+1$. Also by Lemma~\ref{lem_cox_mon_1}~(1), if $\mu_i < \mu_{i+1}$, then $\ell(s_i \sigma) < \ell(\sigma)$, yielding $\ell(s_i \sigma)=\ell(\sigma)-1$.

We thus have shown that, in each line, the left condition implies the right one (we have even shown that we have equivalence in the middle line). Since the three conditions on the right are disjoint, we must have equivalence in each line.
\end{proof}

\begin{lemma}\label{lem_cox_mon_algo}
Let $\sigma\in\mathfrak{S}_n$ and $\alpha=\sigma \lambda$. We can obtain the minimal representative $\hat{\sigma}\in\mathfrak{S}%
_{n}^{\lambda}$ of $\sigma$ from any $\pi_\sigma=\pi_{j_{1}}\pi_{j_{2}}\cdots\pi_{j_{l}%
}\in\mathfrak{M}_{n}$ such that $\pi_{j_{1}}\pi_{j_{2}}\cdots\pi_{j_{l}%
}\lambda=\alpha$ with $s_{j_{1}}s_{j_{2}}\cdots s_{j_{l}}$ a (not necessarily
reduced) word of an element of $\mathfrak{S}_{n}$ as follows: for $r=l,\dots,1$, delete $\pi_{j_{r}}$ in $\pi_{j_{1}}\pi_{j_{2}%
}\cdots\pi_{j_{l}}$ whenever $\mu_{j_{r}}\leq\mu_{j_{r}+1}$ in $(\mu_{1}%
,\dots,\mu_{n})=\pi_{j_{r+1}}\cdots\pi_{j_{l}}(\lambda)$. The resulting decomposition obtained in this way is a reduced
decomposition $\pi_{\hat{\sigma}}$ in $\mathfrak{M}_{n}$ and gives
$\hat{\sigma}\in\mathfrak{S}_{n}^{\lambda}$.

\end{lemma}

\begin{proof}
The fact that the resulting decomposition $\pi_{i_1} \pi_{i_2}\cdots \pi_{i_k}$ satisfies $\pi_{i_1} \pi_{i_2}\cdots \pi_{i_k}(\lambda)=\pi_\sigma(\lambda)$ is clear since a letter $\pi_{j_r}$ is removed whenever its action on $\pi_{j_{r+1}}\cdots\pi_{j_{l}}(\lambda)$ is trivial. We show by decreasing induction on $k$ that $s_{i_j} s_{i_{j+1}} \cdots s_{i_k}(\lambda)= \pi_{i_j} \pi_{i_{j+1}} \cdots \pi_{i_k}  (\lambda)$ for all $j$, and that $s_{i_j} \cdots s_{i_k}$ is reduced and lies in $\mathfrak{S}_n^{\lambda}$. If $k=0$ then the result is trivially true. Hence let $j \leq k$ and assume that $\mu:=s_{i_{j+1}} \cdots s_{i_k}(\lambda)= \pi_{i_{j+1}} \cdots \pi_{i_k}  (\lambda)$, and that $s_{i_{j+1}} \cdots s_{i_k}$ is reduced and lies in $\mathfrak{S}_n^{\lambda}$. We must have $\mu_{i_j}> \mu_{i_{j}+1}$, otherwise the letter $\pi_{i_j}$ would have been removed. Hence by definition of the action of the bubble sort operator $\pi_{i_j}$, we have $\pi_{i_j}(\mu)=s_{i_j}(\mu)$, which yields $s_{i_j} s_{i_{j+1}} \cdots s_{i_k}(\lambda)= \pi_{i_j} \pi_{i_{j+1}} \cdots \pi_{i_k}  (\lambda)$. Setting $w=s_{i_{j+1}}\cdots s_k$ we obtain using Lemma~\ref{lem_equiv_mu_weyl}~(1) that $\ell(s_{i_j}w)= \ell(w)+1$ and $s_{i_j} w\in\mathfrak{S}_n^\lambda$, hence $s_{i_j} s_{i_{j+1}}\cdots s_{i_k}$ is still reduced, and defines an element of $\mathfrak{S}_n^\lambda$.

It only remains to show that $\tau:=s_{i_1} s_{i_2} \cdots s_{i_k}$ is equal to $\hat{\sigma}$. But we have $\pi_\tau (\lambda)= \pi_{\sigma}(\lambda)=\pi_{\hat{\sigma}}(\lambda)$, which by Lemma~\ref{lem_cox_mon_1}~(2) yields $\tau(\lambda)=\hat{\sigma}(\lambda)$. Since both $\tau, \hat{\sigma}$ lie in $\mathfrak{S}_n^\lambda$, this forces $\tau=\hat{\sigma}$, which concludes the proof.
\end{proof}

\begin{example} Let $n=4$. We have
$${\pi_{2}\pi_{2}}{\pi}_{1}{\pi}_{2}{(3,2,2,1)}={\pi_{2}}{\pi}_{2}{\pi}%
_{1}{(3,2,2,1)}=\pi_{2}\pi_2(2,3,2,1)=\pi_2(2,2,3,1)=(2,2,3,1).$$
Applying the algorithm described in Lemma~\ref{lem_cox_mon_algo} to the word $\pi_2 \pi_2 \pi_1 \pi_2$ and the weight $\lambda=(3,2,2,1)$ yields $ \widehat{\pi_2} \pi_2 \pi_1 \widehat{\pi_2}= \pi_2 \pi_1 = \pi_{s_2 s_1}$, where the hat over the bubble sort operator denotes omission. We indeed have $\pi_2 \pi_2 \pi_1 \pi_2= \pi_\sigma$ with $\sigma= s_2 s_1 s_2$ and $\hat{\sigma}=s_2 s_1$ (here $\mathfrak{S}_\lambda=\{ 1, s_2\}$).
\end{example}


\subsection{Crystals}

\subsubsection{Abstract crystals}

To each partition $\lambda\in\mathcal{P}_{n}$ corresponds a crystal graph
$B(\lambda)$ which can be regarded as the combinatorial skeleton of the simple
module $V(\lambda)$. In particular, its vertices label a distinguished basis
of $V(\lambda)$. Its general structure can be defined using the canonical bases introduced by Lusztig \cite{Luca} and subsequently studied
by Kashiwara under the name of global bases (see \cite{kash} and
\cite{kash2}). It also admits various combinatorial realizations (\textit{i.e.}, vertex
labelings) in terms semistandard tableaux, Littelmann's paths (see
\cite{Lit1}) or semi-skyline (see \cite{Mas}, \cite{AE}).\ We will recall the
tableau realization below.\ The (abstract) crystal $B(\lambda)$ is a graph
whose set of vertices is endowed with a weight function $\mathrm{wt}%
:B(\lambda)\rightarrow P$ and with the structure of a colored and oriented
graph given by the action of the crystal operators $\tilde{f}_{i}$ and
$\tilde{e}_{i}$ with $i\in I=\{1,\dots,n-1\}$. More precisely, we have an
oriented arrow $b\overset{i}{\rightarrow}b^{\prime}$ between two vertices $b$
and $b^{\prime}$ in $B(\lambda)$ if and only if $b^{\prime}=\tilde{f}_{i}(b)$
or equivalently $b=\tilde{e}_{i}(b^{\prime})$. We have $\tilde{f}_{i}(b)=0$
(resp. $\tilde{e}_{i}(b)=0$) when no arrow $i$ starts from $b$ (resp. ends at
$b$). Here the symbol $0$ should be understood as a sink vertex not lying in
$B(\lambda)$.\ For any $i\in I$, the crystal $B(\lambda)$ can be decomposed
into its \textsf{$i$-chains} which are obtained just by keeping the $i$-arrows.$\ $For
such a chain $C$, we denote by $s(C)$ and $e(C)$ its source and target
vertices, respectively. There is a unique vertex $b_{\lambda}$ in $B(\lambda)$
such that $\tilde{e}_{i}(b_{\lambda})=0$ for any $i\in I$ (that is,
$b_{\lambda}$ is the source vertex of each $i$-chain containing $b_{\lambda}$)
called the \textsf{highest weight vertex} of $B(\lambda)$ and we have $\mathrm{wt}%
(b_{\lambda})=\lambda$.\ For any $b\in B(\lambda)$, there is a path
$b=\tilde{f}_{i_{1}}\cdots\tilde{f}_{i_{r}}(b_{\lambda})$ from $b_{\lambda}$
to $b$.\ Let us denote by $S=\{\alpha_{1},\ldots,\alpha_{n-1}\}$ the set of
simple roots of $\mathfrak{gl}_{n}$ where $\alpha_{i}=\mathbf{e}%
_{i}-\mathbf{e}_{i+1}$ for $1\leq i<n$. The weight function $\mathrm{wt}$ satisfies
\[
\mathrm{wt}(b)=\lambda-\sum_{k=1}^{r}\alpha_{i_{k}}.
\]
For any $i\in I$, the crystal $B(\lambda)$ decomposes into $i$-chains.\ Thus,
for any vertex $b\in B(\lambda)$, we can define $\varphi_{i}(b)=\max
\{k\mid\tilde{f}_{i}^{k}(b)\neq0\}$ and $\varepsilon_{i}(b)=\max\{k\mid
\tilde{e}_{i}^{k}(b)\neq0\}$.\ We then have
\[
s_{\lambda}=\sum_{b\in B(\lambda)}x^{\mathrm{wt}(b)}.
\]

The Weyl group $W$ also acts on the vertices of $B(\lambda)$: the action of
the simple reflection $s_{i}$ on $B(\lambda)$ sends each vertex $b$ on the
unique vertex $b^{\prime}$ in the $i$-chain of $b$ such that $\varphi
_{i}(b^{\prime})=\varepsilon_{i}(b)$ and $\varepsilon_{i}(b^{\prime}%
)=\varphi_{i}(b)$. This simply means that $b$ and $b^{\prime}$ correspond by
the reflection with respect to the center of the $i$-chain containing $b$. We
shall write
\[
O(\lambda)=\{\sigma\cdot b_{\lambda}=b_{\sigma\lambda}\mid\sigma
\in\mathfrak{S}_{n}^{\lambda}\}
\]
for the orbit of the highest weight vertex $b_{\lambda}$ of $B(\lambda
)$.\ Observe that $b_{\sigma\lambda}$ is then the unique vertex in
$B(\lambda)$ of weight $\sigma\lambda$. {The elements of $O(\lambda)$, called
the \textsf{keys} of $B(\lambda)$, }are those vertices of $B(\lambda)$ which are
completely characterized by their weight. Thereby, one has a direct
correspondence between the keys and the elements of $\mathfrak{S}_{n}%
^{\lambda}$. For convenience, we often abuse notation and identify the key
$b_{\sigma\lambda}$ with $\sigma\in\mathfrak{S}_{n}^{\lambda}$.

In fact, one can associate a crystal to any finite-dimensional $\mathfrak{gl}%
_{n}$-module by considering its decomposition into irreducible components.
This $\mathfrak{gl}_{n}$-crystal is a disjoint union of connected components,
each being isomorphic to a highest weight crystal $B(\lambda), \lambda
\in\mathcal{P}_{n}$.$\ $Given two partitions $\lambda$ and $\mu$ in
$\mathcal{P}_{n}$, the crystal associated to the representation $V(\lambda
)\otimes V(\mu)$ is the crystal $B(\lambda)\otimes B(\mu)$ whose set of
vertices is the direct product of the sets of vertices of $B(\lambda)$ and
$B(\mu)$ and whose crystal structure is given {by }$\mathrm{wt}${$(a\otimes
b)=$}$\mathrm{wt}${$(a)+$}$\mathrm{wt}${$(b)$} and by the following rules%
\begin{equation}
\tilde{e}_{i}(u\otimes v)=\left\{
\begin{array}
[c]{l}%
u\otimes\tilde{e}_{i}(v)\text{ if }\varepsilon_{i}(v)>\varphi_{i}(u)\\
\tilde{e}_{i}(u)\otimes v\text{ if }\varepsilon_{i}(v)\leq\varphi_{i}(u)
\end{array}
\right.  \text{ and }\tilde{f}_{i}(u\otimes v)=\left\{
\begin{array}
[c]{l}%
\tilde{f}_{i}(u)\otimes v\text{ if }\varphi_{i}(u)>\varepsilon_{i}(v)\\
u\otimes\tilde{f}_{i}(v)\text{ if }\varphi_{i}(u)\leq\varepsilon_{i}(v)
\end{array}
\right.  . \label{tens_crys}%
\end{equation}
We adopt the convention that $u\otimes0=0\otimes v=0$. A key result in crystal
theory shows that for any partition $\nu\in\mathcal{P}_{n}$, the tensor
multiplicity $c_{\lambda,\mu}^{\nu}$ of $V(\nu)$ in $V(\lambda)\otimes V(\mu)$
(which is a Littlewood-Richardson coefficient) is equal to the number of
connected components in $B(\lambda)\otimes B(\mu)$ with highest weight vertex
of weight $\nu$.

\subsubsection{Keys and dilatation of crystals}

Consider $k$ a positive integer and $\lambda$ a partition. There exists a
unique embedding of crystals $\psi_{k}:B(\lambda)\hookrightarrow B(k\lambda)$
such that for any vertex $b\in B(\lambda)$ and any path $b=\tilde{f}_{i_{1}%
}\cdots\tilde{f}_{i_{l}}(b_{\lambda})$ in $B(\lambda)$, we have
\[
\psi_{k}(b)=\tilde{f}_{i_{1}}^{k}\cdots\tilde{f}_{i_{l}}^{k}(b_{k\lambda}).
\]
Since the vertex $b_{\lambda}^{\otimes k}$ is of highest weight $k\lambda$ in
$B(\lambda)^{\otimes k}$, one gets a particular realization $B(b_{\lambda
}^{\otimes k})$ of $B(k\lambda)$ in $B(\lambda)^{\otimes k}$ with highest
weight vertex $b_{\lambda}^{\otimes k}$. This thus gives a canonical
embedding
\begin{equation}
\theta_{k}:\left\{
\begin{array}
[c]{c}%
B(b_{\lambda})\hookrightarrow B(b_{\lambda}^{\otimes k})\subset B(b_{\lambda
})^{\otimes k}\\
b\longmapsto b_{1}\otimes\cdots\otimes b_{k}%
\end{array}
\right.  \label{embdedd}%
\end{equation}
with important properties given in the following theorem and illustrated in
Example \ref{Example-Dilatation}.

\begin{theorem}
\label{Th_Dila}(see \cite{kash2})

\begin{enumerate}
\item Let $\sigma\in\mathfrak{S}_{n}^{\lambda}.$ We have $\theta_{k}%
(b_{\sigma\lambda})=b_{\sigma\lambda}^{\otimes k}$.

\item Let $b\in B(\lambda)$. When $k$ has sufficiently many factors, there
exist elements $\sigma_{1},\ldots,\sigma_{k}$ in $\mathfrak{S}_{n}^{\lambda}$
such that $\theta_{k}(b)=b_{\sigma_{1}\lambda}\otimes\cdots\otimes
b_{\sigma_{k}\lambda}$. Moreover, in this case

\begin{enumerate}
\item the elements $b_{\sigma_{1}\lambda}$ and $b_{\sigma_{k}\lambda}$ in
$\theta_{k}(b)$ do not then depend on $k$,

\item up to repetition, the sequence $(\sigma_{1}\lambda,\ldots,\sigma
_{k}\lambda)$ in $\theta_{k}(b)$ does not depend on the realization of the
crystal $B(\lambda)$ and we have $\sigma_{1}\geq\sigma_{2}\geq\cdots\geq
\sigma_{k}$.
\end{enumerate}
\end{enumerate}
\end{theorem}

From Assertion 2 of the above theorem, we can define the keys of an element in
$B(\lambda)$\footnote{We dot use the terminology "left" and "right" keys as in
the original definition {\cite{Las1}} based on the tableaux model since it
does not fit with the positions of $b_{\sigma_{1}\lambda}$ and $b_{\sigma
_{k}\lambda}$ in $\theta_{k}(b)$ with the convention of this paper.}.

\begin{definition}
\label{DefKeys}Let $b\in B(\lambda)$, then the keys $K_{+}(b)$ and $K^{-}(b)$
of $b$ are defined as follows:
\[
K_{+}(b)=b_{\sigma_{1}\lambda}\text{ and }K^{-}(b)=b_{\sigma_{k}\lambda}.
\]
{In particular, $K_{+}(b_{\sigma\lambda})=K^{-}(b_{\sigma\lambda}%
)=b_{\sigma\lambda}$ for any $\sigma\in\mathfrak{S}_{n}^{\lambda}$. The orbit
$O(\lambda)$ is simultaneously the set of left and right keys of ${B}%
(\lambda)$.}
\end{definition}


\subsubsection{Tableau realization}

\label{Subsubsec_tableaux}Recall that each partition $\lambda$ in
$\mathcal{P}_{n}$ can be identified with its Young diagram.\ A semistandard
tableau $T$ of shape $\lambda$ is then a filling of $\lambda$ by letters in
the ordered alphabet $\mathcal{A}_{n}=\{1<\cdots<n\}$ whose rows weakly
increase from left to right and columns strictly increase from {bottom to
top}. The \textsf{row reading} of $T$ is the word $w(T)$ of $\mathcal{A}_{n}^{\ast}$
obtained by reading each row from right to left starting with the bottom row
and ending with the top row. {The \textsf{weight} of $T$ is the vector }$\mathrm{wt}%
${$(T)\in\mathbb{Z}_{\geq0}^{n}$ whose $i$-th entry records the number of
$i$'s in the filling of $T$, for $i=1,\dots,n$.}

\begin{example}
\label{Example_tableau}For $n=4$ the tableau
\[
T=%
\begin{tabular}
[c]{|l|l|l}\cline{1-2}%
$3$ & $4$ & \\\hline
$2$ & $2$ & \multicolumn{1}{|l|}{$4$}\\\hline
$1$ & $1$ & \multicolumn{1}{|l|}{$2$}\\\hline
\end{tabular}
\ \ \ \ \
\]
is a semistandard tableau of shape $\lambda=(3,3,2,0)$ with row reading
$w(T)=21142243$ {and weight }$\mathrm{wt}${$(T)=(2,3,1,2)$.}
\end{example}

One can realize $B(\lambda)$ using the semistandard tableaux of shape
$\lambda$ just by describing the action of the crystals operators $\tilde
{f}_{i}$ and $\tilde{e}_{i},i=1,\ldots,n-1$ on each such tableau. Assume that
$i$ is fixed in $\{1,\ldots,n-1\}$ and $T$ is a semistandard tableau of shape
$\lambda$. Let $w_{i}(T)$ be the subword of $w(T)$ obtained by keeping
only the letters $i$ and $i+1$ in $w(T)$.\ Now delete recursively all the
\emph{factors} $i(i+1)$ in $w_{i}(T)$. This eventually yields a subword
$\tilde{w}_{i}(T)$ of ${w}(T)$ of the form $\tilde{w}_{i}%
(T)=(i+1)^{a}(i)^{b}$. When $b>0$ (resp. $a>0$), $\tilde{f}_{i}(T)$ (resp.
$\tilde{e}_{i}(T)$) is obtained by replacing in $T$ the letter of
${w}(T)$ corresponding to the leftmost letter $i$ (resp. to the
rightmost $i+1$) surviving in $\tilde{w}_{i}(T)$ by $i+1$ (resp. by $i$). When
$b=0$ (resp. $a=0$), we set $\tilde{f}_{i}(T)=0$ (resp. $\tilde{e}_{i}(T)=0$)
where $0$ is understood as a sink vertex{ as before}. This just means that in
this case, there is no arrow $i$ starting at $T$ (resp. no arrow $i$ ending at
$T$). Observe that with the notation of the previous paragraph one gets


\[
\varepsilon_{i}(T)=a\text{ and }\varphi_{i}(T)=b.
\]
Also, it is easy to compute the action of $s_{i}=(i,i+1)\in\mathfrak{S}_{n}$
on $T:$ the tableau $s_{i}.T$ is obtained by replacing in $T$ the $a-b$
{rightmost} letters $i+1$ (resp. the $b-a$ {leftmost} letters $i$) of
$\tilde{w}_{i}(T)$ by $i$ (resp. by $i+1$) when $a\geq b$ (resp. $a<b$).


\begin{example}
By resuming Example~\ref{Example_tableau}, {one} gets%
\begin{align*}
\tilde{f}_{1}(T)  &  =0\text{ and }\tilde{e}_{1}(T)=%
\begin{tabular}
[c]{|l|l|l}\cline{1-2}%
$3$ & $4$ & \\\hline
$2$ & $2$ & \multicolumn{1}{|l|}{$4$}\\\hline
$1$ & $1$ & \multicolumn{1}{|l|}{$ 1$}\\\hline
\end{tabular}
\\
\tilde{f}_{2}(T)  &  =%
\begin{tabular}
[c]{|l|l|l}\cline{1-2}%
$3$ & $4$ & \\\hline
$2$ & $2$ & \multicolumn{1}{|l|}{$4$}\\\hline
$1$ & $1$ & \multicolumn{1}{|l|}{$3$}\\\hline
\end{tabular}
\ \ \text{and }\tilde{e}_{2}(T)=0\text{ with }{s_{2}.T}=%
\begin{tabular}
[c]{|l|l|l}\cline{1-2}%
$3$ & $4$ & \\\hline
$2$ & $3$ & \multicolumn{1}{|l|}{$4$}\\\hline
$1$ & $1$ & \multicolumn{1}{|l|}{$3$}\\\hline
\end{tabular}
\\
\tilde{f}_{3}(T)  &  =%
\begin{tabular}
[c]{|l|l|l}\cline{1-2}%
$4$ & $4$ & \\\hline
$2$ & $2$ & \multicolumn{1}{|l|}{$4$}\\\hline
$1$ & $1$ & \multicolumn{1}{|l|}{$2$}\\\hline
\end{tabular}
\ \ \text{ and }\tilde{e}_{3}(T)=%
\begin{tabular}
[c]{|l|l|l}\cline{1-2}%
$3$ & $3$ & \\\hline
$2$ & $2$ & \multicolumn{1}{|l|}{$4$}\\\hline
$1$ & $1$ & \multicolumn{1}{|l|}{$2$}\\\hline
\end{tabular}
\end{align*}

\end{example}

With the above definition of the crystal operators, it is easy to check that
the set of semistandard tableaux of shape $\lambda$ admits the structure of an
oriented and connected graph isomorphic to the abstract crystal $B(\lambda)$
(see \cite{kash}). In particular its unique highest weight vertex is the
Yamanouchi tableau $T_{\lambda}$ whose $i$-th row only contains letters $i$
for any $i=1,\ldots,n$. In fact the orbit $O(\lambda)$ is also easy to
describe in this model: it exactly contains the so-called \textsf{key
tableaux} of shape $\lambda$ which are the semistandard tableaux in which each
column is contained in the column located immediately at its left. Their
weights correspond to the orbit of $\lambda\in\mathbb{Z}^{n}$ under the action
of $\mathfrak{S}_{n}$.

\begin{example}
For $n=3$, the six {key} tableaux (or simply keys) of shape $\lambda=(2,1,0)$
are%
\begin{align*}
&
\begin{tabular}
[c]{|l|l}\cline{1-1}%
$2$ & \\\hline
$1$ & \multicolumn{1}{|l|}{$1$}\\\hline
\end{tabular}
\ \ \ \ ,\quad%
\begin{tabular}
[c]{|l|l}\cline{1-1}%
$3$ & \\\hline
$1$ & \multicolumn{1}{|l|}{$1$}\\\hline
\end{tabular}
\ \ \ \ ,\quad%
\begin{tabular}
[c]{|l|l}\cline{1-1}%
$2$ & \\\hline
$1$ & \multicolumn{1}{|l|}{$2$}\\\hline
\end{tabular}
\ \ \ \ ,\\
&
\begin{tabular}
[c]{|l|l}\cline{1-1}%
$3$ & \\\hline
$2$ & \multicolumn{1}{|l|}{$2$}\\\hline
\end{tabular}
\ \ \ \ ,\quad%
\begin{tabular}
[c]{|l|l}\cline{1-1}%
$3$ & \\\hline
$1$ & \multicolumn{1}{|l|}{$3$}\\\hline
\end{tabular}
\ \ \ \ ,\quad%
\begin{tabular}
[c]{|l|l}\cline{1-1}%
$3$ & \\\hline
$2$ & \multicolumn{1}{|l|}{$3$}\\\hline
\end{tabular}
\end{align*}

\end{example}

\begin{example}
\label{Example-Dilatation}For $n=3$, $\lambda=(2,1,0)$ and $k=2$,
the crystal $B(\lambda)$ and its dilatation $B(\lambda)^{\otimes2}$ are as follows:%

\[%
\begin{tabular}
[c]{lllllll}
&  &  & $%
\begin{tabular}
[c]{|l|l}\cline{1-1}%
$2$ & \\\hline
$1$ & \multicolumn{1}{|l|}{$1$}\\\hline
\end{tabular}
$ &  &  & \\
&  & $%
\genfrac{.}{.}{0pt}{}{\text{{\tiny 1}}}{\swarrow}%
$ &  & $%
\genfrac{.}{.}{0pt}{}{\text{{\tiny 2}}}{\searrow}%
$ &  & \\
& $%
\begin{tabular}
[c]{|l|l}\cline{1-1}%
$2$ & \\\hline
$1$ & \multicolumn{1}{|l|}{$2$}\\\hline
\end{tabular}
$ &  &  &  & $%
\begin{tabular}
[c]{|l|l}\cline{1-1}%
$3$ & \\\hline
$1$ & \multicolumn{1}{|l|}{$1$}\\\hline
\end{tabular}
$ & \\
& \ \ {\tiny 2}$\downarrow$ &  &  &  & $\ \ ${\tiny 1}$\downarrow$ & \\
& $%
\begin{tabular}
[c]{|l|l}\cline{1-1}%
$2$ & \\\hline
$1$ & \multicolumn{1}{|l|}{$3$}\\\hline
\end{tabular}
$ &  &  &  & $%
\begin{tabular}
[c]{|l|l}\cline{1-1}%
$3$ & \\\hline
$1$ & \multicolumn{1}{|l|}{$2$}\\\hline
\end{tabular}
$ & \\
& \ \ {\tiny 2}$\downarrow$ &  &  &  & \ \ {\tiny 1}$\downarrow$ & \\
& $%
\begin{tabular}
[c]{|l|l}\cline{1-1}%
$3$ & \\\hline
$1$ & \multicolumn{1}{|l|}{$3$}\\\hline
\end{tabular}
$ &  &  &  & $%
\begin{tabular}
[c]{|l|l}\cline{1-1}%
$3$ & \\\hline
$2$ & \multicolumn{1}{|l|}{$2$}\\\hline
\end{tabular}
$ & \\
&  & $%
\genfrac{.}{.}{0pt}{}{\text{{\tiny 1}}}{\searrow}%
$ &  & $%
\genfrac{.}{.}{0pt}{}{\text{{\tiny 2}}}{\swarrow}%
$ &  & \\
&  &  & $%
\begin{tabular}
[c]{|l|l}\cline{1-1}%
$3$ & \\\hline
$2$ & \multicolumn{1}{|l|}{$3$}\\\hline
\end{tabular}
$ &  &  & \\
&  &  &  &  &  &
\end{tabular}
\]

\[%
\begin{tabular}
[c]{lllllll}
&  &  & $%
\begin{tabular}
[c]{|l|l}\cline{1-1}%
$2$ & \\\hline
$1$ & \multicolumn{1}{|l|}{$1$}\\\hline
\end{tabular}
\ \ \otimes%
\begin{tabular}
[c]{|l|l}\cline{1-1}%
$2$ & \\\hline
$1$ & \multicolumn{1}{|l|}{$1$}\\\hline
\end{tabular}
\ \ $ &  &  & \\
&  & $%
\genfrac{.}{.}{0pt}{}{\text{{\tiny 1}}^{2}}{\swarrow}%
$ &  & $%
\genfrac{.}{.}{0pt}{}{\text{{\tiny 2}}^{2}}{\searrow}%
$ &  & \\
& $%
\begin{tabular}
[c]{|l|l}\cline{1-1}%
$2$ & \\\hline
$1$ & \multicolumn{1}{|l|}{$2$}\\\hline
\end{tabular}
\ \ \otimes%
\begin{tabular}
[c]{|l|l}\cline{1-1}%
$2$ & \\\hline
$1$ & \multicolumn{1}{|l|}{$2$}\\\hline
\end{tabular}
\ \ \medskip$ &  &  &  & $%
\begin{tabular}
[c]{|l|l}\cline{1-1}%
$3$ & \\\hline
$1$ & \multicolumn{1}{|l|}{$1$}\\\hline
\end{tabular}
\ \ \otimes%
\begin{tabular}
[c]{|l|l}\cline{1-1}%
$3$ & \\\hline
$1$ & \multicolumn{1}{|l|}{$1$}\\\hline
\end{tabular}
\ \ \medskip$ & \\
& \ \ \ \ \ \ {\tiny 2}$^{2}\downarrow$ &  &  &  & \ \ \ \ \ \ {\tiny 1}%
$^{2}\downarrow$ & \\
& $%
\begin{tabular}
[c]{|l|l}\cline{1-1}%
$3$ & \\\hline
$1$ & \multicolumn{1}{|l|}{$3$}\\\hline
\end{tabular}
\ \ \otimes%
\begin{tabular}
[c]{|l|l}\cline{1-1}%
$2$ & \\\hline
$1$ & \multicolumn{1}{|l|}{$2$}\\\hline
\end{tabular}
\ \ \medskip$ &  &  &  & $\ \
\begin{tabular}
[c]{|l|l}\cline{1-1}%
$3$ & \\\hline
$2$ & \multicolumn{1}{|l|}{$2$}\\\hline
\end{tabular}
\ \ \otimes%
\begin{tabular}
[c]{|l|l}\cline{1-1}%
$3$ & \\\hline
$1$ & \multicolumn{1}{|l|}{$1$}\\\hline
\end{tabular}
\medskip$ & \\
& \ \ \ \ \ \ {\tiny 2}$^{2}\downarrow$ &  &  &  & \ \ \ \ \ \ {\tiny 1}%
$^{2}\downarrow$ & \\
& $%
\begin{tabular}
[c]{|l|l}\cline{1-1}%
$3$ & \\\hline
$1$ & \multicolumn{1}{|l|}{$3$}\\\hline
\end{tabular}
\ \ \otimes%
\begin{tabular}
[c]{|l|l}\cline{1-1}%
$3$ & \\\hline
$1$ & \multicolumn{1}{|l|}{$3$}\\\hline
\end{tabular}
\ \ $ &  &  &  & $%
\begin{tabular}
[c]{|l|l}\cline{1-1}%
$3$ & \\\hline
$2$ & \multicolumn{1}{|l|}{$2$}\\\hline
\end{tabular}
\ \ \otimes%
\begin{tabular}
[c]{|l|l}\cline{1-1}%
$3$ & \\\hline
$2$ & \multicolumn{1}{|l|}{$2$}\\\hline
\end{tabular}
\ \ $ & \\
&  & $%
\genfrac{.}{.}{0pt}{}{\text{{\tiny 1}}^{2}}{\searrow}%
$ &  & $%
\genfrac{.}{.}{0pt}{}{\text{{\tiny 2}}^{2}}{\swarrow}%
$ &  & \\
&  &  & $%
\begin{tabular}
[c]{|l|l}\cline{1-1}%
$3$ & \\\hline
$2$ & \multicolumn{1}{|l|}{$3$}\\\hline
\end{tabular}
\ \ \otimes%
\begin{tabular}
[c]{|l|l}\cline{1-1}%
$3$ & \\\hline
$2$ & \multicolumn{1}{|l|}{$3$}\\\hline
\end{tabular}
\ \ $ &  &  &
\end{tabular}
\ \
\]

\[
{K_{+}\left(
\begin{tabular}
[c]{|l|l}\cline{1-1}%
$2$ & \\\hline
$1$ & \multicolumn{1}{|l|}{$3$}\\\hline
\end{tabular}
\right)  =%
\begin{tabular}
[c]{|l|l}\cline{1-1}%
$3$ & \\\hline
$1$ & \multicolumn{1}{|l|}{$3$}\\\hline
\end{tabular}
,\ ~~K^{-}\left(
\begin{tabular}
[c]{|l|l}\cline{1-1}%
$2$ & \\\hline
$1$ & \multicolumn{1}{|l|}{$3$}\\\hline
\end{tabular}
\right)  =%
\begin{tabular}
[c]{|l|l}\cline{1-1}%
$2$ & \\\hline
$1$ & \multicolumn{1}{|l|}{$2$}\\\hline
\end{tabular}
,~~K_{+}\left(
\begin{tabular}
[c]{|l|l}\cline{1-1}%
$3$ & \\\hline
$1$ & \multicolumn{1}{|l|}{$2$}\\\hline
\end{tabular}
\right)  =%
\begin{tabular}
[c]{|l|l}\cline{1-1}%
$3$ & \\\hline
$2$ & \multicolumn{1}{|l|}{$2$}\\\hline
\end{tabular}
,~~K^{-}\left(
\begin{tabular}
[c]{|l|l}\cline{1-1}%
$3$ & \\\hline
$1$ & \multicolumn{1}{|l|}{$2$}\\\hline
\end{tabular}
\right)  =%
\begin{tabular}
[c]{|l|l}\cline{1-1}%
$3$ & \\\hline
$1$ & \multicolumn{1}{|l|}{$1$}\\\hline
\end{tabular}
.\ }%
\]


\end{example}

\begin{remark}
\

\begin{enumerate}
\item In the previous example, the dilatation of the crystal with $k=2$
suffices to obtain the left and right keys.\ In general, we need to compute
the dilatation with $k$ given by the {  least common multiple of the maximal
lengths of the $i$-chains} with
$i\in\{1,\ldots,n-1\}$ in $B(\lambda)$.

\item The left and right keys associated to a semistandard tableau can be
computed in a more efficient way than the one obtained from Definition
\ref{DefKeys} by using the Jeu de Taquin procedure { \cite{Las1,Fu}}. This was
in fact the initial definition from~\cite{Las1}.\ One can also use the
semi-skyline model { \cite{HHL08, Mas}} to realize the crystal $B(\lambda)$ in
a way which makes the keys very easy to read off (but the crystal structure
becomes then more complicated to describe { \cite{Mas,AE, AO2}}).\ The
advantage of Definition~\ref{DefKeys} is that it is independent of the
realization of the crystal $B(\lambda)$ and strongly connected to general
properties of $\mathfrak{S}_{n}$ viewed as a Coxeter group.

\item In the notation of \S \ \ref{Subsec_representation}, we have
$O(\lambda)=\{u.T_{\lambda}\mid u\in\mathfrak{S}_{n}^{\lambda}\}$. This gives
a direct correspondence between the keys and the elements of $\mathfrak{S}%
_{n}^{\lambda}$. If we denote by $K_{u}$ the key $u.T_{\lambda}$ associated to
$u\in\mathfrak{S}_{n}^{\lambda}$, it then becomes easy to read the Bruhat order.
Indeed, we have $u\leq v$ if and only if for each box of the Young diagram
$\lambda$, the letter obtained in $K_{u}$ is less than or equal to the one obtained in
$K_{v}$.

\item The character $s_{\lambda}$ associated to the partition $\lambda$ is the
Schur function and the tableau realization of crystals allows one to recover
its expression%
\begin{equation}
s_{\lambda}=\sum_{T\in B(\lambda)}x^{\mathrm{wt}{(T)}}. \label{Schur}%
\end{equation}

\end{enumerate}
\end{remark}

\subsubsection{Crystals of Demazure modules}

\label{Subsec_DemaModules}Let $\lambda$ be a partition and  $\sigma
\in\mathfrak{S}_{n}$.\ Up to scalar multiplication, there exists a unique
vector $v_{\sigma\lambda}$ in $V(\lambda)$ of weight $\sigma(\lambda)$. The
Demazure module associated to $v_{\sigma\lambda}$ is the $U(\mathfrak{gl}%
_{n}^{+})$-module defined by
\[
V_{\sigma}(\lambda):=U(\mathfrak{gl}_{n}^{+})\cdot v_{\sigma\lambda}.
\]
Demazure \cite{Dem} introduced the character $\mathrm{\kappa}_{\sigma,\lambda
}$ of $V_{\sigma}(\lambda)$ and showed that it can be computed by applying to
$x^{\lambda}$ a sequence of divided difference operators given by any reduced
decomposition of $\sigma$. More precisely, for any $i\in\{1,\ldots,n-1\}$,
define the linear operator $D_{i}$ on $\mathbb{Z}[x_{1},\ldots,x_{n}]$ by%
\[
D_{i}(P)=\frac{x_{i}P-x_{i+1}(s_{i}\cdot P)}{x_{i}-x_{i+1}}.
\]
Demazure proved that such operators satisfy the relations
\begin{align*}
D_{i}^{2}  &  =D_{i}\text{ for any }i=1,\ldots,n-1\text{,}\\
D_{i}D_{i+1}D_{i}  &  =D_{i+1}D_{i}D_{i+1}\text{ for any }i=1,\ldots
,n-2\text{,}\\
D_{i}D_{j}  &  =D_{j}D_{i}\text{ for any }i,j=1,\ldots,n-1\text{ such that
}\left\vert i-j\right\vert >1\text{.}%
\end{align*}
Thus, given any reduced decomposition $\sigma=s_{i_{1}}\cdots s_{i_{\ell}}$ of
$\sigma$, by Mastumoto's Lemma the operator $D_{\sigma}=D_{i_{1}}\cdots
D_{i_{\ell}}$ only depends on $\sigma$ and not on the chosen reduced decomposition. He also showed that
\[
\mathrm{\kappa}_{\sigma,\lambda}=D_{\sigma}(x^{\lambda})\in\mathbb{Z}%
[x_{1},\ldots,x_{n}]
\]
is the \textsf{(Demazure) character} of $V_{\sigma}(\lambda)$. In particular,
we have $\mathrm{\kappa}_{id,{\lambda}}=x^{\lambda}$ and $\mathrm{\kappa
}_{\sigma_{0,\lambda}}=s_{\lambda}$ and%
\begin{equation}
D_{i}(\mathrm{\kappa}_{\sigma,\lambda})=\left\{
\begin{array}
[c]{l}%
\mathrm{\kappa}_{s_{i}\sigma,\lambda}\text{ if }\ell(s_{i}\sigma)=\ell
(\sigma)+1,\\
\mathrm{\kappa}_{\sigma,\lambda}\text{ otherwise.}%
\end{array}
\right.  \label{D_i(Demazure)}%
\end{equation}
Later Kashiwara \cite{kash} and Littelmann \cite{Lit1} defined a relevant
notion of crystals for the Demazure modules. To this end, for any
$\sigma\in\mathfrak{S}_{n}$, consider the \textsf{Demazure atom}
\[
\overline{\mathrm{B}}_{\sigma}(\lambda)=\{b\in B(\lambda)\mid K_{+}%
(b)=b_{\sigma\lambda}\}.
\]
In particular, $\overline{\mathrm{B}}_{id}(\lambda)=\{b_{\lambda}\}$.

By definition we have $\overline{\mathrm{B}}_{\sigma}(\lambda)=\overline
{\mathrm{B}}_{\sigma^{\prime}}(\lambda)$ whenever $\sigma$ and $\sigma
^{\prime}$ belong to the same left coset of $\mathfrak{S}_{n}/\mathfrak{S}%
_{\lambda}$. Writing $\sigma=uv$ with $u\in\mathfrak{S}_{n}^{\lambda}$ and
$v\in\mathfrak{S}_{\lambda}$, we get $\overline{\mathrm{B}}_{\sigma}%
(\lambda)=\overline{\mathrm{B}}_{u}(\lambda)$ from the characterization of the
strong Bruhat order.\ Thus we can assume that $\sigma$ belongs to
$\mathfrak{S}_{n}^{\lambda}$. We then get $B(\lambda)=%
{\textstyle\bigsqcup\limits_{\sigma\in\mathfrak{S}_{n}^{\lambda}}}
\overline{\mathrm{B}}_{\sigma}(\lambda)$. There also exists a notion of
\textsf{opposite Demazure module}: for any $\sigma\in\mathfrak{S}_{n}$, it is defined
by $V^{\sigma}(\lambda):=U_{q}(\mathfrak{gl}_{n}^{-})\cdot v_{\sigma\lambda}$,
for which it is relevant to define the \textsf{opposite Demazure atom}%
\[
\overline{\mathrm{B}}^{\sigma}(\lambda)=\{b\in B(\lambda)\mid K^{-}%
(b)=b_{\sigma\lambda}\}.
\]
In particular we have $\overline{\mathrm{B}}^{\sigma_{0}}(\lambda)=\{b_{\sigma
_{0}\lambda}\}$.

Given $\sigma$ and $\sigma^{\prime}$ in $\mathfrak{S}_{n}^{\lambda}$, we shall
write $b_{\sigma\lambda}\leq b_{\sigma^{\prime}\lambda}$ when
$\sigma\leq\sigma^{\prime}$ (recall that $\leq$ denotes the strong Bruhat
order on $\mathfrak{S}_{n}$).

\begin{definition}
\label{keys} The \textsf{Demazure crystal} $\mathrm{B}_{\sigma}(\lambda)$ and \textsf{opposite
Demazure crystal} $\mathrm{B}^{\sigma}(\lambda)$ are defined by
\begin{align}
\mathrm{B}_{\sigma}(\lambda)  &  =%
{\textstyle\bigsqcup\limits_{\sigma^{\prime}\in\mathfrak{S}_{n}^{\lambda
},\sigma^{\prime}\leq\sigma}}
\overline{\mathrm{B}}_{\sigma^{\prime}}(\lambda)=\{b\in B(\lambda)\mid
K_{+}(b)\leq b_{\sigma\lambda}\},\label{B_w(lambda)}\\
\mathrm{B}^{\sigma}(\lambda)  &  =%
{\textstyle\bigsqcup\limits_{\sigma^{\prime}\in\mathfrak{S}_{n}^{\lambda
},\sigma\leq\sigma^{\prime}}}
\overline{\mathrm{B}}^{\sigma^{\prime}}(\lambda)=\{b\in B(\lambda)\mid
K^{-}(b)\geq b_{\sigma\lambda}\},\nonumber
\end{align}

{In particular we have $\mathrm{B}_{id}(\lambda)=\{b_{\lambda}\}$, $\mathrm{B}%
^{\sigma_{0}}(\lambda)=\{b_{\sigma_{0}\lambda}\}$ and $\mathrm{B}_{\sigma_{0}%
}(\lambda)=\mathrm{B}(\lambda)=\mathrm{B}^{id}(\lambda)$.}
\end{definition}


To compute the Demazure crystal
$\mathrm{B}_{\sigma}(\lambda)$, it therefore suffices to

\begin{itemize}
	\item compute the key map $K_{+}$ on $B(\lambda)$.
	
	\item compute the strong Bruhat order on $\mathfrak{S}_{n}^{\lambda}$, or
	alternatively on the vertices of $O(\lambda)$.
\end{itemize}

\begin{example}
	Let us resume Example \ref{Example-Dilatation} with the tableaux model.\ For
	$n=3$ and $\lambda=(2,1,0)$, consider $\sigma=s_{1}s_{2}$. We get%
	\[
	K_{+}(T_{s_{1}s_{2}})=%
	\begin{tabular}
		[c]{|l|l}\cline{1-1}%
		$3$ & \\\hline
		$2$ & \multicolumn{1}{|l|}{$2$}\\\hline
	\end{tabular}
	\]
	and $\mathrm{B}_{\sigma}(\lambda)$ contains exactly the tableaux $T$ such that
	$K_{+}(T)\leq K_{+}(T_{s_{1}s_{2}})$ (recall that this means that each entry in $T$
	is less than or equal to its corresponding entry in $T_{s_{1}s_{2}}$). These are
	all the tableaux in $B(\lambda)$ except%
	\[
	T_{1}=%
	\begin{tabular}
		[c]{|l|l}\cline{1-1}%
		$2$ & \\\hline
		$1$ & \multicolumn{1}{|l|}{$3$}\\\hline
	\end{tabular}
	,\text{ }T_{2}=%
	\begin{tabular}
		[c]{|l|l}\cline{1-1}%
		$3$ & \\\hline
		$1$ & \multicolumn{1}{|l|}{$3$}\\\hline
	\end{tabular}
	\text{ and }T_{3}=%
	\begin{tabular}
		[c]{|l|l}\cline{1-1}%
		$3$ & \\\hline
		$2$ & \multicolumn{1}{|l|}{$3$}\\\hline
	\end{tabular}
	\]
	for which we have
	\[
	K_{+}(T_{1})=T_{2}, K_{+}(T_{2})=T_{2}\text{ and }K_{+}(T_{3})=T_{3}.
	\]
	
\end{example}

The following theorem gathers results established by Kashiwara and
Littelmann (see Assertion 2 of Proposition 9.1.3 and Theorem 9.2.4 in
\cite{kash2}). For convenience, we extend $\tilde{f}_{i}$ and $\tilde{e}_{i}$,
$i\in\{1,\ldots,n-1\}$, to $B(\lambda)\sqcup\{0\}$ by setting them to map $0$
to $0$.

\begin{theorem}
\label{Th_KL} Let $\lambda\in\mathcal{P}_{n}$.

\begin{enumerate}
\item We have $\mathrm{\kappa}_{\sigma,\lambda}=\sum_{b\in\mathrm{B}_{\sigma
}(\lambda)}x^{\mathrm{wt}(b)}$.

\item For any reduced decomposition $s_{i_{1}}\cdots s_{i_{\ell}}$ of $\sigma$,
we have
\[
\mathrm{B}_{\sigma}(\lambda)=\{\tilde{f}_{i_{1}}^{k_{1}}\cdots\tilde{f}%
_{i_{\ell}}^{k_{\ell}}(b_{\lambda})\mid(k_{1},\ldots,k_{\ell})\in\mathbb{Z}_{\geq
0}^{\ell}\}\setminus\{0\}.
\]

\item For any $i$-chain $C$ in $B(\lambda)$ and any $\sigma\in\mathfrak{S}%
_{n}$, only the three following situations can appear%
\[
C\cap\mathrm{B}_{\sigma}(\lambda)=\emptyset,\quad C\cap\mathrm{B}_{\sigma
}(\lambda)=C\text{ or }C\cap\mathrm{B}_{\sigma}(\lambda)=s(C),
\]
where we recall that $S(C)$ denotes the source vertex of the chain $C$.
\end{enumerate}
\end{theorem}

\begin{remark}
\label{rmk_chains} By the previous theorem, {for any $\sigma\in\mathfrak{S}%
_{n}$ and $i\in\{1,\ldots,n-1\}$ such that $\ell(s_{i}\sigma)=\ell(\sigma)+1$
and $s_{i}\sigma\lambda\neq\sigma\lambda$, we have $\mathrm{B}_{\sigma
}(\lambda)\subset\mathrm{B}_{s_{i}\sigma}(\lambda)$. Moreover, for any
$i$-string $C\subseteq\mathrm{B}(\lambda)$, either $\mathrm{B}_{s_{i}\sigma
}(\lambda)\cap C=\mathrm{B}_{\sigma}(\lambda)\cap C=\emptyset$, $\mathrm{B}%
_{s_{i}\sigma}(\lambda)\cap C=\mathrm{B}_{\sigma}(\lambda)\cap C=C$, or
$s(C)=\mathrm{B}_{\sigma}(\lambda)\cap C$ in which case $C\subseteq
\mathrm{B}_{s_{i}\sigma}(\lambda)$.}
\end{remark}






\subsubsection{Additional remarks}

\label{Rq_Involution}

\begin{enumerate}
\item The computation of the key map on $B(\lambda)$ from the definition by
dilatation of crystals becomes quickly untractable when $\lambda$ is far
enough in the interior of the Weyl chamber.\ But as explained in
\S \ \ref{Subsubsec_tableaux} it becomes much easier if we use the tableaux
realization of crystals.

\item One can also define the \textsf{Demazure atom polynomials}
$\overline{\mathrm{\kappa}}_{\sigma,\lambda}=\sum_{b\in\overline{\mathrm{B}%
}_{\sigma}(\lambda)}x^{\mathrm{wt}(b)}$. In fact, they can also be obtained
without using the crystal theory directly from the {linear} operators
$D_{i}^{\prime}=D_{i}-id,i=1,\ldots,n-1$. These operators still satisfy the
braid relations, but here $(D_{i}^{\prime})^{2}=-D_{i}^{\prime}$ (see~
\cite{Las2}). Then for any reduced decomposition $\sigma=\sigma_{i_{1}}%
\cdots\sigma_{i_{\ell}}$, we have%
\[
\overline{\mathrm{\kappa}}_{\sigma,\lambda}:=D_{i_{1}}^{\prime}\cdots
D_{i_{\ell}}^{\prime}(x^{\lambda})=\sum_{b\in\overline{\mathrm{B}}_{\sigma
}(\lambda)}x^{\mathrm{wt}(b)}.
\]

\item Rather than labeling the Demazure crystals and the Demazure characters
of $B(\lambda)$ by elements of $\mathfrak{S}_{n}^{\lambda}$, it is often
convenient to label them directly by the elements of the orbit $\mathfrak{S}%
_{n}\lambda$. Given $\mu\in\mathfrak{S}_{n}\lambda$ such that $\mu
=\sigma\lambda$ with $\sigma\in\mathfrak{S}_{n}^{\lambda}$, we will write
$\mathrm{B}_{\mu},$ $\mathrm{B}^{\mu}$ instead of $\mathrm{B}_{\sigma}%
(\lambda),$ $\mathrm{B}^{\sigma}(\lambda)$ and $\mathrm{\kappa}_{\mu
},\overline{\mathrm{\kappa}}_{\mu}$ instead of $\mathrm{\kappa}_{\sigma
,\lambda}$ and $\overline{\mathrm{\kappa}}_{\sigma,\lambda}$\footnote{This
notation should not be confused with the subset consisting of those vertices in $B(\lambda)$
with weight $\mu$ sometimes also denoted $B(\lambda)_{\mu}$ in the
literature.}. Note that $\mathrm{\kappa}_{\sigma_{0}\lambda}=\mathrm{s}%
_{\lambda}$.

\item {Demazure characters $\{\kappa_{\mu}:$ $\mu\in$ $\mathbb{N}^{n}\}$ and
Demazure atoms $\{\bar{\kappa}_{\mu}:$ $\mu$ $\in\mathbb{N}^{n}\}$ both form
linear $\mathbb{Z}$-bases for $\mathbb{Z}[x_{1},\ldots,x_{n}]$.
The operators $D_{i}$ act on Demazure characters $\kappa_{\mu}$ via elementary
bubble sort operators $\pi_{i}$ on the entries of the weak composition
$\mu=(\mu_{1},\dots,\mu_{n})$ as follows
\begin{equation}
D_{i}(\kappa_{\mu})=%
\begin{cases}
\kappa_{s_{i}\mu} & \mbox{if }\mu_{i}>\mu_{i+1}\\
\kappa_{\mu} & \mbox{if }\mu_{i}\leq\mu_{i+1}%
\end{cases}
\Leftrightarrow D_{i}(\kappa_{\mu})=\kappa_{\pi_{i}(\mu)}. \label{dempro}%
\end{equation}
}

\item We will adopt the usual convention of {\cite{Las2}}, identifying each
$\mu\in\mathbb{Z}^{n}$ such that $\mu_{m+1}=\cdots=\mu_{n}=0$ with $(\mu
_{1},\ldots,\mu_{m})\in\mathbb{Z}^{m}$. This notation is compatible with the
definition of the Demazure characters since for any $\mu\in\mathbb{Z}^{m}$, we
have $\mathrm{s}_{\mu}(x_{1},\ldots,x_{m})=\mathrm{s}_{\mu}(x_{1},\ldots
,x_{n})$. It is also compatible with the tableaux realization of the crystals
because for any such $\mu\in\mathbb{Z}^{m}$, the Demazure crystal
$\mathrm{B}_{\mu}(\lambda)$ only contains tableaux with letters in
$\{1,\ldots,m\}$.

\item The Demazure and opposite Demazure crystals and atoms can be connected
using the Lusztig-Sch\"{u}tzenberger involution on the crystal $B(\lambda)$
defined as follows. Let $\sigma_{0}$ be the longest element of $\mathfrak{S}%
_{n}$ (defined by $\sigma_{0}(i)=n+1-i$ for any $i=1,\ldots,n$). For any
$b=\tilde{f}_{i_{1}}\cdots\tilde{f}_{i_{r}}(b_{\lambda})$, set $\mathrm{\iota
}(b)=\tilde{e}_{n-i_{1}}\cdots\tilde{e}_{n-i_{r}}(b_{\sigma_{0}\lambda})$
{where }$\mathrm{wt}${$(\iota b)=\sigma_{0}$}$\mathrm{wt}${$(b)$}. One can
prove that the map $\mathrm{\iota}$ is an involution on $B(\lambda)$ reversing
the arrows and flipping the labels $i$ and $n-i$, {and reversing the weight}.
We then have $K^{-}(b)=\sigma_{0}.K_{+}(\mathrm{\iota}(b))$. This implies
that, for any reduced decomposition $\sigma_0\sigma=s_{i_{1}}\cdots s_{i_{\ell}}%
\in\mathfrak{S}_{n}^{\lambda}$, we get
\begin{align}
\mathrm{B}^{\sigma}(\lambda)  &  =\mathrm{\iota}(\mathrm{B}_{\sigma_{0}\sigma
}(\lambda))={\{\tilde{e}_{n-i_{1}}^{k_{1}}\cdots\tilde{e}_{n-i_{\ell}%
}^{k_{\ell}}(b_{\sigma_{0}\lambda})\mid(k_{1},\ldots,k_{\ell})\in
\mathbb{Z}_{\geq0}^{\ell}\}\setminus\{0\}}\text{ and }\label{Demaz_Updow1}\\
\overline{\mathrm{B}}^{\sigma}(\lambda)  &  =\mathrm{\iota}(\overline
{\mathrm{B}}_{\sigma_{0}\sigma}(\lambda)). \label{Demaz_Updow2}%
\end{align}

\item There is also a notion of \textsf{opposite Demazure character}
$\mathrm{\kappa}_{\lambda}^{\sigma}$ for the opposite Demazure module
$V^{\sigma}(\lambda)$. It satisfies $\mathrm{\kappa}_{\lambda}^{\sigma}%
=\sum_{b\in\mathrm{B}^{\sigma}(\lambda)}x^{\mathrm{wt}(b)}$ and using the
involution $\mathrm{\iota}${ and \eqref{Demaz_Updow1}}, we have in fact%
\[
\mathrm{\kappa}_{\lambda}^{\sigma}(x_{1},\ldots,x_{n})=\mathrm{\kappa}^{\mu
}(x_{1},\ldots,x_{n})=\mathrm{\kappa}_{\sigma_{0}\mu}(x_{n},\ldots,x_{1})
\]
where $\mu=\sigma\lambda$. Since $\overline{\mathrm{B}}^{\sigma}%
(\lambda)=\mathrm{\iota}(\overline{\mathrm{B}}_{\sigma_{0}\sigma}(\lambda))$
we similarly have
\[
\overline{\mathrm{\kappa}}_{\lambda}^{\sigma}(x_{1},\ldots,x_{n}%
)=\overline{\mathrm{\kappa}}^{\mu}(x_{1},\ldots,x_{n})=\overline
{\mathrm{\kappa}}_{\sigma_{0}\mu}(x_{n},\ldots,x_{1})=\sum_{b\in
\overline{\mathrm{B}}^{\sigma}(\lambda)}x^{\mathrm{wt}(b)}.
\]

\end{enumerate}

\subsection{Bicrystals and RSK correspondence}

\label{SubsecRSK}{ Let} $m$ and $n$ be two positive integers. Denote by
$\mathcal{M}_{m,n}$ the set of matrices with $m$ rows and $n$ columns with entries in~$\mathbb{Z}_{\geq 0}$. The set $\mathcal{M}_{m,n}$ is endowed with the
structure of a $(\mathfrak{gl}_{m},\mathfrak{gl}_{n})$-bicrystal. This means
that we can define on $\mathcal{M}_{m,n}$ two commuting\footnote{\textit{i.e.}, two
operators chosen in each family commute with each other.} families of crystal
operators $\tilde{e}_{i},\tilde{f}_{i},i=1,\ldots,m-1$ and $\hat{e}_{j}%
,\hat{f}_{j},j=1,\ldots,n-1$ so that $\mathcal{M}_{m,n}$ is a crystal for both
$\mathfrak{gl}_{m}$ and $\mathfrak{gl}_{n}$. In fact $\mathcal{M}_{m,n}$ is
the crystal of the $(\mathfrak{gl}_{m},\mathfrak{gl}_{n})$-module of the
symmetric space $S(\mathbb{C}^{m}\times\mathbb{C}^{n})$ (see{
 \cite{CK, DK,Leeuw}}).

One can define the crystal operators directly on $\mathcal{M}_{m,n}$ or from
the RSK correspondence. This is a bijection
\[
\psi:\left\{
\begin{array}
[c]{c}%
\mathcal{M}_{m,n}\rightarrow%
{\textstyle\bigsqcup\limits_{\lambda\in\mathcal{P}_{\min(m,n)}}}
B_{m}(\lambda)\times B_{n}(\lambda)\\
A\longmapsto(P(A),Q(A))
\end{array}
\right.
\]
where we use the tableaux realization\footnote{Here we have written
$B_{m}(\lambda)$ and $B_{n}(\lambda)$ to make apparent the fact that we have a
$\mathfrak{gl}_{m}\times\mathfrak{gl}_{n}$-crystal.} of crystals so that
$P(A)$ and $Q(A)$ are semistandard tableaux with the same shape on the
alphabets $\{1,\ldots,m\}$ and $\{1,\ldots,n\}$, respectively. We refer to
\cite{Fu} for a complete description of the combinatorial procedure
(illustrated in the example below) based on the Schensted column insertion
procedure\footnote{The convention that we use agrees with that of \cite{KN} to
which we refer for another description of the RSK\ procedure and the
connection with biwords.}.

\begin{example}
\label{exple_suite} \label{Ex_RSK}Assume $m=4$ and $n=3$ and consider the
matrix%
\[
A=\left(
\begin{array}
[c]{ccc}%
2 & 2 & 0\\
1 & 0 & 1\\
2 & 1 & 1\\
0 & 1 & 1
\end{array}
\right)  .
\]
It can first be encoded as a tensor product of $n=3$ row tableaux on the
alphabet $\{1,2,3,4\}$ where $m_{i,j}$ gives the number of letters $i$ in the
$j$-th component of the tensor product:
\[
L_{A}=%
\begin{tabular}
[c]{|l|l|l|l|l|}\hline
$1$ & $1$ & $2$ & $3$ & $3$\\\hline
\end{tabular}
\ \ \ \otimes%
\begin{tabular}
[c]{|l|l|l|l|}\hline
$1$ & $1$ & $3$ & $4$\\\hline
\end{tabular}
\ \ \ \otimes%
\begin{tabular}
[c]{|l|l|l|}\hline
$2$ & $3$ & $4$\\\hline
\end{tabular}
\ \ \ .
\]
One then applies the column insertion procedure from left to right. This means
that we begin by reading the second column (this gives $4311$ with the
convention of \S \ \ref{Subsubsec_tableaux}) and then compute the column
insertions
\[
1\rightarrow1\rightarrow3\rightarrow4\rightarrow%
\begin{tabular}
[c]{|l|l|l|l|l|}\hline
$1$ & $1$ & $2$ & $3$ & $3$\\\hline
\end{tabular}
\ \ \ \text{.}%
\]
We thus get the tableau%
\[%
\begin{tabular}
[c]{|l|l|lllll}\cline{1-2}%
$3$ & $4$ &  &  &  &  & \\\hline
$1$ & $1$ & $1$ & \multicolumn{1}{|l}{$1$} & \multicolumn{1}{|l}{$2$} &
\multicolumn{1}{|l}{$3$} & \multicolumn{1}{|l|}{$3$}\\\hline
\end{tabular}
\ \ \ \ \ \ \
\]
in which we successively insert the letters corresponding to the reading $432$
of the third row. This gives the tableau
\[
P(A)=%
\begin{tabular}
[c]{|l|llllll}\cline{1-1}%
$4$ &  &  &  &  &  & \\\cline{1-1}\cline{1-4}%
$2$ & $3$ & \multicolumn{1}{|l}{$3$} & \multicolumn{1}{|l}{$4$} &
\multicolumn{1}{|l}{} &  & \\\hline
$1$ & $1$ & \multicolumn{1}{|l}{$1$} & \multicolumn{1}{|l}{$1$} &
\multicolumn{1}{|l}{$2$} & \multicolumn{1}{|l}{$3$} & \multicolumn{1}{|l|}{$3$%
}\\\hline
\end{tabular}
\ \ \ \ \ .
\]
The so-called "recording tableau" $Q(A)$ is obtained by filling {with} letters
$i$ the new boxes appearing during the insertion of row $i$ (the first row
being considered as inserted in the empty tableau at the beginning of the
procedure).\ We thus get%
\[
Q(A)=%
\begin{tabular}
[c]{|l|llllll}\cline{1-1}%
$3$ &  &  &  &  &  & \\\cline{1-1}\cline{1-4}%
$2$ & $2$ & \multicolumn{1}{|l}{$3$} & \multicolumn{1}{|l}{$3$} &
\multicolumn{1}{|l}{} &  & \\\hline
$1$ & $1$ & \multicolumn{1}{|l}{$1$} & \multicolumn{1}{|l}{$1$} &
\multicolumn{1}{|l}{$1$} & \multicolumn{1}{|l}{$2$} & \multicolumn{1}{|l|}{$2$%
}\\\hline
\end{tabular}
.
\]
Finally%
\[
\psi(A)=(P(A),Q(A)).
\]
Observe that we also have%
\[
\psi(^{t}A)=(Q(A),P(A))
\]
where $^{t}A$ is the transpose of the matrix $A$.
\end{example}

In any matrix $A$ in $\mathcal{M}_{m,n}$, one can consider all the paths $\pi$
starting at position $(i,j)=(1,n)$ (northeast corner of $A$) and ending at
position $(i,j)=(m,1)$ (southwest corner of $A$) where the authorized steps
have the form $(i,j)\rightarrow(i+1,j)$ or $(i,j)\rightarrow(i,j-1)$.\ To any
path $\pi$, we associate its \textsf{time} $t(\pi)$, given by the sum of the
entries along the path. Here, one can imagine that the path stops for a
duration $a_{i,j}$ at position $(i,j)$. We set%
\[
p(A)=\max_{\pi\text{ path in }A}t(\pi).
\]
The following theorem gathers a few results about the RSK correspondence that
we shall use later.

\begin{theorem}
\ \label{Th_RSK}

\begin{enumerate}
\item The map $\psi$ is bijective.

\item For any matrix $A$ in $\mathcal{M}_{m,n}$ we have $P(^{t}A)=Q(A)$ and
$Q(^{t}A)=P(A)$.

\item $\mathcal{M}_{m,n}$ has the structure of a bicrystal: given
$A\in\mathcal{M}_{m,n}$, the action of the operators $\widetilde{o}=\tilde
{e}_{i},\tilde{f}_{i},i=1,\ldots,m-1$ and $\widehat{o}=\hat{e}_{j}, \hat
{f}_{j},j=1,\ldots,n-1$ satisfies%
\[
\widetilde{o}(A)=\psi^{-1}(\widetilde{o}P(A),Q(A))\text{ and }\widehat
{o}(A)=\psi^{-1}(P(A),\widehat{o}Q(A)).
\]

\item For any matrix $A$, the integer $p(A)$ is equal to the {length of the}
longest row of the tableau $P(A)$ (or $Q(A)$). It also equals the length of a
longest decreasing sequence of the word read off from $L_{A}$.
\end{enumerate}
\end{theorem}

\begin{example}
Resuming Example~\ref{exple_suite}, one checks that $p(A)=7$. A longest
decreasing word of the word $w(L_{A})=332114311432$ read off from $L_{A}$ is
given by $33211\,11$, which has length 7. This subword corresponds in the matrix $A$
to the path
\[
a_{1,3}=0\rightarrow a_{1,2}=2\rightarrow a_{1,1}=2\rightarrow a_{2,1}%
=1\rightarrow a_{3,1}=2\rightarrow a_{4,1}=0.
\]

\end{example}

The \textsf{weight} of the matrix $A$ is the monomial in the set of variables
$X=\{x_{1},\ldots,x_{n}\}$ and $Y=\{y_{1},\ldots,y_{m}\}$ defined by
\[
(xy)^{A}=\prod_{1\leq i\leq m,1\leq j\leq n}(x_{i}y_{j})^{a_{i,j}}.
\]
On the one hand, using that $\frac{1}{1-x_{i}y_{j}}=\sum_{a_{i,j}=0}^{+\infty
}(x_{i}y_{j})^{a_{i,j}}$, we can write%

\[
\prod_{1\leq i\leq m,1\leq j\leq n}\frac{1}{1-x_{i}y_{j}}=\sum_{A\in
\mathcal{M}_{m,n}}(xy)^{A}.
\]
On the other hand, observing that, from RSK, we have $(xy)^{A}=x^{\mathrm{wt}%
(P(A))}y^{\mathrm{wt}(Q(A))}$, we obtain a Cauchy-like identity using the
bijection $\psi$ and (\ref{Schur}):%
\[
\prod_{1\leq i\leq m,1\leq j\leq n}\frac{1}{1-x_{i}y_{j}}=\sum_{A\in
\mathcal{M}_{m,n}}x^{\mathrm{wt}(P(A))}y^{\mathrm{wt}(Q(A))}=\sum_{\lambda
\in\mathcal{P}_{\min(m,n)}}s_{\lambda}(x_{1},\dots,x_{m})s_{\lambda}%
(y_{1},\dots,y_{n}).
\]

\begin{remark}
\label{Rq_doublecrystal}Recall the rule given in~\eqref{tens_crys} for the
action of $\tilde{e}_{i},\tilde{f}_{i}$ on a tensor product of crystals. The
action of any operator $\tilde{f}_{i},\tilde{e}_{i},i=1,\ldots,m-1$ on a
matrix $A$ can be computed from $P(A)$ but also from the product of row
tableaux $L_{A}$ appearing in Example \ref{Ex_RSK} just by concatenating their
reading words. In particular when $\tilde{f}_{i}$ (resp. $\tilde{e}_{i}$) acts
on the $j$-th component of $L_{A},$ the matrix $\tilde{f}_{i}(A)$ (resp.
$\tilde{e}_{i}(A)$) is obtained from $A$ just by changing $a_{i,j}$ into
$a_{i,j}-1$ and $a_{i+1,j}$ into $a_{i+1,j}+1$ (resp. $a_{i,j}$ into
$a_{i,j}+1$ and $a_{i+1,j}$ into $a_{i+1,j}-1$). Similarly, when $\hat{f}_{j}$
(resp. $\hat{e}_{j}$) acts on $A$, there is an integer $i\in\{1,\ldots,m\}$
such that $\hat{f}_{j}(A)$ (resp. $\hat{e}_{j}(A)$) is obtained from $A$ by
changing $a_{i,j}$ into $a_{i,j}-1$ and $a_{i,j+1}$ into $a_{i,j+1}+1$ (resp.
$a_{i,j}$ into $a_{i,j}+1$ and $a_{i,j+1}$ into $a_{i,j+1}-1$)$.$
\end{remark}

\subsection{Restriction of the RSK correspondence}

Let $D$ be any subset of $\{1,\ldots,m\}\times\{1,\ldots,n\}$ and write
$\mathcal{M}_{m,n}^{D}$ for the subset of $\mathcal{M}_{m,n}$ containing the
matrices $A$ such that $a_{i,j}\neq0$ only if $(i,j)\in D$. In general, the
set $\psi(\mathcal{M}_{m,n}^{D})$ is not stable by the $\mathfrak{gl}%
_{m}\times\mathfrak{gl}_{n}$-crystals operators. Nevertheless, when $D$
corresponds to the Young diagram of a fixed partition $\Lambda$, it follows from Remark \ref{Rq_doublecrystal} that $D=D_{\Lambda}$ is
stable under the action of the operators $\tilde{f}_{i},i=1,\ldots,m-1$ and
$\widehat{e}_{j},j=1,\ldots,n-1$. The case where $m=n$ and $\varrho
=(n,n-1,\ldots,1)$ is particularly interesting. In matrix coordinates, we
indeed get that%
\[
D_{\varrho}=\{(i,j)\mid1\leq j\leq i\leq n\}.
\]
The following theorem, initially established in~{\cite{Las2}} using the
combinatorics of tableaux, has been reproved in \cite{KN} using Littelmann
paths and in \cite{AE} using semi-skyline diagrams combinatorics. In these
different versions, the convention for the crystals {is} not the same and we
here follow the one from~\cite{KN} which {is} compatible with Kashiwara and
Littelmann convention for the tensor products of crystals, which is the most
usual one. {Later Fu and Lascoux \cite{FL} reproved this theorem using
properties of divided differences.}


\begin{theorem}
\label{Th_LKO}The restriction of the RSK correspondence $\psi$ to
$\mathcal{M}_{n,n}^{D_{\varrho}}$ gives a one-to-one correspondence%
\[
\psi:\mathcal{M}_{n,n}^{D_{\varrho}}\rightarrow%
{\textstyle\bigsqcup\limits_{\lambda\in\mathcal{P}_{n}}}
{\textstyle\bigsqcup\limits_{\sigma\in\mathfrak{S}_{n}^{\lambda}}}
\overline{\mathrm{B}}^{\sigma}(\lambda)\times\mathrm{B}_{\sigma}(\lambda).
\]
Then by considering the weights of the elements in both sides, we get the
Cauchy-like identity%
\begin{equation}
\prod_{1\leq j\leq i\leq n}\frac{1}{1-x_{i}y_{j}}=\sum_{\lambda\in
\mathcal{P}_{n}}\sum_{\sigma\in\mathfrak{S}_{n}^{\lambda}}\overline
{\mathrm{\kappa}}_{\lambda}^{\sigma}(x)\mathrm{\kappa}_{\sigma,\lambda}(y).
\label{StaircaseKernel}%
\end{equation}

\end{theorem}

\begin{remark}
\label{Rq_rewriteLKO}Observe that, using Remark \ref{Rq_Involution}, {namely
\eqref{Demaz_Updow2} and (\ref{B_w(lambda)})}, we have%
\begin{equation}%
{\textstyle\bigsqcup\limits_{\lambda\in\mathcal{P}_{n}}}
{\textstyle\bigsqcup\limits_{\sigma\in\mathfrak{S}_{n}^{\lambda}}}
\overline{\mathrm{B}}^{\sigma}(\lambda)\times\mathrm{B}_{\sigma}(\lambda)=%
{\textstyle\bigsqcup\limits_{\lambda\in\mathcal{P}_{n}}}
{\textstyle\bigsqcup\limits_{\sigma\in\mathfrak{S}_{n}^{\lambda}}}
\mathrm{\iota}\left(  \overline{\mathrm{B}}_{\sigma_{0}\sigma}(\lambda
)\right)  \times\mathrm{B}_{\sigma}(\lambda)=%
{\textstyle\bigsqcup\limits_{\mu=(\mu_{1},\ldots,\mu_{n})\in\mathbb{Z}^{n}}}
\mathrm{\iota}\left(  \overline{\mathrm{B}}_{\sigma_{0}\mu}\right)
\times\mathrm{B}_{\mu} \label{OtherDec}%
\end{equation}
where $\{\lambda\}=\mathfrak{S}_{n}\mu\cap\mathcal{P}_{n}$ in each product set
of the disjoint union. This gives%
\[
\prod_{1\leq j\leq i\leq n}\frac{1}{1-x_{i}y_{j}}=\sum_{\mu\in\mathbb{Z}%
_{\geq0}^{n}}\overline{\mathrm{\kappa}}^{\mu}(x)\mathrm{\kappa}_{\mu}%
(y)=\sum_{\mu\in\mathbb{Z}_{\geq0}^{n}}\overline{\mathrm{\kappa}}_{\sigma
_{0}\mu}(x_{n},\ldots,x_{1})\mathrm{\kappa}_{\mu}(y_{1},\ldots,y_{n}).
\]
{Note that in \cite{Las2} the rows of the Young diagram $\varrho$ are counted
from bottom to top, from $1$ to $n$, whereas here they are counted from $n$ to
$1$ according to the matrix notation. Replacing $x_{i}$ with $x_{n-i+1}$ in
(\ref{StaircaseKernel}), one recovers Lascoux's non-symmetric Cauchy identity
from~\cite{Las2}
\[
\prod_{i+j\leq n+1}\frac{1}{1-x_{i}y_{j}}=\sum_{\mu\in\mathbb{Z}_{\geq0}^{n}%
}\overline{\mathrm{\kappa}}^{\mu}(x_{n},\dots,x_{1})\mathrm{\kappa}_{\mu
}(y)=\sum_{\mu\in\mathbb{Z}_{\geq0}^{n}}\overline{\mathrm{\kappa}}_{\sigma
_{0}\mu}(x_{1},\ldots,x_{n})\mathrm{\kappa}_{\mu}(y_{1},\ldots,y_{n}).
\]
}
\end{remark}

\section{Operations on Demazure crystals and refined RSK}

\subsection{Parabolic restriction in Demazure crystals and truncated
staircases}

\label{SubsecPR}Let $p,n$ be two integers with $1\leq p \leq n$. The
subset $I_{p}=\{1,\ldots,p-1\}\subseteq\{1,\ldots,n-1\}$ with $I:=I_n$,
{defines} a Levi subalgebra $\mathfrak{g}_{I_{p}}$ of $\mathfrak{gl}_{n}$
isomorphic to $\mathfrak{gl}_{p}$ obtained by considering the matrices with
zero entries in positions $(i,j)$ with $i>p$ or $j>p$. {We set $\mathfrak{g}%
_{I_{n}}:=\mathfrak{gl}_{n}$. }The algebra $\mathfrak{g}_{I_{p}}$ has Weyl
group $\mathfrak{S}_{p}=\langle s_{i}\mid i\in I_{p}\rangle$ and root system
$R_{I_{p}}=R\cap\mathrm{span}\langle\alpha_{i}\mid i\in I_{p}\rangle$, where
$R$ denotes the root system of the Weyl group $\mathfrak{S}_{n}$ of
$\mathfrak{gl}_{n}$.\ Its cone of dominant weights can be identified with
$\mathcal{P}_{p}$. Given $\lambda\in \mathcal{P}_{p}=\bigoplus_{i=1}%
^{p}\mathbb{Z}\mathbf{e}_{i}$, let us denote by $B_{p}(\lambda)$ the
subcrystal of the $\mathfrak{gl}_{n}$-crystal $B_{n}(\lambda):=B(\lambda
)=B(\lambda,0^{n-p})$ obtained by keeping only the vertices connected to its
highest weight vertex $b_{\lambda}$ by $i$-arrows with $i\in I_{p}$. It
follows from the general theory of crystals that $B_{p}(\lambda)$ is a
realization of the $\mathfrak{gl}_{p}$-crystal associated to $\lambda$. In
terms of characters, this corresponds to the specialization $x_{p+1}%
=\cdots=x_{n}=0$ in the character $s_{\lambda}(x)${ of $B(\lambda)$}.\ For the
tableaux realization of crystals, we recover with $B_{p}(\lambda)$ the crystal
realization of $\mathfrak{gl}_{p}$-crystals by tableaux{ of shape $\lambda$
with entries in the alphabet $[p]$ as a subcrystal of the crystal
$B(\lambda)$ of tableaux of shape $\lambda$ in the alphabet $[n]$}. Given
$u\in\mathfrak{S}_{p}$, we will denote by $\mathrm{B}_{p,u}(\lambda
),\mathrm{B}_{p}^{u}(\lambda),\overline{\mathrm{B}}_{p,u}(\lambda)$ and
$\overline{\mathrm{B}}_{p}^{u}(\lambda)$ the Demazure, opposite Demazure and
atoms associated to $u$ in the $\mathfrak{gl}_{p}$-crystal $B_{p}(\lambda)$.

The \textsf{Coxeter monoid} associated to the symmetric group $\mathfrak{S}_{n}$ is the
monoid $\mathfrak{M}_{n}$ with generators $\boldsymbol{s}_{i},i=1,\ldots,n-1$
and relations
\begin{align*}
\boldsymbol{s}_{i}\boldsymbol{s}_{j}  &  =\boldsymbol{s}_{j}\boldsymbol{s}%
_{i}\text{ for any }i,j=1,\ldots,n-1\text{ such that }\left\vert
i-j\right\vert >1\\
\boldsymbol{s}_{i}\boldsymbol{s}_{i+1}\boldsymbol{s}_{i}  &  =\boldsymbol{s}%
_{i+1}\boldsymbol{s}_{i}\boldsymbol{s}_{i+1}\text{ for any }i=1,\ldots
,n-2\text{,}\\
\boldsymbol{s}_{i}^{2}  &  =\boldsymbol{s}_{i}\text{ for any }i=1,\ldots
,n-1\text{.}%
\end{align*}
Observe that this is exactly the same relations as those satisfied by the
Demazure operators and the map $\boldsymbol{s}_{i}\longmapsto D_{i}$ yields a
faithful representation of the monoid $\mathfrak{M}_{n}$ on $\mathbb{Z}%
[x_{1},\ldots,x_{n}]$. There is a canonical bijection between $\mathfrak{S}%
_{n}$ and $\mathfrak{M}_{n}$ sending any reduced decomposition of $\sigma
\in\mathfrak{S}_{n}$ to the same (still reduced) decomposition
$\boldsymbol{\sigma}\in\mathfrak{M}_{n}$. Given any $\sigma\in\mathfrak{S}%
_{n}$ and a reduced decomposition $\sigma=s_{i_{1}}\cdots s_{i_{\ell}}$, we
write $\sigma^{I_{p}}$ for the element of $\mathfrak{S}_{p}$ obtained by the
following procedure (see also ~\cite[Section 2]{BFL}).

\begin{algorithm}
\label{algo_cox_mon}
\
\begin{enumerate}
\item Remove all the $\boldsymbol{s}_{i_{a}}$ in $\boldsymbol{\sigma}$ such
that $i_{a}\notin I_{p}$. This yields a word in the generators of $\mathfrak{M}_n$, which may not be reduced.
\item Calculate the element of $\mathfrak{M}_n$ represented by the word obtained in~$(1)$, and denote it  $\boldsymbol{\sigma}^{I_{p}}\in\mathfrak{M}_{n}$.

\item The element $\sigma^{I_{p}}$\footnote{Note that $\sigma^{I_{p}%
}$ is \textit{not} the minimal length element in $\sigma\mathfrak{S}_{p}$ in general.} is the element of $\mathfrak{S}_{p}$ associated to
$\boldsymbol{\sigma}^{I_{p}}$ through the canonical bijection $W=\mathfrak{S}_{n}\longrightarrow
\mathfrak{M}_{n}$.
\end{enumerate}
\end{algorithm}


\begin{lemma}
The element $\sigma^{I_{p}}$ obtained by Algorithm~\ref{algo_cox_mon} does not depend on the initial reduced
decomposition chosen for $\sigma$.
\end{lemma}

\begin{proof}
See the appendix (Lemma~\ref{lem_indep_algo}) for a general proof in arbitrary Coxeter groups.
\end{proof}

We give an example of this algorithm in Example~\ref{ex_proc_cox_mon} below.

For $\sigma\neq1$, $\sigma^{I_{p}}=1$ if and only if $\sigma\in\mathfrak{S}%
_{[p,n]}$. Note that $\sigma_{0}^{I_{p}}$ is the longest element of
$\mathfrak{S}_{p}$; indeed, writing $\sigma_{0}^{[p]}$ for the longest element
of $\mathfrak{S}_{p}$ viewed inside $\mathfrak{S}_{n}$, we have $\ell
(\sigma_{0})=\ell(\sigma_{0}^{[p]})+ \ell(\sigma_{0}^{[p]} \sigma_{0})$, since
$\sigma_{0}$ has every element of $\mathfrak{S}_{n}$ appearing as a prefix.
Chosing a reduced decomposition of $\sigma_{0}$ beginning by a reduced
decomposition of $\sigma_{0}^{[p]}$ and applying Algorithm~\ref{algo_cox_mon} to
this decomposition yields an element in $\mathfrak{M}_{n}$ which is of the form
$\boldsymbol{\sigma_{0}^{[p]} x}$ for some $x\in\mathfrak{S}_{p}$. Using the
fact that $\ell(\boldsymbol{wv})\geq\ell(\boldsymbol{w})$ for all
$\boldsymbol{w}, \boldsymbol{v}\in\mathfrak{M}_{n}$, we see that we must have
$\boldsymbol{\sigma_{0}^{[p]} x}=\boldsymbol{\sigma_{0}^{[p]}}$.

\begin{example}
\label{ex_proc_cox_mon} Consider the reduced decomposition $\sigma=s_{1}%
s_{2}s_{3}s_{1}s_{2}$ in $\mathfrak{S}_{4}$ and choose $p=3$, hence
$I_{3}=\{1,2\}$.$\ $We then get in $\mathfrak{M}_{4}$%
\[
\boldsymbol{\sigma}^{I_{p}}=\boldsymbol{s}_{1}\boldsymbol{s}_{2}\widehat{\boldsymbol{s}_3}%
\boldsymbol{s}_{1}\boldsymbol{s}_{2}=\boldsymbol{s}_{1}\boldsymbol{s}_{2}%
\boldsymbol{s}_{1}\boldsymbol{s}_{2}=\boldsymbol{s}_{2}\boldsymbol{s}%
_{1}\boldsymbol{s}_{2}\boldsymbol{s}_{2}=\boldsymbol{s}_{2}\boldsymbol{s}%
_{1}\boldsymbol{s}_{2}=\boldsymbol{s}_{1}\boldsymbol{s}_{2}\boldsymbol{s}%
_{1}.
\]
Therefore, $\sigma^{I_{3}}=s_{2}s_{1}s_{2}=s_{1}s_{2}s_{1}\in\mathfrak{S}_{3}$.
\end{example}

\begin{proposition}
\label{PropT}Consider $\sigma$ in $\mathfrak{S}_{n}.$ The set $\mathfrak{S}%
_{{p}}^{\leq\sigma}=\{v\in\mathfrak{S}_{p}\mid v\leq\sigma\}$ admits
$\sigma^{I_{p}}$ has unique maximal element for $\leq$, that is%
\[
\mathfrak{S}_{{p}}^{\leq\sigma}=\{v\in\mathfrak{S}_{p}\mid v\leq\sigma^{I_{p}%
}\}.
\]

\end{proposition}

\begin{proof}
This is the same as~\cite[Theorem 2.2]{BFL} when the parabolic subgroup is $\mathfrak{S}_{p}$, and a particular case of~\cite[Lemma 3]{MM} (see also~\cite[Theorem 2.1]{RO}). For the convenience of the reader we also give a proof in the appendix (see Lemma~\ref{lem2}) in the framework of general Coxeter groups.
\end{proof}

Now, set $B^{p}(\lambda)=\mathrm{\iota}(B_{p}(\lambda))$ where $\mathrm{\iota
}$ is the involution in $B(\lambda)$ defined in Remark \ref{Rq_Involution}.{
Since
\[
B_{p}(\lambda)=\{\tilde{f}_{{i_{1}}}^{k_{1}}\cdots\tilde{f}_{i_{N}}^{k_{{N}}%
}(b_{\lambda})\mid{i_{1}},\ldots,{i_{N}}\in\lbrack p-1],N\geq1,(k_{i_{1}%
},\dots,k_{i_{N}})\in\mathbb{Z}_{\geq0}^{N}\}\setminus\{0\},
\]
we get%
\[
B^{p}(\lambda)=\{\tilde{e}_{{i_{1}}}^{k_{i_{1}}}\cdots\tilde{e}_{i_{N}%
}^{k_{i_{N}}}(b_{\sigma_{0}(\lambda,0^{n-p})})\mid{i_{1}},\ldots,{i_{N}}%
\in\lbrack n-p+1,n-1],N\geq1,(k_{i_{1}},\dots,k_{i_{N}})\in\mathbb{Z}_{\geq
0}^{N}\}\setminus\{0\}.
\]
One can
observe that $B^{p}(\lambda)$ also has the structure of a $\mathfrak{gl}_{p}%
$-crystal but this time for the root system with set of simple roots
$\{\alpha_{n-p+1},\ldots\alpha_{n-1}\}$.\ The corresponding character is
obtained from the specialization $x_{1}=\dots=x_{n-p}=0$ in the character
$s_{\lambda}(x)$ of $B(\lambda)$. }

\begin{corollary}
\label{Cor_Intersection}For any $\sigma\in\mathfrak{S}_{n}$, we have the
following equalities of sets%
\begin{align*}
1  &  :\mathrm{B}_{\sigma}(\lambda)\cap B_{p}(\lambda)=\mathrm{B}%
_{p,\sigma^{I_{p}}}(\lambda),\text{ \quad}\overline{\mathrm{B}}^{\sigma
}(\lambda)\cap B_{p}(\lambda)=\left\{
\begin{array}
[c]{l}%
\emptyset\text{ if }\sigma\notin\mathfrak{S}_{p},\\
\overline{\mathrm{B}}_{p}^{\sigma}(\lambda)\text{ otherwise}.%
\end{array}
\right. \\
2  &  :\mathrm{B}^{\sigma}(\lambda)\cap B_{p}(\lambda)=\left\{
\begin{array}
[c]{l}%
\emptyset\text{ if }\sigma\notin\mathfrak{S}_{p},\\
\mathrm{B}_{p}^{\sigma}(\lambda)\text{ otherwise},%
\end{array}
\right.  \text{ \quad}\overline{\mathrm{B}}_{\sigma}(\lambda)\cap
B_{p}(\lambda)=\left\{
\begin{array}
[c]{l}%
\emptyset\text{ if }\sigma\notin\mathfrak{S}_{p},\\
\overline{\mathrm{B}}_{p,\sigma}(\lambda)\text{ otherwise}.%
\end{array}
\right. \\
3  &  :\mathrm{B}_{\sigma}(\lambda)\cap B^{p}(\lambda)=\left\{
\begin{array}
[c]{l}%
\emptyset\text{ if }\sigma\notin\mathfrak{S}_{p},\\
\mathrm{\iota}(\mathrm{B}_{p,\sigma_{0}\sigma}(\lambda))\text{ otherwise},%
\end{array}
\right.  \text{ \quad}\overline{\mathrm{B}}^{\sigma}(\lambda)\cap
B^{p}(\lambda)=\left\{
\begin{array}
[c]{l}%
\emptyset\text{ if }\sigma\notin\sigma_{0}\mathfrak{S}_{p},\\
\mathrm{\iota}(\overline{\mathrm{B}}_{p,\sigma_{0}\sigma}(\lambda))\text{
otherwise}.%
\end{array}
\right. \\
4  &  :\mathrm{B}^{\sigma}(\lambda)\cap B^{p}(\lambda)=\mathrm{\iota
}(\mathrm{B}_{p,\sigma_{0}\sigma^{I_{p}}}(\lambda)),\text{ \quad}%
\overline{\mathrm{B}}_{\sigma}(\lambda)\cap B^{p}(\lambda)=\left\{
\begin{array}
[c]{l}%
\emptyset\text{ if }\sigma\notin\sigma_{0}\mathfrak{S}_{p},\\
\mathrm{\iota}(\overline{\mathrm{B}}_{p}^{\sigma_{0}\sigma}(\lambda))\text{
otherwise}.%
\end{array}
\right.
\end{align*}

\end{corollary}


\begin{proof}
For the equalities in the first point, we have
\[
\mathrm{B}_{\sigma}(\lambda)\cap B_{p}(\lambda)=\{b\in B_{p}(\lambda)\mid
K_{+}(b)\leq b_{\sigma\lambda}\}=\{b\in B_{p}(\lambda)\mid K_{+}(b)\leq
b_{\sigma^{I_{p}}\lambda}\}=\mathrm{B}_{p,\sigma^{I_{p}}}(\lambda)
\]
where the second equality follows from Proposition \ref{PropT} since
$K_{+}(b)$ belongs to $\mathfrak{S}_{p}$ for any $b\in B_{p}(\lambda)$.

For the second equality of the first point, we also obtain%
\[
\overline{\mathrm{B}}^{\sigma}(\lambda)\cap B_{p}(\lambda)=\{b\in
B_{p}(\lambda)\mid K^{-}(b)=\sigma\}.
\]
But $K^{-}(b)$ belongs to $\mathfrak{S}_{p}$ for any $b\in B_{p}(\lambda)$.
Therefore, by definition of the strong Bruhat order, $\sigma\leq K^{-}(b)$ is
only possible when $\sigma\in\mathfrak{S}_{p}$, whence the result. Similarly,
for the set equalities of the second point we have%
\[
\mathrm{B}^{\sigma}(\lambda)\cap B_{p}(\lambda)=\{b\in B_{p}(\lambda)\mid
K^{-}(b)\geq\sigma\}.
\]
We have that $K^{-}(b)$ belongs to $\mathfrak{S}_{p}$
and thus if $\sigma \leq K^{-}(b)$, then $\sigma$ also belongs to $\mathfrak{S}_{p}$ by definition of the Bruhat order. Therefore%
\[
\mathrm{B}^{\sigma}(\lambda)\cap B_{p}(\lambda)=\left\{
\begin{array}
[c]{l}%
\emptyset\text{ if }\sigma\notin\mathfrak{S}_{p}\\
\mathrm{B}_{p}^{\sigma}(\lambda)\text{ otherwise}%
\end{array}
\right.
\]
as claimed.\ Finally we can write%
\[
\overline{\mathrm{B}}_{\sigma}(\lambda)\cap B_{p}(\lambda)=\{b\in
B_{p}(\lambda)\mid K_{+}(b)=\sigma\}
\]
and we get the result by using that $K_{+}(b)\in\mathfrak{S}_{p}$ for any $b$
in $B_{p}(\lambda)$. The third and fourth set equalities are easily deduced
from the two previous ones by applying the involution $\mathrm{\iota}$ and
using the relation (\ref{Demaz_Updow1}).
\end{proof}

\begin{remark}In the setting of Proposition \ref{PropT}, writing $$\mathfrak{S}%
_{{p}}^{\leq\sigma}=\{v\in\mathfrak{S}_{p}\mid v\leq\sigma\}=\{v\in \mathfrak{S}_n: v\le \sigma_0^{[p]}, v\leq\sigma\}=\{v\in \mathfrak{S}_n\mid v\leq\sigma^{I_{p}%
}\},$$  this result  is restated geometrically, in the more general case of a finite Weyl group, in \cite[Section 5]{BFL} as follows. Let $\sigma\in \mathfrak{S}_{n}$. Then $$S_{\sigma_0^{[p]}}\cap S_\sigma=S_{\sigma^{I_{p}}},$$ where $S_\sigma=\cup_{v\le \sigma} \mathcal{O}_v$; here $\mathcal{O}_v$ is the orbit $B vB / B$ of the Borel subgroup $B$ of the reductive group $G$ with Weyl group $W$ acting on the flag variety $G/B$, and $S_\sigma$ is the Schubert variety, also obtained as the orbit closure of $\mathcal{O}_\sigma$ (we refer  to \cite[Section 10.2]{BFL} for definitions). Note that in this case, the identity $S_\sigma=\cup_{v\le \sigma} \mathcal{O}_v$ can be taken to be the definition of the Bruhat order.

 We note the parallel  between this Schubert variety formulation of Proposition \ref{PropT} and  the corresponding Demazure crystal formulation given by Corollary \ref{Cor_Intersection}, $(1)$. Recall that a Demazure crystal is disjoint union of Demazure atoms crystals, $B_\sigma(\lambda)={\textstyle\bigsqcup\limits_{\sigma^{\prime}\in\mathfrak{S}_{n}^{\lambda
},\sigma^{\prime}\leq\sigma}}
\overline{\mathrm{B}}_{\sigma^{\prime}}(\lambda)$. (Similarly, this can be taken to be the definition
of the Bruhat order on $\mathfrak{S}_{n}^\lambda$.) Then
 \begin{align}\mathrm{B}_{\sigma}(\lambda)\cap B_{p}(\lambda)&=\mathrm{B}_{\sigma}(\lambda)\cap B_{\sigma_0^{[p]}}(\lambda)=\mathrm{B}_{p,\sigma^{I_{p}}}(\lambda)\nonumber\\
&\Leftrightarrow
{\textstyle\bigsqcup\limits_{\sigma^{\prime}\in\mathfrak{S}_{n}^{\lambda
},\sigma^{\prime}\leq\sigma}}
\overline{\mathrm{B}}_{\sigma^{\prime}}(\lambda)\cap {\textstyle\bigsqcup\limits_{\sigma^{\prime}\in\mathfrak{S}_{n}^{\lambda
},\sigma^{\prime}\leq\sigma_0^{[p]}}}
\overline{\mathrm{B}}_{\sigma^{\prime}}(\lambda)={\textstyle\bigsqcup\limits_{\sigma^{\prime}\in\mathfrak{S}_{n}^{\lambda
},\sigma^{\prime}\leq\sigma,\sigma^{\prime}\leq\sigma_0^{[p]}}}
\overline{\mathrm{B}}_{\sigma^{\prime}}(\lambda)\nonumber\\
&={\textstyle\bigsqcup\limits_{\sigma^{\prime}\in\mathfrak{S}_{n}^{\lambda
},\sigma^{\prime}\leq\sigma^{I_p},}}
\overline{\mathrm{B}}_{\sigma^{\prime}}(\lambda)
=\mathrm{B}_{p,\sigma^{I_{p}}}(\lambda).\nonumber\end{align}

\end{remark}

\subsection{Truncated staircase}

\label{Subsec_truncated}In the following, we fix $p$ and $q$ two
nonnegative integers such that $n\geq q\geq p\geq1$. We consider the Young
diagram%
\[
D_{p,q}=\{(i,j)\mid n-p+1\leq i\leq n,1\leq j\leq q\}\cap D_{\varrho}%
\]
defined by using the matrix coordinates $(i,j)$. It is the intersection
of $D_{\varrho}$ with a quarter of plane defined by the lines $i=p$ and $j=q$
(in Cartesian coordinates). When $~n-p+1\leq q$, we get the Young diagram
(see Figure~\ref{fig:trunc})
\[
D_{p,q}=D_{\Lambda(p,q)}\text{ with }\Lambda(p,q)=(q^{n-q+1},q-1,\dots
,n-p+1).
\]
We have in particular $D_{n,n}=D_{\Lambda(n,n)}=D_{\varrho}$. Observe that {if
$n-p+1>q$,} there are also other {Young sub-diagrams} appearing but they all reduce to a
rectangle and thus do not yield anything new.

\begin{figure}[h]
\begin{center}
\begin{tikzpicture} [scale=0.3] \filldraw[color=green!25] (-1,0)
rectangle (3,5); \filldraw[color=green!25] (3,0)
rectangle (5,2);  \filldraw[color=green!25] (-1,0)
rectangle (3,5); \filldraw[color=green!25] (3,1) rectangle
(5,3);\filldraw[color=green!25] (3,3) rectangle (4,4);
\draw[line width=0.5pt](-1,0)-- (-1,5);
\draw[line width=0.5pt, red](-1,6)-- (-1,8);
\draw[line width=0.5pt, red](-1,5)-- (-1,6);
\draw[line width=0.5pt](0,0)-- (0,5);
\draw[line width=0.5pt, red](0,8)-- (0,7);
\draw[line width=0.5pt, red](1,7)-- (1,6);
\draw[line width=0.5pt](1,0)-- (1,5);
\draw[line width=0.5pt](2,0)-- (2,5);
\draw[line width=0.5pt](3,0)-- (3,5);
\draw[line width=0.5pt](4,0)-- (4,4);
\draw[line width=0.5pt](-1,0)-- (5,0);
\draw[line width=0.5pt, red](4,0)-- (7,0);
\draw[line width=0.5pt](-1,1)-- (5,1);
\draw[line width=0.5pt](4,0) rectangle  (5,3);
\draw[line width=0.5pt, red](7,1)-- (6,1);
\draw[line width=0.5pt](-1,2)-- (5,2);
\draw[line width=0.5pt, red](6,2)-- (5,2);
\draw[line width=0.5pt](-1,3)-- (4,3);
\draw[line width=0.5pt](-1,4)-- (4,4);
\draw[line width=0.5pt](-1,5)-- (3,5);
\draw[line width=0.5pt,red](1,6)-- (2,6);
\draw[line width=0.5pt,red](2,5)-- (2,6);
\draw[line width=0.5pt, red](7,0)-- (7,1);
\draw[line width=0.5pt, red](6,2)-- (6,1);
\draw[line width=0.5pt, red](5,3)-- (5,2);
\draw[line width=0.5pt, red](5,3)-- (4,3);
\draw[line width=0.5pt](4,2)rectangle (5,3);
\draw[line width=0.5pt, red](-1,8)-- (0,8);
\draw[line width=0.5pt, red](1,7)-- (0,7);
\draw[line width=0.5pt] (-4,0) -- (-4,8);
\draw[line width=0.5pt] (-4,0) -- (-3.8,0);
\draw[line width=0.5pt] (-4,8) -- (-3.8,8);
\draw[line width=0.5pt] (-2,0) -- (-2,5);
\draw[line width=0.5pt] (-2,0) -- (-1.8,0);
\draw[line width=0.5pt] (-2,5) -- (-1.8,5);
\node at (-2.7,2.5){$p$}; \node at (-4.7,4){$n$};
\draw[line width=0.5pt] (-1,-1) -- (5,-1);
\draw[line width=0.5pt] (-1,-1) -- (-1,-0.8);
\draw[line width=0.5pt] (5,-1) -- (5,-0.8);\node at (2.5,-1.7){$q$};
\end{tikzpicture}
\end{center}
\caption{The truncated Ferrers shape $\Lambda(p,q)$, in green, fitting the $p$
by $q$ rectangle so that the staircase $D_{\varrho}$ of size $n$ is the
smallest one containing $\Lambda(p,q)$. If $p\leq q$, $(p,p-1,\ldots,1)$ is
the biggest staircase inside $\Lambda(p,q)$.}%
\label{fig:trunc}%
\end{figure}
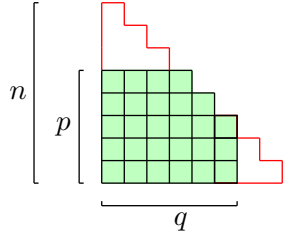


\begin{definition}
\label{Def_mu_tilde}For any $\mu=(\mu_{1},\ldots,\mu_{p})\in\mathbb{Z}_{\geq
0}^{p}$, let $\lambda\in\mathcal{P}_{p}$ and $\tau\in\mathfrak{S}_{p}%
^{\lambda}$ such that $\mu=\tau\lambda$. By applying $\sigma_{0}%
\in\mathfrak{S}_{n}$ to $\mu$, one gets $\sigma_{0}\mu=\sigma_{0}%
\tau(\lambda,0^{n-p})$. We set
\[
\widetilde{\mu}=(\sigma_{0}\tau)^{I_{q}}(\lambda,0^{q-p},0^{n-q}).
\]

\end{definition}

Note that $\tilde{\mu}$ has its last $n-q$ entries equal to zero because
$(\sigma_{0}\tau)^{I_{q}}\in\mathfrak{S}_{q}$. We will see in
\S \ \ref{SubsecSEmutilde} that it also has its first $q-p$ entries equal to zero.



\begin{example}
Consider $\mu=(1,3,2)$ and let $q=4$ and $n=5$. Then letting $\lambda=(3,2,1)$, we have
\[
\sigma_{0}\mu=(0,0,2,3,1)=s_{2}s_{1}s_{3}s_{2}s_{4}s_{3}s_{1}(3,2,1,0,0).
\]
We have
\[
(s_{2}s_{1}s_{3}s_{2}s_{4}s_{3}s_{1})^{I_{4}}=s_{2}s_{1}s_{3}s_{2}s_{3}s_{1}%
\]
which gives $\widetilde{\mu}=s_{2}s_{1}s_{3}s_{2}s_{3}s_{1}(\lambda
)=(0,1,2,3,0).$
\end{example}

\begin{theorem}
\label{Th_truncatedRectangle}With the above notation, the restriction of the
RSK correspondence $\psi$ to $\mathcal{M}_{n,n}^{D_{\Lambda(p,q)}}$ gives a
one-to-one correspondence%
\[
\psi:\mathcal{M}_{n,n}^{D_{\Lambda(p,q)}}\rightarrow%
{\textstyle\bigsqcup\limits_{\mu\in\mathbb{Z}_{\geq0}^{p}}}
\mathrm{\iota}(\overline{\mathrm{B}}_{p,\mu})\times\mathrm{B}_{q,\widetilde
{\mu}}.
\]
In particular, we have%
\[
\prod_{(i,j)\in D_{\Lambda(p,q)}}\frac{1}{1-x_{i}y_{j}}=\sum_{(\mu_{1}%
,\ldots,\mu_{p})\in\mathbb{Z}_{\geq0}^{p}}\overline{\mathrm{\kappa}}_{(\mu
_{p},\ldots,\mu_{1})}(x_{n},\ldots,x_{n-p+1})\mathrm{\kappa}_{\widetilde{\mu}%
}(y_{1},\ldots,y_{q}).
\]

\end{theorem}

\begin{proof}
By Theorem \ref{Th_LKO} together with (\ref{OtherDec}), the restriction of the map
$\psi$ from $\mathcal{M}_{n,n}^{D\varrho}$ to $\mathcal{M}_{n,n}%
^{D_{\Lambda(p,q)}}$ gives%
\begin{align*}
\psi(\mathcal{M}_{n,n}^{D_{\Lambda(p,q)}})=%
{\textstyle\bigsqcup\limits_{\lambda\in\mathcal{P}_{n}}}
{\textstyle\bigsqcup\limits_{\sigma\in\mathfrak{S}_{n}^{\lambda}}}
\overline{\mathrm{B}}^{\sigma}(\lambda)\cap B^{p}(\lambda)\times
\mathrm{B}_{\sigma}(\lambda)\cap B_{q}(\lambda)=%
{\textstyle\bigsqcup\limits_{\mu\in\mathbb{Z}_{\geq0}^{n}}}
\overline{\mathrm{B}}^{\mu}\cap B^{p}(\lambda)\times\mathrm{B}_{\mu}\cap
B_{q}(\lambda)
\end{align*}
By Corollary \ref{Cor_Intersection}, we have $\overline{\mathrm{B}}^{\sigma
}(\lambda)\cap B^{p}(\lambda)=\emptyset$ unless $\sigma\in\sigma
_{0}\mathfrak{S}_{p}^{\lambda}$, $\lambda\in\mathcal{P}_{p}$ and then
$\overline{\mathrm{B}}^{\sigma}(\lambda)\cap B^{p}(\lambda)=\mathrm{\iota
}(\overline{\mathrm{B}}_{p,\sigma_{0}\sigma}(\lambda))$. We also obtain in
this case that $\mathrm{B}_{\sigma}(\lambda)\cap B_{q}(\lambda)=\mathrm{B}%
_{q,\sigma^{I_{q}}}(\lambda).$ We thus get%
\[
\psi(\mathcal{M}_{n,n}^{D_{\Lambda(p,q)}})=%
{\textstyle\bigsqcup\limits_{\lambda\in\mathcal{P}_{p}}}
{\textstyle\bigsqcup\limits_{\sigma\in\mathfrak{S}_{n}^{\lambda}\cap\sigma
_{0}\mathfrak{S}_{p}}}
\mathrm{\iota}(\overline{\mathrm{B}}_{p,\sigma_{0}\sigma}(\lambda))\times
B_{q,\sigma^{I_{q}}}(\lambda).
\]
As usual, one can replace the two disjoint unions on $\mathcal{P}_{p}%
\times\mathfrak{S}_{n}^{\lambda}\cap\sigma_{0}\mathfrak{S}_{p}^{\lambda}$ by a
simple disjoint union on $\mathbb{Z}_{\geq0}^{p}$ by setting $\mu=\sigma
_{0}\sigma\lambda$ with $\sigma\in\mathfrak{S}_{p}^{\lambda}$. To determine
$\sigma^{I_{q}}(\lambda)$ from $\mu$, we can compute $\lambda$ by reordering
its coordinates; one then gets $\widehat{\mu}=\sigma_{0}\mu$, and $\sigma
\in\mathfrak{S}_{n}^{\lambda}\cap\sigma_{0}\mathfrak{S}_{p}$ is determined by
the equality $\widehat{\mu}=\sigma\lambda$. Finally, one computes
$\sigma^{I_{q}}$ by applying Algorithm~\ref{algo_cox_mon} to $\sigma$.\ In particular
$\widehat{\mu}$ has its first $n-p$ coordinates equal to zero and can be
written $\widehat{\mu}=(0^{n-p},\mu_{p},\ldots,\mu_{1})$ with $(\mu
,0^{n-p})=\sigma_{0}.\widehat{\mu}=(\mu_{1},\ldots,\mu_{p},0^{n-p}%
)$.\ Therefore, $\sigma^{I_{q}}(\lambda)=\widetilde{\mu}$ as introduced in
Definition \ref{Def_mu_tilde}.\ By considering all the partitions $\lambda$ in
$\mathcal{P}_{p}$, we thus obtain%
\[
\psi(\mathcal{M}_{n,n}^{D_{\Lambda(p,q)}})=%
{\textstyle\bigsqcup\limits_{\mu\in\mathbb{Z}_{\geq0}^{p}}}
\mathrm{\iota}(\overline{\mathrm{B}}_{p,\mu})\times\mathrm{B}_{q,\widetilde
{\mu}}%
\]
with $\widetilde{\mu}=\sigma^{I_{q}}\lambda$ in each set of the disjoint
union. Finally, we get the Cauchy-like identity by considering the characters
of both sides of the set equality.
\end{proof}
\begin{remark}{
When $p=q=n$, $\tilde\mu=\sigma_0\mu$ and we recover Theorem \ref{Th_LKO}
\begin{align*}
\psi:\mathcal{M}_{n,n}^{D_{\Lambda(n,n)}}&\rightarrow%
{\textstyle\bigsqcup\limits_{\mu\in\mathbb{Z}_{\geq0}^{n}}}
\mathrm{\iota}(\overline{\mathrm{B}}_{\mu})\times\mathrm{B}_{\sigma_0
{\mu}}={\textstyle\bigsqcup\limits_{\mu\in\mathbb{Z}_{\geq0}^{n}}}
\mathrm{\iota}(\overline{\mathrm{B}}_{\sigma_0\mu})\times\mathrm{B}_{
{\mu}}={\textstyle\bigsqcup\limits_{\mu\in\mathbb{Z}_{\geq0}^{n}}}
\overline{\mathrm{B}}^{\mu}\times\mathrm{B}_{
{\mu}}\\
A&\mapsto \psi(A)=(P(A),Q(A)): K_+(Q(A))\le K^-(P(A)).
\end{align*}
}
\end{remark}

\subsection{Demazure operators on crystals and augmented staircases}

\label{Subsec_ASC}Consider a partition $\lambda$ in $\mathcal{P}_{n}$ and any
subset $\Omega$ of $B(\lambda)$.\ We define the \textsf{character} of $\Omega$ by
setting%
\[
\mathrm{char}(\Omega)=\mathrm{char}(\Omega)(x_{1},\ldots,x_{n})=\sum
_{b\in\Omega}x^{\mathrm{wt}(b)}.
\]
Observe that
\begin{equation}
\mathrm{char}(\mathrm{\iota}(\Omega))=\sum_{b\in\Omega}x^{\sigma
_{0}\mathrm{wt}(b)}=\mathrm{char}(\Omega)(x_{n},\ldots,x_{1}).
\label{CharInvolution}%
\end{equation}
For any $i=1,\ldots,n-1$, denote by $\Delta_{i}(\Omega)$ the subset of $B(\lambda)$
obtained from $\Omega$ by applying operators $\tilde{f}_{i}^{k},k\geq0$ to the
vertices in $\Omega$, that is%
\[
\Delta_{i}(\Omega)=\{b\in B(\lambda)\mid\exists k\in\mathbb{Z}_{\geq0},\tilde
{e}_{i}^{k}(b)\in \Omega\}.
\]
By Remark~\ref{rmk_chains}, for any $\sigma\in\mathfrak{S}_{n}$ and any
$i=1,\ldots,n-1$, we have%
\begin{equation}
\Delta_{i}(\mathrm{B}_{\sigma}(\lambda))=\mathrm{B}_{\pi_{i}(\sigma\lambda
)}=\left\{
\begin{array}
[c]{l}%
\mathrm{B}_{s_{i}\sigma}(\lambda)\text{ if }\ell(s_{i}\sigma)=\ell
(\sigma)+1\text{ and }s_{i}\sigma\lambda\neq\sigma\lambda,\\
\mathrm{B}_{\sigma}(\lambda)\text{ if }\ell(s_{i}\sigma)=\ell(\sigma
)-1\mbox{ or }s_{i}\sigma\lambda=\sigma\lambda,
\end{array}
\right.  \label{Delta_iDemazure}%
\end{equation}
that is,
\[
\Delta_{i}(\mathrm{B}_{\mu})=\mathrm{B}_{\pi_{i}(\mu)}\text{ with }\mu
=\sigma\lambda
,\]
where $\pi_i$ is as defined in (\ref{Bubble}). In particular, we have $\Delta_{i}^{2}%
(\mathrm{B}_{\sigma}(\lambda))=\Delta_{i}(\mathrm{B}_{\sigma}(\lambda)).$ This
thus gives%
\begin{equation}
\sum_{b\in\Delta_{i}(\mathrm{B}_{\sigma}(\lambda))}x^{\mathrm{wt}(b)}%
=D_{i}\left(  \sum_{b\in\mathrm{B}_{\sigma}(\lambda)}x^{\mathrm{wt}%
(b)}\right)  =D_{i}(\mathrm{\kappa}_{\sigma,\lambda})=\mathrm{\kappa}_{\pi
_{i}(\sigma\lambda)} \label{Delta_iCharacter}%
\end{equation}
and by using (\ref{D_i(Demazure)}) one can interpret $\Delta_{i}$ as an
operator on Demazure crystals analogous to the operator $D_{i}$ on Demazure
characters. For the atoms, we get the following lemma.

\begin{lemma}
\label{Lemma_Delta_i_atoms}For any $\sigma$ in $\mathfrak{S}_{n}$ and any
$s_{i}$ such that $\ell(s_{i}\sigma)=\ell(\sigma)+1$ and $s_{i}\sigma
\lambda\neq\sigma\lambda$, we have%
\[
\Delta_{i}(\overline{\mathrm{B}}_{\sigma}(\lambda))=\overline{\mathrm{B}%
}_{\sigma}(\lambda)%
{\textstyle\bigsqcup}
\overline{\mathrm{B}}_{s_{i}\sigma}(\lambda)
\]
and
\begin{equation}
\overline{\mathrm{\kappa}}_{s_{i}\sigma,\lambda}+\overline{\mathrm{\kappa}%
}_{\sigma,\lambda}=\sum_{b\in\Delta_{i}(\overline{\mathrm{B}}_{\sigma}%
(\lambda))}x^{\mathrm{wt}(b)}=D_{i}\left(  \sum_{b\in\overline{\mathrm{B}%
}_{\sigma}(\lambda)}x^{\mathrm{wt}(b)}\right)  =D_{i}(\overline{\mathrm{\kappa
}}_{\lambda,\sigma}). \label{Delta_iAtoms}%
\end{equation}

\end{lemma}

\begin{proof}
For any $b$ in $B(\lambda)$ and $i=1,\ldots,n-1$, we have by definition of the
key $K_{+}$%
\[
K_{+}(\tilde{f}_{i}(b))\in\{K_{+}(b),s_{i}K_{+}(b)\}.
\]
This gives%
\[
\Delta_{i}(\overline{\mathrm{B}}_{\sigma}(\lambda))\subset\overline
{\mathrm{B}}_{\sigma}(\lambda)%
{\textstyle\bigsqcup}
\overline{\mathrm{B}}_{s_{i}\sigma}(\lambda)\text{.}%
\]

Conversely, it is clear that $\overline{\mathrm{B}}_{\sigma}(\lambda
)\subset\Delta_{i}(\overline{\mathrm{B}}_{\sigma}(\lambda))$ by definition of
$\Delta_{i}$. Now, if $b^{\prime}$ belongs to $\overline{\mathrm{B}}%
_{s_{i}\sigma}(\lambda)$, we have $K_{+}(b^{\prime})=s_{i}\sigma$ with
$\varepsilon_{i}(b_{K_{+}(b^{\prime})})>0$ because $\ell(s_{i}\sigma
)=\ell(\sigma)+1$. By the tensor product rules in crystals (\ref{tens_crys}),
there exists an integer $k$ such that $K_{+}(\tilde{e}_{i}^{k}b^{\prime
})=\sigma$, that is such that $\tilde{e}_{i}^{k}b^{\prime}\in\overline
{\mathrm{B}}_{\sigma}(\lambda)$.\ This shows the inclusion $\overline
{\mathrm{B}}_{s_{i}\sigma}(\lambda)\subset\Delta_{i}(\overline{\mathrm{B}%
}_{\sigma}(\lambda))$. The equality of characters follows from the equality of sets.
\end{proof}

\begin{remark}
\label{Remark_Delta_i_mu} \

\begin{enumerate}
\item Here again, we can reformulate (\ref{Delta_iDemazure}) and Lemma
\ref{Lemma_Delta_i_atoms} by setting $\mu=\sigma\lambda$. Using Lemma~\ref{lem_equiv_mu_weyl}, this gives%
\[
\Delta_{i}(\mathrm{B}_{\mu})=\left\{
\begin{array}
[c]{l}%
\mathrm{B}_{s_{i}\mu}\text{ if }\mu_{i}>\mu_{i+1},\\
\mathrm{B}_{\mu}\text{ otherwise},%
\end{array}
\right.
\]
and%
\[
\Delta_{i}(\overline{\mathrm{B}}_{\mu})=\overline{\mathrm{B}}_{\mu}%
{\textstyle\bigsqcup}
\overline{\mathrm{B}}_{s_{i}\mu}\text{ if }\mu_{i}>\mu_{i+1}.
\]

\item Observe that Lemma \ref{Lemma_Delta_i_atoms} does not {remain} true when
$\mu_{i}<\mu_{i+1}$. In this case, we indeed have $\Delta_{i}(\overline
{\mathrm{B}}_{\mu})=\overline{\mathrm{B}}_{\mu}$ whereas $D_{i}(\overline
{\mathrm{\kappa}}_{\mu})=0$, as can be seen from~\eqref{Delta_iAtoms}. Thus, to mimic the action of the operator
$D_{i}$ on $\overline{\mathrm{\kappa}}_{\lambda,\sigma}$ at the level of its
associated Demazure atoms, we need to replace the action of $\Delta_{i}$ on
$\overline{\mathrm{B}}_{\mu}(\lambda)$ by%
\begin{equation}
\dot{\Delta}_{i}(\overline{\mathrm{B}}_{\mu})=\left\{
\begin{array}
[c]{l}%
\Delta_{i}(\overline{\mathrm{B}}_{\mu})=\overline{\mathrm{B}}_{\mu}%
{\textstyle\bigsqcup}
\overline{\mathrm{B}}_{s_{i}\mu}\text{ if }\mu_{i}>\mu_{i+1}\\
\Delta_{i}(\overline{\mathrm{B}}_{\mu})=\overline{\mathrm{B}}_{\mu}\text{ if
}\mu_{i}=\mu_{i+1},\\
\emptyset\text{ if }\mu_{i}<\mu_{i+1}.
\end{array}
\right.  \label{DeltaDot}%
\end{equation}
We then always have
\[
\mathrm{char}(\dot{\Delta}_{i}(\overline{\mathrm{B}}_{\mu}))=D_{i}%
(\overline{\mathrm{\kappa}}_{\mu}).
\]

\end{enumerate}
\end{remark}

We may linearize the action described in (\ref{DeltaDot}) above by defining an
action of the{ monoid of} Demazure operators $D_{i}$ on a free $\mathbb{Z}%
$-module of rank $|\mathfrak{S}_{n}^{\lambda}|$ generated by the formal
symbols{ $\{\bar{c}_{\sigma\lambda}:\sigma\in\mathfrak{S}_{n}^{\lambda}\}$,
written $\bigoplus_{\sigma\in\mathfrak{S}_{n}^{\lambda}}\mathbb{Z}\bar
{c}_{\sigma\lambda}$}, by setting


%

\begin{equation}
D_{i}(\bar{c}_{\sigma\lambda})=\left\{
\begin{array}
[c]{l}%
\bar{c}_{\sigma\lambda}+\bar{c}_{s_{i}\sigma\lambda}\text{ if }\mu
=\sigma\lambda~\text{satisfies}~\mu_{i}>\mu_{i+1}\\
\bar{c}_{\sigma\lambda}\text{ if }\mu=\sigma\lambda~\text{satisfies}~\mu
_{i}=\mu_{i+1},\\
0\text{ if }\mu=\sigma\lambda~\text{satisfies}~\mu_{i}<\mu_{i+1}.
\end{array}
\right.  \label{D_i(Cbar)}%
\end{equation}
These operators satisfy the braid relations together with the relations
$D_{i}^{2}=D_{i}$, hence for every $w\in\mathfrak{S}_{n}$ we can write $D_{w}$
to mean $D_{i_{1}}D_{i_{2}}\cdots D_{i_{k}}$, where $s_{i_{1}}s_{i_{2}}\cdots
s_{i_{k}}$ is a reduced decomposition of $w$ in $\mathfrak{S}_{n}$. Note that the conditions on the weight $\mu$ can be entirely reformulated in terms of the Weyl group $\mathfrak{S}_n$ (Lemma~\ref{lem_equiv_mu_weyl}).

{The following lemma establishes crucial properties of the action of the
operators }$D_{w}$ on the basis {$\{\bar{c}_{\sigma\lambda}:\sigma
\in\mathfrak{S}_{n}^{\lambda}\}$, which will be used in the proof of
Theorem~\ref{Th_RSK_ASC} below.}

\begin{lemma}
\label{Lemma_Thomas}We have

\begin{enumerate}
\item Let $A\subseteq\mathfrak{S}_{n}^{\lambda}$ and $w\in\mathfrak{S}_{n}$.
Then there exists $B\subseteq\mathfrak{S}_{n}^{\lambda}$ such that
\[
D_{w} \left(  \sum_{\sigma\in A} \bar c_{\sigma\lambda}\right)  =\sum
_{\sigma\in B}\bar c_{\sigma\lambda}.
\]

\item Let $\tau, \tau^{\prime}\in\mathfrak{S}_{n}^{\lambda}$ with $\tau
\neq\tau^{\prime}, w\in\mathfrak{S}_{n}$. Then there are $A_{1},
A_{2}\subseteq\mathfrak{S}_{n}^{\lambda}$ with $A_{1} \cap A_{2} = \emptyset$
such that $D_{w}(\bar c_{\tau\lambda})=\sum_{\sigma\in A_{1}}\bar
c_{\sigma\lambda}$ and $D_{w}(\bar c_{\tau^{\prime}\lambda})=\sum_{\sigma\in
A_{2}} \bar c_{\sigma\lambda}$.
\end{enumerate}
\end{lemma}


\begin{proof}
Let us first prove the first point. By induction on the length $\ell(w)$ of
$w$, it suffices to prove the result for $w=s_{i}$, where $i\in\{1,2,\dots
,n-1\}$. Let $A=A_{1}\bigsqcup A_{2}\bigsqcup A_{3}$, where $A_{1}=\{\sigma\in
A\ |\ \mu=\sigma\lambda~\text{satisfies}~\mu_{i}>\mu_{i+1}\}$, $A_{2}%
=\{\sigma\in A\ |\ \mu=\sigma\lambda~\text{satisfies}~\mu_{i}<\mu_{i+1}\}$,
and $A_{3}=\{\sigma\in A\ |\ \mu=\sigma\lambda~\text{satisfies}~\mu_{i}%
=\mu_{i+1}\}$. By (\ref{D_i(Cbar)}) we have
\begin{align*}
D_{i}\left(  \sum_{\sigma\in A}\bar{c}_{\sigma\lambda}\right)   &
=D_{i}\left(  \sum_{\sigma\in A_{1}}\bar{c}_{\sigma\lambda}\right)
+D_{i}\left(  \sum_{\sigma\in A_{2}}\bar{c}_{\sigma\lambda}\right)
+D_{i}\left(  \sum_{\sigma\in A_{3}}\bar{c}_{\sigma\lambda}\right)  =\\
&  =\left(  \sum_{\sigma\in A_{1}}(\bar{c}_{\sigma\lambda}+\bar{c}%
_{s_{i}\sigma\lambda})\right)  +0+\sum_{\sigma\in A_{3}}\bar{c}_{\sigma
\lambda}\\
&  =\sum_{\sigma\in A_{1}}\bar{c}_{\sigma\lambda}+\sum_{\sigma\in s_{i}A_{1}%
}\bar{c}_{\sigma\lambda}+\sum_{\sigma\in A_{3}}\bar{c}_{\sigma\lambda}.
\end{align*}
To conclude the proof, it suffices to notice that $s_{i}A_{1}\subseteq
\{\sigma\in\mathfrak{S}_{n}^{\lambda}\ |\ \mu=\sigma\lambda~\text{satisfies}%
~\mu_{i}<\mu_{i+1}\}$, hence the union $A_{1}\bigcup s_{i}A_{1}\bigcup A_{3}$
is still disjoint. Therefore setting $B:=A_{1}\bigsqcup s_{i}A_{1}\bigsqcup
A_{3}$ we get the result.

We now prove the second point. By the first point, there is $B\subseteq
\mathfrak{S}_{n}^{\lambda}$ such that $D_{w}(\bar c_{\tau\lambda}+\bar
c_{\tau^{\prime}\lambda})=\sum_{\sigma\in B} \bar c_{\sigma\lambda}$. But,
still by the first point, there are also $A_{1}, A_{2}\subseteq\mathfrak{S}%
_{n}^{\lambda}$ such that $D_{w}(\bar c_{\tau\lambda})=\sum_{\sigma\in A_{1}}
\bar c_{\sigma\lambda}$ and $D_{w}(\bar c_{\tau^{\prime}\lambda})=\sum
_{\sigma\in A_{2}} \bar c_{\sigma\lambda}$. We thus have
\[
\sum_{\sigma\in B} \bar c_{\sigma\lambda}=D_{w}(\bar c_{\tau\lambda}+\bar
c_{\tau^{\prime}\lambda})=D_{w}(\bar c_{\tau\lambda})+D_{w}(\bar
c_{\tau^{\prime}\lambda})=\sum_{\sigma\in A_{1}} \bar c_{\sigma\lambda}%
+\sum_{\sigma\in A_{2}} \bar c_{\sigma\lambda},
\]
which forces $B$ to be the disjoint union of $A_{1}$ and $A_{2}$.
\end{proof}

\bigskip

In {\cite{Las2}} Lascoux gave other non-symmetric Cauchy type identities for any
partition $\Lambda\in\mathcal{P}_{n}$. The idea is to consider the largest
staircase $\rho_{\Lambda}=(m,m-1,\ldots,1)$ contained in the Young diagram of
$\Lambda$. Then one can choose a box $b$ at position $(i_{0},j_{0})$, {in
Cartesian coordinates}, in the augmented staircase $(m+1,m,\ldots,1)$ which is
not in $\Lambda$. The diagonal $L_{i,j}:j-i=j_{0}-i_{0}$, { in Cartesian
coordinates}, cuts $\Lambda$ in a northwest part and a southeast part
corresponding to the boxes above and below $L_{i,j}$, respectively.\ Now fill
the boxes $(i,j)$, {in the $n\times n$ matrix convention}, of the $NW$ part of
$\Lambda$ by $ n-i$ (\textit{i.e.}, by the { $n\times n$ matrix { reverse row index (equivalently counting rows from bottom to top)} minus one)},
and the boxes $(i,j)$ of the $SE$ part by $j-1$ (\textit{i.e.}, {by the index of the
column minus one}). Let $\sigma(\Lambda,NW)=s_{i_{1}}\cdots s_{i_{a}}$ be the
element of $\mathfrak{S}_{n}$ {where the word $i_{1}\cdots{i_{a}}$ is obtained
from} right to left column reading of the $NW$ part of $\Lambda,$ each column
being {read} from top to bottom. Similarly, let $\sigma(\Lambda,SE)=s_{j_{1}%
}\cdots s_{j_{b}}$ be the element of $\mathfrak{S}_{n}$ {where the word
${j_{1}}\cdots j_{b}$ is obtained from} {top to bottom row reading} of the
$SE$ part of $\Lambda,$ each row being {read} from right to left.

\begin{example}
Let $n=8$ and $\Lambda=(7,4,2,2,2)$. Take $(i_{0},j_{0})=(3,3)$. We have
$m=4$ and $\rho_{\Lambda}=(4,3,2,1)$,%
\[
\Lambda=%
\begin{tabular}
[c]{|l|l|lllll}\cline{1-2}%
$4$ & $ 4$ &  &  &  &  & \\\cline{1-2}%
$\blacksquare$ & $ 3$ &  &  &  &  & \\\cline{1-1}\cline{1-2}%
$\blacksquare$ & $\blacksquare$ & $\blacktriangle$ &  &  &  & \\\cline{1-2}%
\cline{1-4}%
$\blacksquare$ & $\blacksquare$ & $\blacksquare$ & \multicolumn{1}{|l}{$3$} &
\multicolumn{1}{|l}{} &  & \\\hline
$\blacksquare$ & $\blacksquare$ & $\blacksquare$ &
\multicolumn{1}{|l}{$\blacksquare$} & \multicolumn{1}{|l}{$4$} &
\multicolumn{1}{|l}{$5$} & \multicolumn{1}{|l|}{$6$}\\\hline
\end{tabular}
\ \
,\]
and we have $\sigma(\Lambda,NW)= s_{4}s_{3}s_{4}$, and $\sigma(\Lambda,SE)=s_{3}%
s_{6}s_{5}s_{4}$.
\end{example}

The following theorem was established in {\cite{Las2}} and reproved { for near stair shapes} in
\cite{AO2}.

\begin{theorem}
\label{Th_ASC}With the previously introduced notation, we have%
\[
\prod_{(i,j)\in\Lambda}\frac{1}{1-x_{i}y_{j}}=\sum_{(\mu_{1},\ldots,\mu
_{m})\in\mathbb{Z}^{m}}D_{\sigma(\Lambda,NW)}\overline{\mathrm{\kappa}}%
_{(\mu_{m},\ldots,\mu_{1})}(x_{n},\ldots,x_{n-m+1})D_{\sigma(\Lambda
,SE)}\mathrm{\kappa}_{(\mu_{1},\ldots,\mu_{m})}(y_{1},\ldots,y_{m}),
\]
where $D_{\sigma(\Lambda,NW)}=D_{i_{1}}\cdots D_{i_{a}}$ and $D_{\sigma
(\Lambda,SE)}=D_{j_{1}}\cdots D_{j_{b}}$.
\end{theorem}

\begin{remark}
\label{Rq_Skew} \

\begin{enumerate}
\item By setting $(\mu_{1},\ldots,\mu_{m})=\sigma\lambda$ with $\sigma
\in\mathfrak{S}_{m}$ and $\lambda\in\mathcal{P}_{m}$, we get by
(\ref{CharInvolution})%
\[
\overline{\mathrm{\kappa}}_{(\mu_{m},\ldots,\mu_{1})}(x_{n},\ldots
,x_{n-m+1})=\mathrm{char}\left(  \mathrm{\iota}\left(  \overline{\mathrm{B}%
}_{(\mu_{m},\ldots,\mu_{1})}\right)  \right)  =\mathrm{char}\left(
\overline{\mathrm{B}}^{(\mu_{1},\ldots,\mu_{m})}\right)  .
\]

\item Observe that both decompositions $s_{i_{1}}\cdots s_{i_{a}}$ and
$s_{j_{1}}\cdots s_{j_{b}}$ of $\sigma(\Lambda,NW)$ and $\sigma(\Lambda,SE)$
are reduced.
\end{enumerate}
\end{remark}

By using the operators $\Delta_{i}$ on Demazure crystals, one can now deduce
from this identity of characters an analogue of Theorem \ref{Th_LKO} for the
augmented staircases.

\begin{theorem}
\label{Th_RSK_ASC}With the previously introduced notation, the restriction of
the RSK correspondence $\psi$ to $\mathcal{M}_{n,n}^{D_{\Lambda}}$ gives a
one-to-one correspondence%
\begin{equation}
\psi:\mathcal{M}_{n,n}^{D_{\Lambda}}\rightarrow%
{\textstyle\bigsqcup\limits_{(\mu_{1},\ldots,\mu_{m})\in\mathbb{Z}_{\geq0}%
^{m}}}
\mathrm{\iota}\left(  \dot{\Delta}_{\sigma(\Lambda,NW)}%
(\overline{\mathrm{B}}_{(\mu_{m},\ldots,\mu_{1})})\right)  \times
\Delta_{\sigma(\Lambda,SE)}\left(  \mathrm{B}_{(\mu_{1},\ldots,\mu_{m}%
)}\right)  \label{Image}%
\end{equation}
where $\Delta_{\sigma(\Lambda,SE)}=\Delta_{j_{1}}\cdots\Delta_{j_{b}}$ and
$\dot{\Delta}_{\sigma(\Lambda,NW)}=\dot{\Delta}_{i_{1}%
}\cdots\dot{\Delta}_{i_{a}}$.\footnote{It follows from the definition of
$\dot{\Delta}$ that product sets of the form $\emptyset\times U$ can appear in
the right {hand} side of (\ref{Image}) and then $\emptyset\times U$
$=\emptyset$ as usual.}
\end{theorem}

\begin{proof}
First we need to prove that the right hand side $\mathcal{I}$ of (\ref{Image})
is indeed a disjoint union.\ To this end, first observe that for any $\nu
\in\mathbb{Z}_{\geq0}^{m}$, we have
\[
\dot{\Delta}_{\sigma(\Lambda,NW)}(\overline{\mathrm{B}}_{\nu
})=\emptyset\Longleftrightarrow D_{\sigma(\Lambda,NW)}(\overline{\mathrm{\kappa}}_{\nu})=0.
\]
When $ \dot{\Delta}_{\sigma(\Lambda,NW)}(\overline{\mathrm{B}%
}_{\nu})\neq\emptyset$, by point~(1) of Lemma \ref{Lemma_Thomas} we get the
existence of a set $A_{\nu}\subset\mathfrak{S}_{n}\lambda$ such that
\[
D_{\sigma(\Lambda,NW)}(\overline{\mathrm{\kappa}}_{\nu}%
)=\sum_{\delta\in A_{\nu}}\overline{\mathrm{\kappa}}_{\delta}\text{ and hence }%
\dot{\Delta}_{\sigma(\Lambda,NW)}(\overline{\mathrm{B}}_{\nu
})=%
{\textstyle\bigsqcup\limits_{\delta\in A_{\nu}}}
\overline{\mathrm{B}}_{\delta}.
\]
Now by point~(2) of Lemma~\ref{Lemma_Thomas}, we must have $A_{\nu}\cap
A_{\nu^{\prime}}=\emptyset$ for any $\nu^{\prime}\in\mathbb{Z}^{m}$ distinct
from $\nu$. Observe also that $\Delta_{\sigma(\Lambda,SE)}\left(  \mathrm{B}_{(\mu_{1},\ldots,\mu_{m})}\right)$ is a Demazure crystal by Lemma \ref{Lemma_Delta_i_atoms}.\ We also get that
\[%
 {\textstyle\bigsqcup\limits_{(\mu_{1},\ldots,\mu_{m})\in\mathbb{Z}_{\geq0}^{m}}}
\dot{\Delta}_{\sigma(\Lambda,NW)}(\overline{\mathrm{B}}_{(\mu_{m},\ldots,\mu_{1})})
\]
is a disjoint union of atoms. 
This permits to conclude that the set
\begin{equation}
\mathcal{I}\subset%
{\textstyle\bigsqcup\limits_{\lambda\in\mathcal{P}_{n}}}
B(\lambda)\times B(\lambda)=\psi(\mathcal{M}_{n,n}) \label{Inclusion}%
\end{equation}
is indeed a disjoint union composed of Cartesian products sets of an opposite
atom and a Demazure crystal which all lie in $\psi(\mathcal{M}_{n,n})$. Indeed, the Cartesian products sets so obtained from distint sequences $(\mu_{1},\ldots,\mu_{m})$ cannot intersect.

Now, by Theorem \ref{Th_LKO} and its alternative formulation (\ref{OtherDec}),
the RSK correspondence on $\mathcal{M}_{n,n}$ restricts to a bijection%

\[
\psi:\mathcal{M}_{n,n}^{D_{\rho_{\Lambda}}}\rightarrow%
{\textstyle\bigsqcup\limits_{(\mu_{1},\ldots,\mu_{m})\in\mathbb{Z}^{m}}}
\mathrm{\iota}\left(  \overline{\mathrm{B}}_{(\mu_{m},\ldots,\mu_{1})}\right)
\times\mathrm{B}_{(\mu_{1},\ldots,\mu_{m})}.
\]

Then, the pre-image $\psi^{-1}(\mathcal{I})\subset\mathcal{M}_{n,n}$ (which is
well-defined by (\ref{Inclusion})) is obtained as the image of $\mathcal{M}%
_{n,n}^{D_{\rho_{\Lambda}}}$ under compositions of crystal operators of the
form $\hat{f}_{j_{1}}^{k_{1}}\cdots\hat{f}_{j_{b}}^{k_{b}}$ and $\tilde
{e}_{i_{1}}^{l_{1}}\cdots\tilde{e}_{i_{a}}^{l_{a}}$ (because the involution
$\mathrm{\iota}$ changes each $\tilde{f}_{n-i}$ into $\tilde{e}_{i}$). By
Remark \ref{Rq_doublecrystal}, this shows that $\psi^{-1}(\mathcal{I})$ is
contained in $\mathcal{M}_{n,n}^{D_{\rho_{\Lambda}}}$. To get the equality
$\psi^{-1}(\mathcal{I})=\mathcal{M}_{n,n}^{D_{\rho_{\Lambda}}}$, it suffices
to consider the characters of both sets which coincide thanks to Theorem
\ref{Th_ASC}, Equalities (\ref{Delta_iCharacter}) and { Remark \ref{Rq_Skew}}.
\end{proof}

\begin{example}

 Let $ n=8$, and $ \Lambda=(7,4,2,2,2)$. We have  $m=4$,  $\varrho_{\Lambda}=(4,3,2,1)$ and
\[\begin{array}{cccccc}
\Lambda=%
\begin{tabular}
[c]{|l|l|lllll}\cline{1-2}%
$\color{blue}4$ & $\color{blue}4$ &  &  &  &  & \\\cline{1-2}%
$\blacksquare$ & $\color{blue}3$ &  &  &  &  & \\\cline{1-1}\cline{1-2}%
$\blacksquare$ & $\blacksquare$ & $\blacktriangle$ &  &  &  & \\\cline{1-2}%
\cline{1-4}%
$\blacksquare$ & $\blacksquare$ & $\blacksquare$ & \multicolumn{1}{|l}{$\color{magenta}3$} &
\multicolumn{1}{|l}{} &  & \\\hline
$\blacksquare$ & $\blacksquare$ & $\blacksquare$ &
\multicolumn{1}{|l}{$\blacksquare$} & \multicolumn{1}{|l}{$\color{magenta}4$} &
\multicolumn{1}{|l}{$\color{magenta}5$} & \multicolumn{1}{|l|}{$\color{magenta}6$}\\\hline
\end{tabular}
\end{array}
\]
with  $\sigma(\Lambda,NW)=s_{\color{blue}4}s_{\color{blue}3}s_{\color{blue}4}$, $\sigma(\Lambda,SE)=s_{\color{magenta}3}%
s_{\color{magenta}6}s_{\color{magenta}5}s_{\color{magenta}4}$.

Then $\psi$ is the RSK applied to $\mathcal{M}_{8,8}^{D_{(7,4,2,2,2)}}$ the set of $8\times 8$  non negative integer matrices whose positive entries fit the shape $\Lambda=(7,4,2,2,2)$,
\begin{align*}\psi:\mathcal{M}_{8,8}^{D_{(7,4,2,2,2)}}&\rightarrow%
{\textstyle\bigsqcup\limits_{(\mu_{1},\ldots,\mu_{4})\in\mathbb{Z}_{\geq0}%
^{4}}}
\mathrm{\iota}\left(  \dot{\Delta}_4 \dot{\Delta}_3 \dot{\Delta}_4 (\overline{\mathrm{B}}_{(\mu_{4},\ldots,\mu_{1})})\right)  \times
{\Delta}_3{\Delta}_6{\Delta}_5{\Delta}_4\left(  \mathrm{B}_{(\mu_{1},\ldots,\mu_{4}%
)}\right)\\
A&\mapsto \psi(A)=(P,Q)
\end{align*}

Let $A\in \mathcal{M}_{8,8}^{D_{(7,4,2,2,2)}}$ be as follows
\begin{align*}\scriptstyle
&A=\begin{pmatrix}
0&0&0&0&0&0&0&0\\
0&0&0&0&0&0&0&0\\
0&0&0&0&0&0&0&0\\
 \textbf{\color{blue}0}&\textbf{\color{blue}1}&0&0&0&0&0&0\\
\textbf{1}&\textbf{\color{blue}1}&0&0&0&0&0&0\\
\textbf{0}&\textbf{0}&0&0&0&0&0&0\\
\textbf{2}&\textbf{0}&\textbf{1}&\textbf{\color{magenta}1}&0&0&0&0\\
\textbf{0}&\textbf{0}&\textbf{0}&\textbf{0}&\textbf{\color{magenta}1}&\textbf{\color{magenta}0}&\textbf{\color{magenta}2}&0\\
\end{pmatrix}& \mbox{encoded by $577\otimes 45\otimes 7\otimes 7\otimes  8\otimes\emptyset\otimes 88\otimes\emptyset$.}
\end{align*}
{ It is useful to write $K(v):=b_v$ with $v\in \mathfrak{S}_n\lambda$.  For $u,v\in \mathfrak{S}_n\lambda$, the entry-wise comparison $K(u)\le K(v)$  is equivalent to $u\le v$ in $\mathfrak{S}_n\lambda$.}
Then
\begin{align*}
&
P=%
\begin{tabular}
[c]{|l|l|lll}\cline{1-2}%
$8$ & $8$ &&&\\\cline{1-3}
$7$ & $7$ & \multicolumn{1}{|l|}{$8$}&&\\\hline
$4$ & $5$ & \multicolumn{1}{|l|}{$5$}&\multicolumn{1}{|l|}{$7$}&\multicolumn{1}{|l|}{$7$}\\\hline
\end{tabular}
&
K^-(P)=\begin{tabular}
[c]{|l|l|lll}\cline{1-2}%
$8$ & $8$ &&&\\\cline{1-3}
$7$ & $7$ & \multicolumn{1}{|l|}{$8$}&&\\\hline
$4$ & $4$ & \multicolumn{1}{|l|}{$4$}&\multicolumn{1}{|l|}{$4$}&\multicolumn{1}{|l|}{$4$}\\\hline
\end{tabular}
=K(0^3,5,0^2,2,3)\\
&Q=\begin{tabular}
[c]{|l|l|lll}\cline{1-2}%
$5$ & $7$ &&&\\\cline{1-3}
$3$ & $4$ & \multicolumn{1}{|l|}{$7$}&&\\\hline
$1$ & $1$ & \multicolumn{1}{|l|}{$1$}&\multicolumn{1}{|l|}{$2$}&\multicolumn{1}{|l|}{$2$}\\\hline
\end{tabular}
 &K_+(Q)=\begin{tabular}
[c]{|l|l|lll}\cline{1-2}%
$7$ & $7$ &&&\\\cline{1-3}
$4$ & $4$ & \multicolumn{1}{|l|}{$7$}&&\\\hline
$2$ & $2$ & \multicolumn{1}{|l|}{$2$}&\multicolumn{1}{|l|}{$2$}&\multicolumn{1}{|l|}{$2$}\\\hline
\end{tabular}=K(0,5,0,2,0^2,3,0).
 \end{align*}
We show that there exists $\mu=(\mu_{1},\mu_{2},\mu_{3},\mu_{4})\in\mathbb{Z}^4_{\geq0}$ such that  \begin{align*}\psi(A)=(P,Q)\in \iota(\dot{\Delta}_4\dot{\Delta}_3\dot{\Delta}_4%
\overline{\mathrm{B}}_{(\sigma_0\mu,0^4)})\times {\Delta}_3{\Delta}_6{\Delta}_5{\Delta}_4B_{(\mu,0^4)},
\end{align*}
 where  $\sigma_0\in \mathfrak{S}_4$ and $\iota $ is the Sch\"utzenberger involution on tableaux on the alphabet $\{1,2,\dots,8\}$.
One has
\begin{align*}
&\iota (\dot{\Delta}_4\dot{\Delta}_3\dot{\Delta}_4%
\overline{\mathrm{B}}_{(\mu_{4},\ldots,\mu_{1},0^4)})=\\
&=\begin{cases}\iota\overline{\mathrm{B}}_{(\mu_{4},\mu_3,\mu_2,0,0^4)} \bigsqcup \iota\overline{\mathrm{B}}_{(\mu_{4},\mu_3,0,\mu_2,0^4)}\bigsqcup \iota\overline{\mathrm{B}}_{(\mu_{4},\mu_3,0^2,\mu_2,0^3)},\mbox{ if $\mu_2>\mu_1=0$}\\
\iota\overline{\mathrm{B}}_{(\mu_{4},\mu_3,0,0,0^4)}, \mbox{ if $\mu_1=\mu_2=0$}\\
\iota\overline{\mathrm{B}}_{(\mu_{4},\mu_3,\mu_2,\mu_1,0^4)} \bigsqcup \iota\overline{\mathrm{B}}_{(\mu_{4},\mu_3,\mu_2,0,\mu_1, 0^3)}\bigsqcup \iota\overline{\mathrm{B}}_{(\mu_{4},\mu_3,0,\mu_2,\mu_1,0^3)},\mbox{ if $\mu_1=\mu_2>0$}\\
 \emptyset,\; \mbox{ if  $\mu_1>\mu_2\ge 0$}\\
\iota\overline{\mathrm{B}}_{(\mu_{4},\dots,\mu_{1},0^4)}\bigsqcup \iota\overline{\mathrm{B}}_{(\mu_{4},\mu_3,\mu_1,\mu_{2},0^4)}\bigsqcup  \iota\overline{\mathrm{B}}_{(\mu_{4},\mu_3,\mu_2,0,\mu_{1},0^3)}\bigsqcup \iota\overline{\mathrm{B}}_{(\mu_{4},\mu_3,0,\mu_2,\mu_{1},0^3)}
\bigsqcup \iota\overline{\mathrm{B}}_{(\mu_{4},\mu_3,0,\mu_1,\mu_{2},0^3)}\bigsqcup \\ \bigsqcup\iota\overline{\mathrm{B}}_{(\mu_{4},\mu_3,\mu_1,0,\mu_{2},0^3)}, \mbox{  if $\mu_2>\mu_1>0$.}
\end{cases}
\end{align*}
and
\begin{align*}
      {\Delta}_3{\Delta}_6{\Delta}_5{\Delta}_4B_{(\mu,0^4)}= B_{\pi_3\pi_6\pi_5\pi_4(\mu,0^4)}=B_{(\mu_1,\mu_2,0,\mu_3,0,0,\mu_4,0)}.
      \end{align*}
      Then
      \begin{align*}
& K^-(P)=K(0^3,5,0^2,2,3)\Leftrightarrow
P\in \overline{\mathrm{B}}^{(0^3,5,0^2,2,3)}=
\iota
\overline{\mathrm{B}}_{(3,2,0^2,5,0^3)})\Rightarrow \mu_2=5>\mu_1=0,\,\mu_3=2,\,\mu_4=3\\
&\Rightarrow\mu=(0,5,2,3).
      \end{align*}
Indeed
$K_+(Q)=K(0,5,0,2,0^2,3,0)\le K(0,\mu_2,0,\mu_3,0,0,\mu_4,0)=K(0,5,0,2,0^2,3,0)$ and $$ ~~~~Q\in B_{(0,5,0,2,0^2,3,0)}.$$

Therefore,
\begin{align*}(P,Q)\in \overline{\mathrm{B}}^{(0^3,5,0^2,2,3)}\times \mathrm{B}_{(0,5,0,2,0^2,3,0)}\Rightarrow(P,Q)\in \iota(\dot{\Delta}_4\dot{\Delta}_3\dot{\Delta}_4%
\overline{\mathrm{B}}_{(3,2,5,0,0^4)})\times {\Delta}_3{\Delta}_6{\Delta}_5{\Delta}_4B_{(0,5,2,3,0^4)}.
\end{align*}

Given \[
R=%
\begin{tabular}
[c]{|l|l|lll}\cline{1-2}%
$6$ & $7$ &&&\\\cline{1-3}
$2$ & $4$ & \multicolumn{1}{|l|}{$7$}&&\\\hline
$1$ & $1$ & \multicolumn{1}{|l|}{$2$}&\multicolumn{1}{|l|}{$2$}&\multicolumn{1}{|l|}{$2$}\\\hline
\end{tabular}
\ \ \ \ \
\]
 with $K_+(R)= K(0,5,0,2,0^2,3,0) $ by reverse column Schensted insertion we get the matrix
\begin{align*}\psi^{-1}(P,R)=\begin{pmatrix}
0&0&0&0&0&0&0&0\\
0&0&0&0&0&0&0&0\\
0&0&0&0&0&0&0&0\\
 \textbf{\color{blue}1}&\textbf{\color{blue}0}&0&0&0&0&0&0\\
\textbf{1}&\textbf{\color{blue}1}&0&0&0&0&0&0\\
\textbf{0}&\textbf{0}&0&0&0&0&0&0\\
\textbf{0}&\textbf{3}&\textbf{1}&\textbf{\color{magenta}1}&0&0&0&0\\
\textbf{0}&\textbf{0}&\textbf{0}&\textbf{0}&\textbf{\color{magenta}1}&\textbf{\color{magenta}0}&\textbf{\color{magenta}2}&0\\
\end{pmatrix}\in\mathcal{M}_{8,8}^{D_{(7,4,2,2,2)}}.
\end{align*}

\end{example}
\subsection{The southeast approach for $\tilde\mu$}

\label{SubsecSEmutilde}We now resume the notation of \S \ \ref{Subsec_ASC} and
in particular consider integers $p$ and $q$ such that $1\leq p\leq q\leq n$
and $n-q+1\leq p$ to perform an augmentation in the SE part of the staircase
$\rho=(p,p-1,\ldots,1)$ as an alternative way to describe the truncated staircase
from \S \ \ref{Subsec_truncated}. As illustrated by the figure below, the
element $\sigma(\Lambda(p,q),SE)\in\mathfrak{S}_{q}$ {is obtained from} {top
to bottom row reading} of the $SE$ part of the augmented staircase, each row
being {read} from right to left. We thus get the following reduced decomposition
in $\mathfrak{S}_{q}$:
\begin{equation}
\sigma(\Lambda(p,q),SE)=\prod_{i=1}^{p-(n-q)-1}(s_{i+n-p-1}\dots s_{i}%
)\prod_{i=0}^{n-q}(s_{q-1}\dots s_{p-(n-q)+i}).\label{SE}%
\end{equation}

\begin{figure}[h]
\begin{center}
\begin{tikzpicture}[scale=0.68]
			\filldraw[color=green!15] (-1,0) rectangle (8,1);
			\filldraw[color=green!15] (-1,1) rectangle (8,2);
			\filldraw[color=green!15] (-1,2) rectangle (8,3);
			\filldraw[color=green!15] (-1,3) rectangle (7,4);
			\filldraw[color=green!15] (-1,4) rectangle (6,5);
			\draw[line width=1pt] (-1,0) rectangle (8,5); \draw[line width=1pt] (0,0)
			-- (0,5);
			\draw[line width= 1pt]
			(1,0) -- (1,5);
			\draw[line width= 1pt] (2,0) -- (2,5);
			\draw[line width= 1pt] (3,0) -- (3,5); \draw[line width= 1pt] (4,0)
			-- (4,5);\draw[line width= 1pt] (5,0) -- (5,5);\draw[line width=
			1pt] (6,0) -- (6,5);\draw[line width=
			1pt] (7,0) -- (7,5);
			\draw[line width= 1pt] (-1,1)-- (8,1);
			\draw[line width= 1pt] (-1,2)-- (8,2);
			\draw[line width= 1pt] (-1,3)--(8,3);
			\draw[line width= 1pt] (-1,4)--(8,4);
			\draw[line width= 2 pt, red ] (3.97,0)--(3.97,1);
			\draw[line width= 2 pt, red ] (4,1)--(3,1);
			\draw[line width= 2 pt, red ] (3,1)--(3,2);
			\draw[line width= 2 pt, red ] (3,2)--(2,2);
			\draw[line width= 2 pt, red ] (2,2)--(2,3);
			\draw[line width= 2 pt, red ] (2,3)--(1,3);
			\draw[line width= 2 pt, red ] (1,3)--(1,4);
			\draw[line width= 2 pt, red ] (1,3.95)--(0,3.95);
			\draw[line width= 2 pt, red ] (0,3.95)--(0,4.95);
			\draw[line width= 2 pt, red ] (-1,4.95)--(0,4.95);
			\draw[line width= 2 pt, green ] (8,0)--(8,3);
			\draw[line width= 2 pt, green ] (8,3)--(7,3);
			\draw[line width= 2
			pt, green ] (7,3)--(7,4); \draw[line width= 2 pt, green ]
			(6,4.03)--(7,4.03);\draw[line width= 2 pt, green ]
			(-1,5.03)--(6,5.03);\draw[line width= 2 pt, green ]
			(6,4.03)--(6,5.03); \footnotesize
			\node at (0.5, 4.5) {$1$}; \node at (1.5, 4.5) {$2$};
			\node at (1.5, 3.5) {$\ddots$};  \node at (6.5,
			3.5) {$\ddots$};\node at (5.5,
			4.5) {n-p};  \node at (3.5, 1.5) {$\ddots$};
			\node at (3.5,
			2.5) { $\ddots$};  \node at (2.5, 4.5) { $\ddots$};
			\node at (2.5, 2.7) { p-};  \node at (2.5, 2.3) {\scriptsize n+q};  \node
			at (4.5, 0.5) { p};
			\node at (5.5, 0.5) {$\ldots$};
			\node at (7.5, 0.5) { q-1};\node at (7.5, 1.5) {$\vdots$};\node
			at (7.5,
			2.5) { q-1};
			\draw[arrows=<->,line width=1 pt] (-0.8,-0.5)--(8.0,-0.5); \node at
			(4,-1){$q$};
			\draw[arrows=<->,line width=1 pt] (-1.5,0.2)--(-1.5,5.0); \node at
			(-2,2.5){$p$};
		\end{tikzpicture}
\end{center}
\caption{The labels in $\Lambda(p,q)/\rho$, $\Lambda(p,q)=(q^{n-q+1}%
,q-1,\dots,,n-p+1)$, $\rho=(p,\dots,1)$ the maximal staircase contained in
$\Lambda$, indicate the column index of $\Lambda$ minus one. The reading word,
from right to left and from the top to bottom, defines the reduced word
${\sigma(\Lambda(p,q),SE)}$.}%
\label{fig2:word}%
\end{figure}
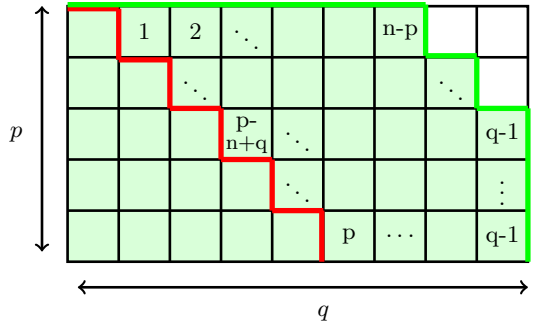

{Resuming the notation of Definition \ref{Def_mu_tilde}, let $\mu=(\mu
_{1},\ldots,\mu_{p})\in\mathbb{Z}_{\geq0}^{p}$ and $\lambda\in\mathcal{P}_{p}$
such that $\mu=\tau\lambda$, $\tau\in\mathfrak{S}_{p}^{\lambda}$, with $1\leq
p\leq q\leq n$ and $p-(n-q)\geq1\Leftrightarrow q\geq n-p+1$. Let
$\widehat{\sigma_{0}\tau}\in\mathfrak{S}_{p}^{\lambda}$ such that
$\widehat{\sigma_{0}\tau}\lambda=\sigma_{0}^{[p]}\mu$ with $\sigma_{0}^{[p]}$
the longest element of $\mathfrak{S}_{p}$ (also recall that }$\sigma_{0}$ is
the longest element of $\mathfrak{S}_{n}$). We build on \cite[Proposition 3]{AE} to show the
following proposition.

\begin{proposition}
\label{Prop_Olga}The element $\tilde{\mu}$ introduced in Definition~\ref{Def_mu_tilde} satisfies%
\[
\tilde{\mu}=(\sigma_{0}\tau)^{I_{q}}(\lambda,0^{n-p})=\pi_{{\sigma
(\Lambda(p,q),SE)}}\pi_{{\widehat{\sigma_{0}\tau}}}(\lambda,0^{n-p}%
)=\pi_{{\sigma(\Lambda(p,q),SE)}}(\sigma_{0}^{[p]}\mu,0^{n-p})
\]
where $\sigma(\Lambda(p,q),SE)\in\mathfrak{S}_{q}$ is defined as in
\eqref{SE}. Equivalently, $\pi_{(\sigma_{0}\tau)^{I_{q}}}$ and $\pi
_{\sigma(\Lambda(p,q),SE)}\pi_{\widehat{\sigma_{0}\tau}}$ have the same action on $(\lambda,0^{n-p})$ and therefore by Lemma~\ref{lem_cox_mon_1}~(2), and Lemma~ \ref{lem_cox_mon_algo} they correspond to the same minimal
representative in $\mathfrak{S}_{n}^{(\lambda,0^{n-p})}$.
\end{proposition}

\begin{proof}
On the one hand we have
\begin{align}
\sigma_{0}\tau(\lambda,0^{n-p})=\sigma_{0}(\mu,0^{n-p})=(0^{n-p}, \mu_p, \dots, \mu_2, \mu_1) \end{align}

We will show that the product $\prod_{i=1}^{p}(\pi_{i+n-p-1}\cdots\pi_{i})\pi_{\widehat{\sigma_{0}\tau}}$ of bubble sort operators has the same action on $(\lambda, 0^{n-p})$. We have
\begin{align}
\prod_{i=1}^{p}(\pi_{i+n-p-1}\cdots\pi_{i})\pi_{\widehat{\sigma_{0}\tau}}(\lambda,0^{q-p},0^{n-q})&=\nonumber\\
&=\prod_{i=1}^{p}(\pi_{i+n-p-1}\cdots\pi_{i})(\sigma_{0}^{[p]}\mu
,0^{q-p},0^{n-q})\label{EQSE00}\\
&  =\prod_{i=1}^{p-(n-q)-1}(\pi_{i+n-p-1}\cdots\pi_{i})\cdot\label{EQSE0}\\
&  \cdot\prod_{i=0}^{n-q}(\pi_{q-1+i}\cdots\pi_{p-(n-q)+i})(\mu_{p},\ldots
,\mu_{n-q+1},\ldots,\mu_{1},0^{q-p},0^{n-q}) \label{EQSE}%
\end{align}
The bubble sort operators in \eqref{EQSE00} act on the weak composition
$(\sigma_{0}^{[p]}\mu,0^{n-p})=(\mu_{p},\ldots,\mu_{1},0^{n-p})$, shifting
$n-p$ times to the right each of the $p$ entries of $\sigma_{0}^{[p]}\mu$.
This is done by shifting $n-p$ times in $\sigma_{0}^{[p]}\mu$, first in
\eqref{EQSE}, the last $n-q+1$ entries and then, in \eqref{EQSE0}, the
remaining first $p-(n-q)-1\geq0$ entries. That is,
\begin{align}
\eqref{EQSE0},\eqref{EQSE}  &=\prod_{i=1}^{p-(n-q)-1}(\pi_{i+n-p-1}\cdots\pi_{i})\cdot\prod_{i=0}^{n-q}(\pi_{q-1+i}\cdots\pi_{p-(n-q)+i})(\mu_{p},\ldots
,\mu_{n-q+1},\ldots,\mu_{1},0^{q-p},0^{n-q})\nonumber \\
& =\prod_{i=1}^{p-(n-q)-1}(\pi_{i+n-p-1}\cdots
\pi_{i})\cdot\nonumber\\
&  \cdot(\pi_{q-1}\cdots\pi_{p-(n-q)})\cdots(\pi_{n-2}\cdots\pi_{p-1})(\pi
_{n-1}\cdots\pi_{p})(\mu_{p},\ldots,\mu_{n-q+1},\ldots,\mu_{1},0^{q-p}%
,0^{n-q})\label{Turtle}\\
&  =\prod_{i=1}^{p-(n-q)-1}(\pi_{i+n-p-1}\cdots\pi_{i})(\mu_{p},\ldots
,\mu_{n-q+2},0^{n-p},\mu_{n-q+1},\ldots,\mu_{1})\nonumber\\
&  =(\pi_{n-p}\cdots\pi_{1})\cdots(\pi_{q-2}\cdots\pi_{p-(n-q)-1})(\mu
_{p},\ldots,\mu_{n-q+2},0^{n-p},\mu_{n-q+1},\ldots,\mu_{1})\nonumber\\
&  =(0^{n-p},\mu_{p},\ldots,\mu_{n-q+2},\mu_{n-q+1},\ldots,\mu_{1}).\nonumber
\end{align}
The product

\begin{equation}
\prod_{i=1}^{p-(n-q)-1}(\pi_{i+n-p-1}\cdots\pi_{i})\cdot\prod_{i=0}^{n-q}%
(\pi_{q-1+i}\cdots\pi_{p-(n-q)+i})\pi_{\widehat{\sigma_{0}\tau}}
\label{RedExp}%
\end{equation}
is a reduced decomposition in $\mathfrak{M}_{n}$ of an element from $\mathfrak{S}_n^{(\lambda, 0^{n-p})}$ which acts on
$(\lambda,0^{n-p})$ in the same way as $\sigma_{0}\tau$.

Therefore the minimal representative of $\sigma_{0}\tau$ in $\mathfrak{S}%
_{n}^{(\lambda,0^{n-p})}$ is the minimal representative of the element $u$ with reduced decomposition in $\mathfrak{S}_{n}$
\[
\prod_{i=1}^{p-(n-q)-1}(s_{i+n-p-1}\cdots s_{i})\cdot\prod_{i=0}^{n-q}%
(s_{q-1+i}\cdots s_{p-(n-q)+i})\widehat{\sigma_{0}\tau}%
\]
and hence
\[
u^{I_{q}}=\left(  \prod_{i=1}^{p-(n-q)-1}(s_{i+n-p-1}\cdots
s_{i})\cdot\prod_{i=0}^{n-q}(s_{q-1+i}\cdots s_{p-(n-q)+i})\widehat{\sigma_{0}%
\tau}\right)  ^{I_{q}},
\]
which can be calculated using Algorithm~\ref{algo_cox_mon}. Note that $u^{I_q}$ and $(\sigma_0 \tau)^{I_q}$ may not be equal in $\mathfrak{M}_n$, but they have the same action on $(\lambda, 0^{n-p})$: indeed, if $u_0$ is the common minimal representative in $\mathfrak{S}_{(\lambda, 0^{n-p})},$ the elements $\sigma_0 \tau$ and $u$ can be written in the form $u_0 u_\lambda$  and $u_0 x_\lambda$ with $u_\lambda, x_\lambda\in \mathfrak{S}_\lambda$ and $\ell(\sigma_0\tau)=\ell(u_0)+\ell(u_\lambda)$, $\ell(u)=\ell(u_0)+\ell(x_\lambda)$ respectively. By definition of Algorithm~\ref{algo_cox_mon}, we then have $\boldsymbol{(\sigma_0 \tau)}^{I_q}=\boldsymbol{u_0}^{I_q} \boldsymbol{u_{\lambda}}^{I_q}$ and $\boldsymbol{u}^{I_q}=\boldsymbol{u_0}^{I_q} \boldsymbol{x_{\lambda}}^{I_q}$ (where the product is in $\mathfrak{M}_n$; see also Remark~\ref{rem_prod_cox_mon} from the Appendix). It follows that $(\sigma_0 \tau)^{I_q}$ and $u^{I_q}$ are (now in $\mathfrak{S}_n$) of the form $u_0^{I_q} v$ and $u_0^{I_q} z$ for some $v, z\in \mathfrak{S}_{(\lambda, 0^{n-p})}$ respectively, and the second factors $v$ and $z$ thus have trivial action on $(\lambda, 0^{n-p})$.

Passing to $\mathfrak{M}_{n}$ we have a reduced decomposition
\[\boldsymbol{u}=
\left(  \prod_{i=1}^{p-(n-q)-1}(\pi_{i+n-p-1}\cdots\pi_{i})\cdot\prod_{i=0}%
^{n-q}(\pi_{q-1+i}\cdots\pi_{p-(n-q)+i})\pi_{\widehat{\sigma_{0}\tau}}\right),%
\]
hence the first step of Algorithm~\ref{algo_cox_mon} yields the word
\begin{align}
\prod_{i=1}^{p-(n-q)-1}(\pi_{i+n-p-1}\cdots\pi_{i})\cdot\left(  \prod
_{i=0}^{n-q}(\pi_{q-1+i}\cdots\pi_{p-(n-q)+i})\right)  ^{I_{q}}\pi
_{\widehat{\sigma_{0}\tau}}\nonumber,\end{align} hence \begin{align} \boldsymbol{u}^{I_q} =\prod_{i=1}^{p-(n-q)-1}(\pi_{i+n-p-1}\cdots\pi_{i})\cdot\prod_{i=0}^{n-q}%
(\pi_{q-1}\cdots\pi_{p-(n-q)+i})\pi_{\widehat{\sigma_{0}\tau}}=\pi
_{\sigma(\Lambda(p,q),SE)}\pi_{\widehat{\sigma_{0}\tau}}. \label{whale}%
\end{align}
Note that we omitted in (\ref{RedExp}) the operators with indices $\geq q$, to obtain $\pi_{\sigma
	(\Lambda(p,q),SE)}$ with $\sigma(\Lambda(p,q),SE)$ the reduced
decomposition in $\mathfrak{S}_{q}$ given in~\eqref{SE}. Hence $(\sigma_{0}%
\tau)^{I_{q}}$ and $\pi_{\sigma(\Lambda(p,q),SE)}\pi_{\widehat{\sigma_{0}\tau
}}$ have the same action on $(\lambda,0^{n-p})$ and the reduced decomposition of
the latter explicitly provides $(\sigma_{0}\tau)^{I_{q}}$ in $\mathfrak{S}%
_{n}^{(\lambda,0^{n-p})}$. This gives the desired result.
\end{proof}

We now give a simple algorithm for computing $\tilde{\mu}=(\sigma_{0}%
\tau)^{I_{q}}(\lambda,0^{n-p})$. Recall that $n-q+1\leq p$.

\begin{theorem}
\label{Th_Olga}With the previous notation, we have
\[
\tilde{\mu}=\pi_{\sigma(\Lambda(p,q),SE)}(\sigma_{0}^{[p]}\mu,0^{n-p}%
)=(0^{q-p},\alpha_{1},\dots,\alpha_{p},0^{n-q})
\]
where $\alpha=(\alpha_{1},\dots,\alpha_{p})\in\mathbb{Z}_{\geq0}^{p}$ is
computed by the following algorithm: for $i$ \emph{running from }$p$\emph{ to
}$1$

\begin{itemize}
\item for $j=i+1,\ldots,p,$ successively ignore in $\sigma_{0}^{[p]}\mu=(\mu_{p},\dots
,\mu_{1})$ the rightmost entry equal to $\alpha_{j}$,
\item set $k_{i}%
=\min\{i,n-q+1\}$,

\item then $\alpha_{i}$ is the maximum element among the remaining rightmost
$k_{i}$ entries of $(\mu_{p},\dots,\mu_{1})$.
\end{itemize}
\end{theorem}

\begin{example}\label{ex_olga_1}
Let $n=6$, $p=4$, $q=5$, $n-q+1=2$,

$(a)$ If $\mu=(2,1,2,3)$ and $\sigma_{0}^{[4]}\mu=(3,2,1,2)=\pi_{3}(3,2,2,1)$ then
$\alpha=(1,3,2,2)$ is obtained as follows: $\alpha_{4}=2$ is the maximum among
the rightmost $\min\{4,2\}=2$ entries of $(3,2,1,2)$, $\alpha_{3}=2$ is the
maximum among the rightmost $\min\{3,2\}=2$ entries of $(3,2,1)$, $\alpha_{2}=3$
is the maximum among the rightmost $\min\{2,2\}=2$ entries of $(3,1)$,
$\alpha_{1}=1$ is the maximum among the rightmost $\min\{1,2\}=1$ entries of
$(1)$.

$(b)$ If $\mu=(1,2,3,2)$ and $\sigma_{0}^{[4]}\mu=(2,3,2,1)=\pi_{1}(3,2,2,1)$ then
$\alpha=(1,2,3,2)$ is given by $\alpha_{4}=2$ is the maximum among the rightmost
$\min\{4,2\}=2$ entries of $(2,3,2,1)$, $\alpha_{3}=3$ is the maximum among the
rightmost $\min\{3,2\}=2$ entries of $(2,3,1)$, $\alpha_{2}=2$ is the maximum
among the rightmost $\min\{2,2\}=2$ entries of $(2,1)$, $\alpha_{1}=1$ is the
maximum among the rightmost $\min\{1,2\}=1$ entries of $(1)$.
\end{example}

\begin{proof}
The bubble sort operators in
\begin{align}
\pi_{\sigma(\Lambda(p,q),SE)}  &  =\prod_{i=1}^{p-(n-q)-1}(\pi_{i+n-p-1}%
\cdots\pi_{i})\cdot\prod_{i=0}^{n-q}(\pi_{q-1}\cdots\pi_{p-(n-q)+i})\nonumber\\
&  =(\pi_{n-p}\cdots\pi_{1})\cdots(\pi_{q-1}\cdots\pi_{p-(n-q)-2})(\pi
_{q-2}\cdots\pi_{p-(n-q)-1})\cdot\label{Shark0}\\
&  \cdot(\pi_{q-1}\cdots\pi_{p-(n-q)})\cdots(\pi_{q-1}\cdots\pi_{p-1})(\pi
_{q-1}\cdots\pi_{p}) \label{Shark}%
\end{align}
act on the weak composition $(\mu_{p},\ldots,\mu_{1},0^{q-p},0^{n-q})$, first
in \eqref{Shark}, shifting $q-p$ times to the right the last $n-q+1$ entries,
$\mu_{n-q+1},\ldots,\mu_{1}$, of $(\mu_{p},\ldots,\mu_{n-q+1},\ldots,\mu_{1}%
)$, and one checks that it sorts them in ascending order $(\mu_{n-q+1},\ldots,\mu_{1})_{\uparrow}%
$, to get%
\begin{align}
(\mu_{p},\ldots,\mu_{n-q+2}, 0^{q-p},(\mu_{n-q+1},\ldots,\mu_{1})_{\uparrow
},0^{n-q}). \label{Snale}%
\end{align}

Let $\alpha_{p}$ be the entry $q$ of \eqref{Snale}. Next, the operators in
\eqref{Shark0} act similarly on the resulting vector \eqref{Snale},
reordering in ascending order $\mu_{n-q+2}$ and $(\mu_{n-q+1},\ldots,\mu
_{1})\setminus\{\alpha_{p}\}$, that is, ignoring the entry $q$, in
$(\mu_{n-q+1},\ldots,\mu_{1})_{\uparrow}$, \eqref{Snale}, to get the vector%
\begin{align}
(\mu_{p},\ldots,\mu_{n-q+3}, 0^{q-p},(\mu_{n-q+2},(\mu_{n-q+1},\ldots,\mu
_{1})\setminus\{\alpha_{p}\})_{\uparrow},\alpha_{p},0^{n-q}). \label{Snale1}%
\end{align}
Let $\alpha_{p-1}$ be the entry $q-1$ of \eqref{Snale1}. Then reordering
$\mu_{n-q+3}$ with the just new previous vector \eqref{Snale1}, ignoring the
entries $q-1$ and $q$, and so on. Observe that after some point, the number of remaining
entries in $\sigma_{0}^{[p]}\mu$ is less than $n-q+1$ and just the $i$
remaining entries are considered.
\end{proof}

Let us give two examples illustrating the notation and the results of
Proposition~\ref{Prop_Olga} and Theorem~\ref{Th_Olga}:

\begin{example}
Let $n=6$, $p=4$, $q=5$ and $\Lambda=(5^{2},4,3)$ where $n-p+1=6-4+1=3<q$ and
$n-q+1=2$,
\[
\Lambda=%
\begin{matrix}
\square & 1 & 2 &  & \\
\square & \square & 2 & 3 & \\
\square & \square & \square & 3 & 4\\
\square & \square & \square & \square & 4
\end{matrix}
\qquad\sigma(\Lambda(4,5),SE)=s_{2}s_{1}\,s_{3}s_{2}\,s_{4}s_{3}\,s_{4},
\]

$(a)$ Let $\mu=(2,1,2,3)=\tau\lambda\in\mathbb{Z}_{\geq0}^{4}$, and
$\lambda=(3,2,2,1)$, $\widehat{\sigma_{0}\tau}\lambda=s_{3}\lambda=(3,2,1,2)=\sigma
_{0}^{[4]}\mu$. Then on the one hand we have
\[
\sigma_{0}\tau(3,2,2,1,0,0)=\sigma_{0}(2,1,2,3,0,0)=\sigma_{0}(\mu,0^{2})=(0,0,3,2,1,2).\] On the other hand, mimicking the proof of Proposition~\ref{Prop_Olga}, we have $\pi_{\widehat{\sigma_{0}\tau}}=\pi_{3}$ and the product of the bubble sort operators $\prod_{i=1}^{p}(\pi_{i+n-p-1}\cdots\pi_{i})\pi_{\widehat{\sigma_{0}\tau}}$ is given in this case by $\pi
_{2}\pi_{1}\pi_{3}\pi_{2}\pi_{4}\pi_{3}\pi_{5}\pi_{4}\pi_{3}$ and we have
\[\pi
_{2}\pi_{1}\pi_{3}\pi_{2}\pi_{4}\pi_{3}\pi_{5}\pi_{4}\pi_{3}(3,2,2,1,0,0)=(0,0,3,2,1,2)=\sigma_0\tau(3,2,2,1,0,0).
\]
The decomposition $s_2 s_1 s_3 s_2 s_4 s_3 s_5 s_4 s_3$ is reduced and lies in $\mathfrak{S}_6^\lambda$. We calculate
\begin{align*}
(\pi_{2}\pi_{1}\pi_{3}\pi_{2}\pi_{4}\pi_{3}\pi_{5}\pi_{4}\pi_{3})^{I_{5}} &=\pi_{2}\pi_{1}\pi_{3}\pi_{2}\pi_{4}\pi_{3}\widehat{\pi_{5}}\pi_{4}\pi_{3}=\pi_{2}\pi_{1}\pi_{3}\pi_{2}\pi_{4}\pi_{3}\pi_{4}\pi_{3}\\ &=\pi_{2}\pi_{1}\pi_{3}\pi_{2}\pi_{4}\pi_{4}\pi_3 \pi_4 =\pi_{2}\pi_{1}\pi_{3}\pi_{2}\pi_{4}\pi_3 \pi_4.
\end{align*}
Note that $(\pi_{2}\pi_{1}\pi_{3}\pi_{2}\pi_{4}\pi_{3}\pi_{5}\pi_{4}%
\pi_{3})^{I_{5}}= \pi_{\sigma(\Lambda(4,5),SE)}$ in this case.
Now we have
\begin{align}
\pi_{\sigma(\Lambda(4,5),SE)}\pi_{3}(3,2,2,1,0,0)  &  =\pi_{\sigma(\Lambda
(4,5),SE)}(3,2,1,2,0,0)= \pi_{2}\pi_{1}\pi_{3}\pi_{2}\pi_{4}\pi_{3}\pi
_{4}(3,2,1,2,0,0)\nonumber\\
&  =\pi_{2}\pi_{1}\pi_{3}\pi_{2}\pi_{4}\pi_{3}(3,2,1,0,2,0) =\pi_{2}\pi_{1}\pi
_{3}\pi_{2}\pi_{4}(3,2,0,1,2,0)\nonumber \\ &=\pi_{2}\pi_{1}\pi_{3}\pi_{2}(3,2,0,1,2,0)
  =\pi_{2}\pi_{1}\pi_{3}(3,0,2,1,2,0) =\pi_{2}\pi_{1}(3,0,1,2,2,0)\nonumber\\ &=\pi_{2}%
(0,3,1,2,2,0)
 =(0,1,3,2,2,0)=(0,\alpha,0).\nonumber
\end{align}
Note that $\alpha$ was also computed in part $(a)$ of Example~\ref{ex_olga_1}.

$(b)$ Let $\mu=(1,2,3,2)=\tau\lambda\in\mathbb{Z}^{4}_{\ge}0$, and $\lambda
=(3,2,2,1)$, $\widehat{\sigma_{0}\tau}\lambda=s_{1}\lambda=(2,3,2,1)=\sigma
_{0}^{[4]}\mu$. Then on the one hand we have
\[
\sigma_{0}\tau(3,2,2,1,0,0)=\sigma_{0}(1,2,3,2,0,0)=\sigma_{0}(\mu,0^{2})=(0,0,2,3,2,1).\]
On the other hand, mimicking the proof of Proposition~\ref{Prop_Olga}, we have $\pi_{\widehat{\sigma_{0}\tau}}=\pi_{1}$ and the product of the bubble sort operators $\prod_{i=1}^{p}(\pi_{i+n-p-1}\cdots\pi_{i})\pi_{\widehat{\sigma_{0}\tau}}$ is given in this case by $\pi
_{2}\pi_{1}\pi_{3}\pi_{2}\pi_{4}\pi_{3}\pi_{5}\pi_{4}\pi_{1}$ and we have
\[\pi
_{2}\pi_{1}\pi_{3}\pi_{2}\pi_{4}\pi_{3}\pi_{5}\pi_{4}\pi_{1}(3,2,2,1,0,0)=(0,0,2,3,2,1)=\sigma_0\tau(3,2,2,1,0,0).
\]
The decomposition $s_{2}s_{1}s_{3}s_{2}s_{4}s_{3}s_{5}s_{4}s_{1}$ is reduced and lies in $\mathfrak{S}_6^\lambda$. We calculate
\begin{align}
(\pi_{2}\pi_{1}\pi_{3}\pi_{2}\pi_{4}\pi_{3}\pi_{5}\pi_{4}\pi_{1})^{I_{5}}  &
= \pi_{2}\pi_{1}\pi_{3}\pi_{2}\pi_{4}\pi_{3}\pi_{4}\pi_{1}= \pi
_{\sigma(\Lambda(4,5),SE)}\pi_{1}.\nonumber
\end{align}
Now we have
\[
\pi_{\sigma(\Lambda(4,5),SE)}\pi_{1}(3,2,2,1,0,0)=\pi_{2}\pi_{1}\pi_{3}\pi_{2}%
\pi_{4}\pi_{3}\pi_{4}\pi_{1}(3,2,2,1,0,0)=(0,1,2,3,2,0)=(0,\alpha,0).
\]
Note that $\alpha$ was also computed in part $(b)$ of Example~\ref{ex_olga_1}.
\end{example}

\section{Last passage percolation in a Young diagram}\label{Secpercolation}

\subsection{LPP on rectangle Young diagrams}

We resume the notation of \S \ \ref{SubsecRSK}.\ Let $u_{1},\ldots,u_{m}$ and
$v_{1},\ldots,v_{m}$ be two sets of real numbers in the interval $[0,1[$ and
consider a family $w_{i,j}$ of independent random variables, with values in
$\mathbb{Z}_{\geq0}$, and such that%
\begin{equation}
\mathbb{P}(w_{i,j}=k)=(1-u_{i}v_{j})(u_{i}v_{j})^{k}\text{ for any }%
k\in\mathbb{Z}_{\geq0}. \label{geometricw(i,j)}%
\end{equation}
In other words, each $w_{i,j}$ follows a geometric distribution of parameter
$u_{i}v_{j}$. We then obtain a random matrix $\mathcal{W}$ with values in
$\mathcal{M}_{m\text{,}n}$ whose entry at position $(i,j)$ is defined as
$w_{i,j}$ for $1\leq i\leq m$ and $1\leq j\leq m$.\ Since the random variables
$w_{i,j}$ are independent, for any $A\in\mathcal{M}_{m\text{,}n}$ we get
\[
\mathbb{P}(\mathcal{W}=A)=\left(  \prod_{1\leq i\leq m,1\leq j\leq n}%
(1-u_{i}v_{j})\right)  (uv)^{A}%
\]
where $(uv)^{A}=\prod_{1\leq i\leq m,1\leq j\leq n}(u_{i}v_{j})^{a_{i,j}}.$

Now consider the paths in the matrices in $\mathcal{M}_{m\text{,}n}$ starting
at entry $(1,n)$ and ending at entry $(m,1)$ with possible steps
$\longleftarrow$ or $\downarrow$. The length of such a path is defined as the
sum of all the entries that it contains. Let us define de map $\mathrm{perc}$
which associates to each matrix $A$ in $\mathcal{M}_{m\text{,}n}$ the maximum
of the length path of all possible aforementioned paths in the matrix $A$. By
Assertion 4 of Theorem \ref{Th_RSK}, the integer $\mathrm{perc}(A)$ coincides
with the longest row of the tableaux $P(A)$ and $Q(A)$. This is the
\textsf{last passage percolation} associated to $A$. We then define the random
variable $G=\mathrm{perc}\circ\mathcal{W}$. Thanks to the above observation
and Theorem \ref{Th_RSK}, it becomes easy to give the law of the random
variable $G$. Set $\Delta_{m,n}=\prod_{1\leq i\leq m,1\leq j\leq n}%
(1-u_{i}v_{j})$. The following theorem was established in \cite{joh}.

\begin{theorem}
For any nonnegative integer $k$, we have
\[
\mathbb{P}(G=k)=\Delta_{m,n}\sum_{\lambda\in\mathcal{P}_{\min(m,n)}\mid
\lambda_{1}=k}s_{\lambda}(u)s_{\lambda}(v).
\]

\end{theorem}

In fact the results in \cite{joh} also give a law of large numbers of the
variable $G$ and also a Tracy-Widom renormalization theorem, both of which are
outside the scope of this note.

\subsection{LPP on staircases and non-symmetric Cauchy Kernel}

Thanks to Theorem \ref{Th_LKO}, the non-symmetric Cauchy kernel identity also
yields an interesting last percolation model. This time, we assume $m=n$ and
only consider independent random variables $w_{i,j}$ when $1\leq j\leq i\leq
n$ with geometric distributions as in (\ref{geometricw(i,j)}).\ This defines a
lower random square matrix $\mathcal{L}$ with nonnegative integer entries and
we get%
\[
\mathbb{P}(\mathcal{L}=A)=\prod_{1\leq j\leq i\leq n}(1-u_{i}v_{j})(pq)^{A}.
\]
One can interpret this model as follows. Consider paths from position $(1,n)$
to position $(n,1)$ where only the entries in the lower part of $A$ contribute
to the length of the paths.\ We can then define the random variable
$L=\mathrm{perc}\circ\mathcal{L}$ and try to determine its law. Since
Theorem~\ref{Th_LKO} gives a bijective correspondence obtained as the
restriction to lower triangular matrices of the RSK map defined on
$\mathcal{M}_{n,n}$, the value of $L$ still corresponds to the length of the
largest part of the partitions appearing in the right hand side of
(\ref{StaircaseKernel}). By Remark~\ref{Rq_rewriteLKO}, this yields the
following theorem.

\begin{theorem}
For any nonnegative integer $k$ we have%
\[
\mathbb{P}(L=k)=S_{n}\sum_{\mu\in\mathbb{Z}_{\geq0}^{n}\mid\max(\mu
)=k}\overline{\mathrm{\kappa}}^{\mu}(u)\mathrm{\kappa}_{\mu}(v)=S_{n}\sum_{\mu
\in\mathbb{Z}^{n}\mid\max(\mu)=k}\overline{\mathrm{\kappa}}_{\sigma_{0}(\mu)}%
(u_{n},\ldots,u_{1})\mathrm{\kappa}_{\mu}(v_{1},\ldots,v_{n}),
\]
where
\[
S_{n}=\prod_{1\leq j\leq i\leq n}(1-u_{i}v_{j}).
\]

\end{theorem}

\subsection{LPP and parabolic restrictions in non-symmetric Cauchy Kernel}

Given $p$ and $q$ as in \S \ \ref{Subsec_representation}, one can similarly
use Theorem \ref{Th_truncatedRectangle} to study the percolation model on
random matrices $\mathcal{T}_{p,q}$ with nonnegative random integer
coefficients having zero entries in each position $(i,j)$ such that $i\leq
n-p$ and $j>q$. Each random variable $w_{i,j}$ with $i\geq n-p+1$ and $j\leq
q$ follows a geometric distribution of parameter $u_{i}v_{j}$. Using the same
arguments as in \S \ \ref{SubsecPR}, we can obtain the law of the random
variable $T_{p,q}=\mathrm{perc}\circ\mathcal{T}_{p,q}$.

\begin{theorem}
For any nonnegative integer $k$, we have%
\[
\mathbb{P}(T_{p,q}=k)=T_{p,q}\sum_{(\mu_{1},\ldots,\mu_{p})\in\mathbb{Z}%
_{\geq0}^{p}\mid\max(\mu)=k}\overline{\mathrm{\kappa}}_{(\mu_{p},\ldots,\mu_{1}%
)}(u_{n},\ldots,u_{n-p+1})\mathrm{\kappa}_{\widetilde{\mu}}(v_{1},\ldots,v_{q}),
\]
where
\[
T_{p,q}=\prod_{(i,j)\in D_{\Lambda(p,q)}}(1-u_{i}v_{j}).
\]

\end{theorem}

\subsection{LPP and augmented staircases}

We now resume the notation of \S \ \ref{Subsec_ASC}. For a fixed partition
$\Lambda$ in $\mathcal{P}_{n}$, we consider random matrices $\mathcal{A}%
_{\Lambda}$ with nonnegative random integer coefficients having zero entries
in each position $(i,j)$ such that $(i,j)\notin\Lambda$. Here again each
random variable $w_{i,j}$ for $(i,j)\in\Lambda$ follows a geometric
distribution of parameter $u_{i}v_{j}$. Let us define the random variable
$A_{\Lambda}=\mathrm{perc}\circ\mathcal{A}_{\Lambda}$. Then, by Theorems
\ref{Th_ASC} and \ref{Th_RSK_ASC}, we get the law of $A_{\Lambda}$.

\begin{theorem}
For any nonnegative integer $k$, we have%
\begin{multline*}
\mathbb{P}(A_{\Lambda}=k)=\\
T_{\Lambda}\sum_{(\mu_{1},\ldots,\mu_{m})\in\mathbb{Z}^{m}\mid\max(\mu
)=k}D_{\sigma(\Lambda,NW)}\overline{\mathrm{\kappa}}_{(\mu_{m},\ldots,\mu_{1}%
)}(u_{n},\ldots,u_{n-m+1})D_{\sigma(\Lambda,SE)}\mathrm{\kappa}_{(\mu_{1}%
,\ldots,\mu_{m})}(v_{1},\ldots,v_{m}),
\end{multline*}
where
\[
T_{\Lambda}=\prod_{(i,j)\in D_{\Lambda}}(1-u_{i}v_{j}).
\]

\end{theorem}

\section{Appendix}\label{Appe}

Let $(W,S)$ be a Coxeter system. Let $\mathfrak{M}_{W}$ be the attached
\emph{Coxeter monoid}, that is, the monoid with generators a copy $\mathbf{S}$
of $S$, the same braid relations as $(W,S)$, and relations $\mathbf{s}%
^{2}=\mathbf{s}$ for all $s\in S$ replacing the relations $s^{2}=1$ for all
$s\in S$. Here by \textit{braid relations} we mean the defining relations $st
\cdots ts\cdots$, where $t\neq s$ and both sides are strictly alternating
products of $s$ and $t$ with $m_{s,t}=m_{t,s}$ factors, where $m_{s,t}$ is the
entry of the Coxeter matrix. These relations first appeared in work of
Demazure~\cite[Section 5.6]{Dem2}, and the Coxeter monoid was first
investigated by Richardson and Springer~\cite[Section 3.10]{RS}. It is
well-known (and a consequence of Matsumoto's Lemma) that there is a canonical
set-theoretic bijection between $W$ and $\mathfrak{M}_{W}$: it just sends any
reduced decomposition of an element of $W$ or $\mathfrak{M}_{W}$ to the same decomposition.

Let $I\subseteq S$, $w\in W$ and $s_{1}s_{2}\cdots s_{k}$ a reduced
decomposition of $w$. Consider the subword $s_{i_{1}}s_{i_{2}}\cdots
s_{i_{\ell}}$ of $s_{1}s_{2}\cdots s_{k}$ consisting of those letters in
$s_{1}s_{2}\cdots s_{k}$ lying in $I$. Set
\[
M((s_{1},s_{2},\dots,s_{k}),I):=\mathbf{s}_{i_{1}}\mathbf{s}_{i_{2}}%
\cdots\mathbf{s}_{i_{\ell}}\in\mathfrak{M}_{W}.
\]

\begin{lemma}\label{lem_indep_algo}
\label{lem1} The element $M((s_{1}, s_{2},\dots, s_{k}),I)$ is independent of
the choice $s_{1} s_{2}\cdots s_{k}$ of reduced decomposition for $w$, and we
simply denote it by $M(w,I)$.
\end{lemma}

\begin{proof}
By Matsumoto's Lemma, we know that any two reduced decompositions of $w$ are
related by applying a sequence of braid relations. It therefore suffices to
show that applying a braid relation to a reduced word for $w$ does not change
the element of $\mathfrak{M}_{W}$ obtained by keeping only those letters in
the words which lie in $I$.

Let $s_{1}s_{2}\cdots s_{k}$ be the first reduced decomposition of $w$, and
$s_{1}^{\prime}s_{2}^{\prime}\cdots s_{k}^{\prime}$ be the one obtained after
application of a single braid relation. A braid relation involves only two
letters $s,t\in S$ ($s\neq t$). Then $s_{1}s_{2}\cdots s_{k}$ (as a word) is
of the form $xsts\cdots y$ while $s_{1}^{\prime}s_{2}^{\prime}\cdots
s_{k}^{\prime}$ is of the form $xtst\cdots y$. If $s,t\notin I$, then it is
clear that the two subwords of $s_{1}s_{2}\cdots s_{k}$ and $s_{1}^{\prime
}s_{2}^{\prime}\cdots s_{k}^{\prime}$ consisting of those letters which are
not in $I$ coincide, hence that $M((s_{1},s_{2},\dots,s_{k}),I)=M((s_{1}%
^{\prime},s_{2}^{\prime},\dots,s_{k}^{\prime}),I)$. If $s,t\in I$, then the
two subwords differ by a single braid relation, which holds in $\mathfrak{M}%
_{W}$, hence define the same element of $\mathfrak{M}_{W}$. Finally, if only
one letter among $s$ and $t$, say $s$, is in $I$, then since $t\notin I$, it
follows that the substring $sts\cdots$ contributes $k$ consecutive copies of
$s$ to the subword of $s_{1}s_{2}\cdots s_{k}$ obtained by deleting the letter
not in $S$, while $tst\cdots$ contributes $k$ or $k-1$ copies of $s$ to the
subword of $s_{1}^{\prime}s_{2}^{\prime}\cdots s_{k}^{\prime}$, depending on
whether $m_{s,t}$ is odd or even. Moreover, since $s$ appears in both sides of
the braid relation $sts\cdots=tst\cdots$, then at least one copy of $s$ is
contributed in each word. Thanks to the relation $\mathbf{s}^{2}=\mathbf{s}$,
these consecutive copies of $s$ get reduced to $\mathbf{s}$ in $\mathfrak{M}%
_{W}$, again yielding $M((s_{1},s_{2},\dots,s_{k}),I)=M((s_{1}^{\prime}%
,s_{2}^{\prime},\dots,s_{k}^{\prime}),I)$.
\end{proof}

We denote by $\leq$ the strong Bruhat order on $W$ (or $\mathfrak{M}_{W}$). We
recall that, for $u,v\in W$, the following three conditions are equivalent
(see~\cite[Corollary 2.2.3]{BB})

\begin{enumerate}
\item $u\leq v$,

\item There is a reduced decomposition of $v$ having a reduced decomposition
of $u$ as a subword,

\item Every reduced decomposition of $v$ has a reduced decomposition of $u$ as
a subword.
\end{enumerate}

\begin{lemma}
\label{lem2}

\begin{enumerate}
\item Let $\mathbf{s_{1}s_{2}\cdots s_{k}}$ be a word in the generators of
$\mathfrak{M}_{W}$ and $1\leq i_{1}<i_{2}<\dots<i_{\ell}\leq k$ such that
$s_{i_{1}}s_{i_{2}}\cdots s_{i_{\ell}}$ is a reduced decomposition of an
element $w$ of $W$. Let $\mathbf{s_{1}^{\prime}s_{2}^{\prime}\cdots
s_{m}^{\prime}}$ be a word obtained from $\mathbf{s_{1}s_{2}\cdots s_{k}}$ by
applying a single defining relation of $\mathfrak{M}_{W}$. Then there is a
sequence $1\leq j_{1}<j_{2}<\dots<j_{\ell}\leq m$ such that $s_{j_{1}}%
^{\prime}s_{j_{2}}^{\prime}\cdots s_{j_{\ell}}^{\prime}$ is a reduced
decomposition of $w$.

\item Let $\mathbf{s_{1}s_{2}\cdots s_{k}}$ and $\mathbf{s_{1}^{\prime}%
s_{2}^{\prime}\cdots s_{m}^{\prime}}$ be two (not necessarily reduced) words
for the same element $\mathbf{w}$ of $\mathfrak{M}_{W}$. Let $\Omega_{1}$
(resp. $\Omega_{2}$) be the set of elements of $W$ having a reduced
decomposition which is a subword of $\mathbf{s_{1}s_{2}\cdots s_{k}}$ (resp.
$\mathbf{s_{1}^{\prime}s_{2}^{\prime}\cdots s_{m}^{\prime}}$). Then
$\Omega_{1}=\Omega_{2}$. In particular, this set $\Omega(\mathbf{w})$ depends
only on $\mathbf{w}$, and we have
\[
\Omega(\mathbf{w})=\{x\in W~|~x\leq w\}.
\]

\end{enumerate}
\end{lemma}

\begin{proof}
The second point is an immediate corollary of the first one; the last
statement is used by taking as word $\mathbf{s_{1}s_{2}\cdots s_{k}}$ any
reduced decomposition of $\mathbf{w}$.

Let us show the first point. The result is clear if the relation which is
applied to the word $\mathbf{s_{1} s_{2} \cdots s_{k}}$ is $\mathbf{s}%
\rightarrow\mathbf{s}^{2}$ or $\mathbf{s}^{2}\rightarrow\mathbf{s}$, since in
the case where we have to consecutive copies of $\mathbf{s}$ in the first or
the last word, then at most one can contribute to a reduced decomposition as
$ss$ is not reduced in $W$. Hence $s_{i_{1}} s_{i_{2}}\cdots s_{i_{\ell}}$
also appears as a reduced word of $\mathbf{s_{1}^{\prime}s_{2}^{\prime}\cdots
s_{m}^{\prime}}$ in this case. Hence assume that the relation which is applied
is a braid relation $\mathbf{w_{1}}=\mathbf{s} \mathbf{t} \cdots
\rightarrow\mathbf{t} \mathbf{s} \cdots=\mathbf{w_{2}}$. That is, we have
$k=m$ and (as words) $\mathbf{s_{1} s_{2} \cdots s_{k}}=\mathbf{s_{1} s_{2}
\cdots s_{i} w_{1} s_{j} s_{j+1}\cdots s_{k}}$ while $\mathbf{s_{1}^{\prime
}s_{2}^{\prime}\cdots s_{k}^{\prime}}=\mathbf{s_{1} s_{2} \cdots s_{i} w_{2}
s_{j} s_{j+1}\cdots s_{k}}$.

Denote by $p$ the number $\ell(\mathbf{w_{1}})$ of factors in either side of
the braid relation. The subword $\mathbf{u}$ of $\mathbf{w_{1}}$ which
contributes to the reduced word $s_{i_{1}} s_{i_{2}}\cdots s_{i_{k}}$ is
necessarily and alternating product of $s$ and $t$, otherwise it is not
reduced. Moreover, it contributes a subword $u$ of $s_{i_{1}} s_{i_{2}}\cdots
s_{i_{k}}$ which is made of consecutive letters, the letters before that
subword (resp. after that subword) coming from $\mathbf{s_{1} s_{2} \cdots
s_{i}}$ (resp. $\mathbf{s_{j} s_{j+1} \cdots s_{k}}$). If $\ell(u)<p$, then
$\mathbf{u}$ has a unique reduced decomposition, and $\mathbf{w_{2}}$ also has
$\mathbf{u}$ as a subword. Hence the claim holds true in this case. If
$\ell(u)=p$, then the whole left side $\mathbf{w_{1}}$ of the braid relation
is contributed as a consecutive subword of $s_{i_{1}} s_{i_{2}}\cdots
s_{i_{k}}$. Replacing that subword by the right side $ts\cdots$ of the braid
relation yields the required subword of $\mathbf{s_{1}^{\prime}s_{2}^{\prime
}\cdots s_{k}^{\prime}}$. It stays reduced as it is just obtained from a
reduced decomposition by applying a braid relation.
\end{proof}

\begin{proposition}
Let $w\in W$. The set $w_{I}^{\leq}:=\{x\in W~|~x\in W_{I}~\text{and}~x\leq
w\}$ is equal to $\{x\in W~|~x\leq M(w,I)\}$. In particular, it has a unique
maximal element for $\leq$, given by $M(w,I)$.
\end{proposition}

\begin{proof}
Let $s_{1} s_{2} \cdots s_{k}$ be a reduced decomposition of $w$.

Let $x\in w_{I}^{\leq}$. Since $x\leq w$, there is a subword $s_{j_{1}}
s_{j_{2}}\cdots s_{j_{m}}$, $1\leq j_{1} <j_{2} < \dots< j_{m} \leq k$ which
is a reduced decomposition of $x$. Since $x\in W_{I}$, all the letters of
$s_{j_{1}}s_{j_{2}}\cdots s_{j_{m}}$ lie in $I$. In particular, the reduced
decomposition $s_{j_{1}}s_{j_{2}}\cdots s_{j_{m}}$ is a subword of the subword
$s_{i_{1}}s_{i_{2}}\cdots s_{i_{\ell}}$ of $s_{1} s_{2}\cdots s_{k}$
consisting of those letters which lie in $I$. Putting Lemmas~\ref{lem1}
and~\ref{lem2}~(2) together we get that $x\leq M(w,I)$.

To conclude the proof, it therefore suffices to see that $M(w,I)\leq w$. By
Lemma~\ref{lem2}~(2), we know that any (not necessarily reduced) word for
$M(w,I)$ in $\mathfrak{M}_{W}$ has a subword which is a reduced word for
$M(w,I)$, as this property is independent of the chosen word, and it holds if
we take any reduced decomposition of $M(w,I)$ in $\mathfrak{M}_{W}$. But by
definition of $M(w,I)$, there is a subword if $s_{1}s_{2}\cdots s_{k}$ which
is a (not necessarily reduced) decomposition of $M(w,I)$ in $\mathfrak{M}_{W}%
$. Hence $s_{1}s_{2}\cdots s_{k}$ must have a reduced decomposition of
$M(w,I)$ appearing as a subword.
\end{proof}

\begin{example}
Let $W$ be of type $A_{3}$ and let $w=s_{1} s_{2} s_{3} s_{1} s_{2}$. The list
of reduced words for $w$ is given by

\begin{itemize}
\item $s_{1} s_{2} s_{3} s_{1} s_{2}, $

\item $s_{1} s_{2} s_{1} s_{3} s_{2},$

\item $s_{2} s_{1} s_{2} s_{3} s_{2},$

\item $s_{2} s_{1} s_{3} s_{2} s_{3},$

\item $s_{2} s_{3} s_{1} s_{2} s_{3}$.
\end{itemize}

Extracting the subword with letters in $I$ from every such decomposition yields

\begin{itemize}
\item $s_{1} s_{2} s_{1} s_{2}, $

\item $s_{1} s_{2} s_{1} s_{2},$

\item $s_{2} s_{1} s_{2} s_{2},$

\item $s_{2} s_{1} s_{2} ,$

\item $s_{2} s_{1} s_{2} $.
\end{itemize}

In $\mathfrak{M}_{W}$ we get

\begin{itemize}
\item $\mathbf{s_{1} s_{2} s_{1} s_{2}}=\mathbf{s_{1} s_{1} s_{2} s_{1}%
}=\mathbf{s_{1} s_{2} s_{1}}, $

\item $\mathbf{s_{1} s_{2} s_{1} s_{2}}=\mathbf{s_{1} s_{1} s_{2} s_{1}%
}=\mathbf{s_{1} s_{2} s_{1}}, $

\item $\mathbf{s_{2} s_{1} s_{2} s_{2}}=\mathbf{s_{2} s_{1} s_{2}%
}=\mathbf{s_{1} s_{2} s_{1}},$

\item $\mathbf{s_{2} s_{1} s_{2}}=\mathbf{s_{1} s_{2} s_{1}} ,$

\item $\mathbf{s_{2} s_{1} s_{2}}=\mathbf{s_{1} s_{2} s_{1}}$.
\end{itemize}

Note that the obtained is element is distinct from $w_{I}$, the element from
the canonical decomposition $w=w^{I} w_{I}$, which is given here by $s_{1}
s_{2}$.
\end{example}

\begin{remark}\label{rem_prod_cox_mon}
\label{rmqadd} It is a consequence of the definition of $M(w,I)$ and
Lemma~\ref{lem1} that if $u,v\in W$ with $\ell(uv)=\ell(u)+\ell(v)$, then
$M(uv,I)=M(u,I)M(v,I)$ (where the product is taken in $\mathfrak{M}_{W}$).
\end{remark}

\bigskip

\noindent Olga Azenhas: University of Coimbra, CMUC, Department of Mathematics.

\noindent{oazenhas@mat.uc.pt}

\bigskip

\noindent Thomas Gobet: Institut Denis Poisson Tours.

\noindent Universit\'{e} de Tours Parc de Grandmont, 37200 Tours, France.

\noindent{thomas.gobet@lmpt.univ-tours.fr}

\bigskip

\noindent C\'{e}dric Lecouvey: Institut Denis Poisson Tours.

\noindent Universit\'{e} de Tours Parc de Grandmont, 37200 Tours, France.

\noindent{cedric.lecouvey@lmpt.univ-tours.fr}

\end{document}